\documentclass[a4paper,12pt]{article}
\usepackage{graphicx} % Required for inserting images
\usepackage{amsthm}
\usepackage{amsmath,amsfonts,amssymb}
\usepackage{a4wide}
\usepackage[usenames,dvipsnames]{color}
\usepackage[all,cmtip]{xy}
\usepackage[verbose,colorlinks=true,linktocpage=true,linkcolor=blue,citecolor=blue]{hyperref}
\usepackage{footnote}
\usepackage{tikz}
\usepackage{tikz-cd}
\usepackage{extarrows}%2023.01.25 arrows
\usepackage[title]{appendix}%2023.09.11
\usepackage{hyperref}
\usepackage{marvosym}
\usepackage{indentfirst} % 导入indentfirst包，可以进行首行缩进
\usepackage{multirow} % 用于行合并
\usepackage{booktabs} % 优化表格线条（可选，使线条更美观）

\theoremstyle{definition}
\newtheorem{definition}{Definition}[section]

\theoremstyle{plain}
\newtheorem{theorem}[definition]{Theorem}
\newtheorem{proposition}[definition]{Proposition}
\newtheorem{lemma}[definition]{Lemma}
\newtheorem{corollary}[definition]{Corollary}
\newtheorem{example}[definition]{Example}
\theoremstyle{remark}
\newtheorem{remark}[definition]{Remark}

\numberwithin{equation}{section}

\newcommand\keywords[1]{\textbf{keywords}:#1}

\newcommand{\EmailD}[1]{\href{mailto:#1}{\Letter\ #1}}

\hypersetup{
    colorlinks=true, % 允许颜色设置生效
    linkcolor=blue, % 设置普通链接的颜色
    urlcolor=blue, % 设置URL链接颜色
}

\allowdisplaybreaks[1]

\title{Representations of Quantum Affine General Linear Superalgebras at Arbitrary 01-Sequences}
\author{Hongda Lin$^{1,2}$ and Honglian Zhang$^{1,3,}$\thanks{Corresponding Author.}}
\date{}

\begin{document}

\maketitle

\begin{center}
\footnotesize
\begin{itemize}
\item[1] Department of Mathematics, Shanghai University, Shanghai 200444, P.R. China.
\item[2] Shenzhen International Center for Mathematics, Southern University of Science and Technology, Guangdong 518055, P.R. China.

\EmailD{linhd@sustech.edu.cn}

\item[3] Newtouch Center for Mathematics at Shanghai University, Shanghai 200444, P.R. China.

\EmailD{hlzhangmath@shu.edu.cn}
\end{itemize}
\end{center}

\begin{abstract}
   
In this paper, we investigate finite-dimensional irreducible representations of the quantum affine general linear superalgebra $\mathrm{U}_q\big(\widehat{\mathfrak{gl}}_{m|n,\mathbf{s}}\big)$ for arbitrary 01-sequences $\mathbf{s}$, using the RTT presentation. We systematically construct the RTT presentation for  quantum general linear superalgebra $\mathrm{U}_q\big(\mathfrak{gl}_{m|n,\mathbf{s}}\big)$,  and derive a PBW basis induced by the action of the braid group, compatible with non-standard parities. We determine the necessary and sufficient conditions for the finite-dimensionality of irreducible representations of $\mathrm{U}_q\big(\mathfrak{gl}_{m|n,\mathbf{s}}\big)$ and extend  the results to the affine case via the evaluation homomorphism. Specific cases such as $(m,n)=(1,1)$ demonstrate that all finite-dimensional representations are tensor products of typical evaluation representations. This work extends existing representation frameworks and classification methods to encompass arbitrary 01-sequences, establishing the foundation for subsequent research on representations of quantum affine superalgebras.

\end{abstract}
\keywords  {\ Quantum affine superalgebras; RTT presentation; finite-dimensional irreducible representations; evaluation representations.}

\vspace{1em}
\section{Introduction}

Quantum groups represent a pivotal advancement in modern mathematics and theoretical physics. Among the most prominent examples of quantum groups are quantized enveloping algebras, which were initially introduced independently by Drinfeld \cite{Dr85} and Jimbo \cite{Ji85}. These algebras constitute a family of $q$-deformed topological Hopf algebras $\mathrm{U}_q\big(\mathfrak{a}\big)$, derived from the classical simple Lie algebra or Kac-Moody algebras $\mathfrak{a}$, and are commonly referred to as the Drinfeld-Jimbo presentation. As the classical limit $q \rightarrow 1$, the quantized enveloping algebras specialize to the corresponding universal enveloping algebras. This specialization preserves several key properties, including triangular decompositions, Hopf algebra structures, and character formulas for highest weight modules \cite{HK02}.

Another construction for the quantized enveloping algebra $\mathrm{U}_q\big(\mathfrak{a}\big)$ in terms of a matrix $R$ was defined depend on a finite-dimensional representation $V$ of $\mathrm{U}_q\big(\mathfrak{a}\big)$. The R-matrix $R$ is a solution of the quantum Yang-Baxter equation with value in $\operatorname{End} V^{\otimes 3}$:
\begin{gather*}
	R^{12}R^{13}R^{23} = R^{23}R^{13}R^{12},
\end{gather*}
where $R^{12} := R \otimes 1$, with analogous definitions for other indices. Reshetikhin, Faddeev, and Takhtajan(FRT) \cite{RTF90} demonstrated that this R-matrix construction offers an alternative presentation  of the quantized enveloping algebra. In the FRT formalism, the quantized enveloping algebra is realized through an upper triangular matrix and a lower triangular matrix that satisfies a set of ternary relations, known as the RTT presentation. This presentation is equipped with a natural comultiplication that enables the investigation of tensor product of representations. 

Quantum affine algebras, in addition to the Drinfeld-Jimbo and RTT presentations, admit a third presentation via Drinfeld currents \cite{Dr87}. The equivalence between Drinfeld and RTT presentations has been established for various types: Ding and Frenkel \cite{DF93} gave the proof for type A quantum affine algebras, while Jing, Liu, and Molev \cite{JLM20-1,JLM20-2} extended this result to types B, C, D, respectively.
While the Drinfeld realization lacks a finite-sum comultiplication, it remains essential in representation-theoretic investigations. Chari and Pressley \cite{CP94-2} have provided a classification of finite-dimensional irreducible representations of quantum affine algebras of type A, utilizing the evaluation homomorphism \cite{CP91,CP94-1}.

Quantum superalgebras are defined as $\mathbb{Z}_2$-graded generalizations of quantum groups, specifically designed to describe supersymmetric physical fields. Within the framework of quantum superalgebras, Bracken, Gould, and Zhang \cite{BGZ90,ZGB91} developed an R-matrix that serves as a solution to the supersymmetric quantum Yang-Baxter equation. While Yamane \cite{Ya94} introduced a similar quasi-triangular Hopf algebra structure for quantized enveloping superalgebras, employing a graded universal R-matrix derived from the quantum Drinfeld double construction. He also presented Serre-type presentations for affine Kac-Moody superalgebras and their quantizations, addressing both ABCD types and exceptional types \cite{Ya99}. Yamane showed that (affine) Kac-Moody superalgebras with different parities of generators are linked to various Dynkin diagrams and Serre-like relations, which can be transformed into one another via odd reflections. 

Typically, 01-sequences are used to describe the parities of generators of (affine) Lie superalgebras in type A, where 0 corresponds to even indices and 1 corresponds to odd indices. 
Most studies \cite{FHS97,JLZ25,Lu21,Ts21,Zhf14,Zhf16,Zhf17} for the structures and representations of the quantum affine general linear superalgebra $\mathrm{U}_q\big(\widehat{\mathfrak{gl}}_{m|n,\mathbf{s}}\big)$ rely on the standard 01-sequence,  equivalently, on the standard Borel subalgebra of $\mathfrak{gl}_{m|n,\mathbf{s}}$ (corresponding to the standard positive root system).  However, unlike semisimple Lie algebras, the presence of odd roots in classical Lie superalgebras implies that not every Borel subalgebra is conjugate to the standard one. Consequently, methods developed for standard cases are often inapplicable to non-standard 01-sequences.
Peng \cite{Pe16}, investigating Drinfeld-type parabolic presentations for super Yangians, identified challenges in constructing a partition of that ensures uniform parity within each block for arbitrary 01-sequence $\mathbf{s}$. 

Chang and Hu \cite{CH23} further presented an explicit formulation of quantum Berezian for parabolic diagonal generators at arbitrary partitions and parities. Additionally, Molev \cite{Mo22} developed an inductive rule to determine the finite-dimensionality conditions of irreducible highest weight representations for super Yangians associated with non-standard $\mathfrak{gl}_{m|n,\mathbf{s}}$ through a chain of odd reflections. Lu \cite{Lu22} revisited Molev’s results on odd reflections for the super Yangian $\mathrm{Y}\big(\mathfrak{gl}_{m|n,\mathbf{s}}\big)$ using Drinfeld current generators instead of RTT generators, linked them to XXX-type Bethe ansatz, and provided a $q$-character algorithm.

Consequently, it is natural to  explore finite-dimensional irreducible representations of $\mathrm{U}_q\big(\widehat{\mathfrak{gl}}_{m|n,\mathbf{s}}\big)$ for arbitrary 01-sequences $\mathbf{s}$. A significant challenge lies in the construction of odd reflections of quantum affine general linear superalgebras for studying their finite-dimensional irreducible representations. As this procedure is unintuitive and infeasible for Drinfeld current generators, we adopt the RTT presentation, motivated by \cite{GM10,MTZ04} for quantum affine algebras and \cite{Mo22,Zrb95,Zrb96} for super Yangians. Although constructing $q$-analogues of odd reflections for $\mathrm{U}_q\big(\widehat{\mathfrak{gl}}_{m|n,\mathbf{s}}\big)$ remains difficult, explicit realizations for the underlying quantum superalgebra $\mathrm{U}_q\big(\mathfrak{gl}_{m|n,\mathbf{s}}\big)$ are achieved through braid group actions.
Therefore, we initiate our analysis by examining finite-dimensional irreducible representations of $\mathrm{U}_q\big(\mathfrak{gl}_{m|n,\mathbf{s}}\big)$ for all $\mathbf{s}$. The primary objective of this work  is to determine the necessary and sufficient conditions for the irreducible representations of $\mathrm{U}_q\big(\widehat{\mathfrak{gl}}_{m|n,\mathbf{s}}\big)$ to be finite-dimensional. To this end, we employ the evaluation map $\mathrm{U}_q\big(\widehat{\mathfrak{gl}}_{m|n,\mathbf{s}}\big)\rightarrow \mathrm{U}_q\big(\mathfrak{gl}_{m|n,\mathbf{s}}\big)$ to induce families of finite-dimensional irreducible representations for $\mathrm{U}_q\big(\widehat{\mathfrak{gl}}_{m|n,\mathbf{s}}\big)$.

The paper is organized as follows. Section \ref{se:notations} establishes notation used throughout this work. In Section \ref{se:irreducible representationsQGLSAs}, we first introduce the RTT presentation $\mathrm{U}_q\big(\mathfrak{gl}_{m|n,\mathbf{s}}\big)$ of the quantum general linear superalgebra for arbitrary 01-sequence $\mathbf{s}$. We then define the braid group action on $\mathrm{U}_q\big(\mathfrak{gl}_{m|n,\mathbf{s}}\big)$, and utilize this action to construct a PBW basis for $\mathrm{U}_q\big(\mathfrak{gl}_{m|n,\mathbf{s}}\big)$. Additionally, we extend Zhang's results \cite{Zrb93} to general $\mathbf{s}$ by providing the transition rules for both typical and atypical irreducible representations. Differing from \cite{Mo22},  we further establish the explicit equivalent condition for the finite-dimensionality of these representations. 
In Section \ref{se:qafsuperalgebra}, we generalize \cite[Definition 3.1]{Zhf16} to define the quantum affine general linear superalgebra $\mathrm{U}_q\big(\widehat{\mathfrak{gl}}_{m|n,\mathbf{s}}\big)$ for all $\mathbf{s}$, and prove that it admits a PBW basis with the order the same as \cite[Corollary 2.13]{GM10}. 
Section \ref{se:highestweightreps} shows that every finite-dimensional irreducible representation of $\mathrm{U}_q\big(\widehat{\mathfrak{gl}}_{m|n,\mathbf{s}}\big)$, up to isomorphism, is the quotient of a Verma module over $\mathrm{U}_q\big(\widehat{\mathfrak{gl}}_{m|n,\mathbf{s}}\big)$. 
Additionally, we formulate the evaluation representations of $\mathrm{U}_q\big(\widehat{\mathfrak{gl}}_{m|n,\mathbf{s}}\big)$ by pulling back the finite-dimensional irreducible representations of $\mathrm{U}_q\big(\mathfrak{gl}_{m|n,\mathbf{s}}\big)$ via the evaluation homomorphism. 
In Section \ref{se:fdirreducible representations:af}, we classify the finite-dimensional irreducible representations for $\mathrm{U}_q\big(\widehat{\mathfrak{gl}}_{m|n,\mathbf{s}}\big)$. We
begin by examining several specific cases, namely $(m,n)=(1,1)$, $(2,0)$, $(0,2)$, and then proceed to verify the main result in the general case. For $(m,n)=(1,1)$, every finite-dimensional representation is actually a tensor product of typical evaluation representations.

\vspace{1em}
\section{Notations and sets up}\label{se:notations}
In this section, we need to introduce some primiliaries to standardize our notations. 
Let $\mathbb{C}$ be the set of complex numbers, $\mathbb{Z}$ the set of integers, and $\mathbb{Z}_+$ the set of non-negative integers, respectively. Write $\mathbb{Z}_2 =\mathbb{Z}/2\mathbb{Z}:= \{\bar{0}, \bar{1}\}$ as the two-element field. Throughout this paper, unless otherwise specified, all superspaces, associative superalgebras, and Lie superalgebras are considered to be over $\mathbb{C}$. Let $\delta_{\mathbf{con}}$ be the Kronecker function, which takes the value 1 if the condition `$\mathbf{con}$' is true and $0$ otherwise. We abbreviate $\delta_{i=j}$ to $\delta_{ij}$. 

For a superspace (resp. (Lie) superalgebra) $\mathcal{X}=\mathcal{X}(\bar{0})\bigoplus \mathcal{X}(\bar{1})$, the parity $|\cdot|$ of a homogeneous element $X\in \mathcal{X}$ is a $\mathbb{Z}_2$-value fuction denoted by 
\begin{equation*}
    |X|=\begin{cases}
        0, &\text{if}~~ X\in \mathcal{X}(\bar{0}), \\
        1, &\text{if}~~ X\in \mathcal{X}(\bar{1}).
    \end{cases}
\end{equation*}
We say $X$ is even if $|X|=\bar{0}$, and odd otherwise. If both $\mathcal{X}$ and $\mathcal{Y}$ are associative superalgebras, then the tensor product $\mathcal{X}\otimes\mathcal{Y}$ can be viewed as an associative superalgebra with the graded multiplication
$$(X_1\otimes Y_1)(X_2\otimes Y_2)=(-1)^{|Y_1||X_2|}X_1X_2\otimes Y_1Y_2,$$
for all homogeneous elements $X_1, X_2\in \mathcal{X}$, $Y_1,Y_2\in \mathcal{Y}$. 

Consider $m,n\in\mathbb{Z}_+$ with $N=m+n\geqslant 2$. 
We define $\mathcal{S}(m|n)$ as the set of all 01-sequences $\mathbf{s}=s_1s_2\cdots s_N$ that contain exactly $m$ 0s and $n$ 1s; any sequence $\mathbf{s}\in\mathcal{S}(m|n)$ is called a \textit{parity sequence}. A parity sequence $\mathbf{s}$ is said to be \textit{standard} if $s_i=0$ for $i=1,\ldots,m$ and $s_i=1$ for $i=m+1,\ldots,N$, and we denote this standard parity sequence by $\mathbf{s}^{\rm st}$. 

Introduce the following two functions on the index set $I^{m|n}_{\mathbf{s}}=\{1,\ldots,N\}$ (denoted briefly by $I_{\mathbf{s}}$) subject to a parity sequence $\mathbf{s}$: for $i\in I_{\mathbf{s}}$, 
\begin{equation*}
    |i|=\begin{cases}
        \bar{0}, &\text{if}~~s_i=0, \\
        \bar{1}, &\text{otherwise}.
    \end{cases}
    \qquad 
    d_i=\begin{cases}
        1, &\text{if}~~s_i=0, \\
        -1, &\text{otherwise}.
    \end{cases}
\end{equation*}
The following discussion summarizes the fundamentals of the general linear Lie superalgebra associated with a parity sequence $\mathbf{s}$, with reference to works \cite{Ka77-2,Mo22,Mu12} etc. 

Fix $\mathbf{s}\in\mathcal{S}(m|n)$, let $e_{1,\mathbf{s}},e_{2,\mathbf{s}},\ldots,e_{N,\mathbf{s}}$ be the standard basis of the superspace $\mathcal{V}_{\mathbf{s}}=\mathbb{C}^{m|n}$ with parities $|e_i|=|i|$ for all $i\in I_{\mathbf{s}}$. The endomorphism ring $\operatorname{End}\mathcal{V}_{\mathbf{s}}$ acts on $\mathcal{V}_{\mathbf{s}}$ via the rule
\begin{gather*}
    E_{ij,\mathbf{s}}(e_{k,\mathbf{s}})=\delta_{jk}e_{i,\mathbf{s}},\quad i,j,k\in I_{\mathbf{s}},
\end{gather*}
where $E_{ij,\mathbf{s}}$ with $|E_{ij,\mathbf{s}}|=|i|+|j|$ is the fundamental matrix whose $(i,j)$-entry is 1 and all other entries are 0. The $\operatorname{End}\mathcal{V}_{\mathbf{s}}$ admits a Lie superalgebra structure endowed with the super-bracket
\begin{gather*}
    [E_{ij,\mathbf{s}},\,E_{kl,\mathbf{s}}]=\delta_{jk}E_{il,\mathbf{s}}-(-1)^{(|i|+|j|)(|k|+|l|)}\delta_{il}E_{kj,\mathbf{s}}.
\end{gather*}
In this sense, we refer to $\operatorname{End}\mathcal{V}_{\mathbf{s}}$ as the \textit{general linear Lie superalgebra}, denoted by $\mathfrak{gl}(m|n)_{\mathbf{s}}$. To simplify the notation, we always write $\mathfrak{g}_{\mathfrak{s}}=\mathfrak{gl}(m|n)_{\mathbf{s}}$. 

Let $\mathfrak{h}_{\mathfrak{s}}$ be the span of all diagonal matrices $E_{ii,\mathbf{s}}$, denote $\mathfrak{h}_{\mathfrak{s}}$ as the \textit{Cartan subalgebra} of $\mathfrak{g}_{\mathfrak{s}}$. Consider the basis $\{\varepsilon_{1,\mathbf{s}}, \ldots, \varepsilon_{N,\mathbf{s}}\}$ of $\mathfrak{h}^{\ast}_{\mathfrak{s}}$ such that $\varepsilon_{i,\mathbf{s}}(E_{jj,\mathbf{s}}) = \delta_{ij}$ for all $i,j\in I_{\mathbf{s}}$, we introduce a non-degenerate symmetric bilinear form $(\,\cdot\,|\,\cdot\,)$ on $\mathfrak{h}^{\ast}_{\mathfrak{s}}$ defined by $(\varepsilon_{i,\mathbf{s}}|\varepsilon_{j,\mathbf{s}}) = d_i \delta_{ij}$. For $i \in I_{\mathbf{s}}\setminus\{N\}$, we define the simple roots by $\alpha_{i,\mathbf{s}} := \varepsilon_{i,\mathbf{s}}-\varepsilon_{i+1,\mathbf{s}}$, then set $\mathbf{P}_{\mathfrak{s}}:=\bigoplus_{i \in I_{\mathbf{s}}} \mathbb{Z} \varepsilon_{i,\mathbf{s}}$ the \textit{weight lattice} and $\mathbf{Q}_{\mathfrak{s}}:= \bigoplus_{i \in I_{\mathbf{s}}\setminus\{N\}} \mathbb{Z} \alpha_{i,\mathbf{s}}$ the \textit{root lattice}. The systems of even and odd positive roots are given by
\begin{align*}
\Phi_{\bar{0},\mathbf{s}}^+& := \{ \varepsilon_{i,\mathbf{s}} - \varepsilon_{j,\mathbf{s}} \mid 1 \leqslant i < j \leqslant N ~~\text{and}~~ |i| + |j|= \bar{0} \, \}, \\
\Phi_{\bar{1},\mathbf{s}}^+ & := \{ \varepsilon_{i,\mathbf{s}} - \varepsilon_{j,\mathbf{s}} \mid 1 \leqslant i < j \leqslant N~~ \text{and}~~ |i| + |j|= \bar{1}\, \},
\end{align*}
respectively.

Let $\mathcal{X}$ be an associative superalgebra, we use some conventional notation in the tensor product superalgebras 
$\mathcal{X}\otimes \mathrm{End}\mathcal{V}_{\mathbf{s}}^{\otimes K}$. For any $1\leqslant a\leqslant k$ and $X=\sum_{i\in I_{\mathbf{s}}}X_{ij}\otimes E_{ij,\mathbf{s}} \in\mathcal{X}$, we  denote by $X^a$ the element associated with the $a$-th copy of $\mathrm{End}\mathcal{V}_{\mathbf{s}}$ so that 
$$X^a=\sum\limits_{i,j\in I_{\mathbf{s}}}  X_{ij}\otimes 1^{\otimes (a-1)}\otimes E_{ij,\mathbf{s}}\otimes 1^{\otimes (K-a)} \in \mathcal{X}\otimes \mathrm{End}\mathcal{V}_{\mathbf{s}}^{\otimes K}.$$ 

In addition, for a matrix $R=\sum_{i} u_{(i)}\otimes v_{(i)}\in \operatorname{End}\mathcal{V}_{\mathbf{s}}^{\otimes 2}$ and an integer $K\geqslant 2$, denote by $1\leqslant a<b\leqslant K$,
\begin{gather*}
    R^{ab}=\sum 1^{\otimes (a-1)}\otimes u_{(i)}\otimes 1^{\otimes (b-1)}\otimes v_{(i)}\otimes 1^{\otimes (K-b)}\in \operatorname{End}\mathcal{V}_{\mathbf{s}}^{\otimes K}.
\end{gather*}
For example, if $K=3$, we have
\begin{gather*}
    R^{12}=R\otimes 1,\qquad R^{13}=\sum_i u_{(i)}\otimes 1\otimes v_{(i)},\qquad R^{23}=1\otimes R. 
\end{gather*}

To simplify notation, we adopt the convention of omitting the subscript $\mathbf{s}$ whenever $\mathbf{s} = \mathbf{s}^{\rm st}$, provided no confusion is likely to arise. For example, we write $\mathfrak{g}$ for $\mathfrak{g}_{\mathbf{s}^{\rm st}}$.

\vspace{1em}
\section{Irreducible representations of quantum general linear superalgebra}\label{se:irreducible representationsQGLSAs}

Let $q$ be a complex nonzero number that is not a root of unity, and let $d_i$ be given integers. We define $q_i = q^{d_i}$. This section reviews the definition and fundamental properties of the quantum general linear superalgebra. For any nonzero complex number $a$ and homogeneous elements $X$ and $Y$, we define the a-supercommutator as follows,
\begin{align*}
[X,\,Y]_a = XY - (-1)^{|X||Y|} a YX.
\end{align*}
We write $[X,\,Y] = [X,\,Y]_1$ for simplicity. 

\subsection{Two equivalent presentations of quantum general linear superalgebra}
\begin{definition}
Given $\mathbf{s} \in \mathcal{S}(m|n)$, the corresponding quantum general linear superalgebra $\mathcal{U}q(\mathfrak{g}{\mathbf{s}})$ (in its Drinfeld-Jimbo presentation) is an associative superalgebra. Its generators are $x_{i,\mathbf{s}}^{\pm}$ ($i \in I_{\mathbf{s}}\setminus{N}$) and $k_{a,\mathbf{s}}^{\pm 1}$ ($a \in I_{\mathbf{s}}$), whose parities are defined as $|x_{i,\mathbf{s}}^{\pm}| = |i| + |i+1|$ and $|k_{a,\mathbf{s}}^{\pm 1}| = \bar{0}$.
The defining relations are given as follows, 
\begin{align}\label{DJ:fin-1}
& k_{a,\mathbf{s}} k_{a,\mathbf{s}}^{-1} = k_{a,\mathbf{s}}^{-1} k_{a,\mathbf{s}}=1, \quad k_{a,\mathbf{s}} k_{b,\mathbf{s}} = k_{b,\mathbf{s}} k_{a,\mathbf{s}}, \\ \label{DJ:fin-2}
& k_{a,\mathbf{s}} x_{i,\mathbf{s}}^{\pm} k_{a,\mathbf{s}}^{-1} = q^{\pm(\varepsilon_{a,\mathbf{s}}| \varepsilon_{i,\mathbf{s}} - \varepsilon_{i+1,\mathbf{s}})} x_{i,\mathbf{s}}^{\pm}, \\ \label{DJ:fin-3}
& [x_{i,\mathbf{s}}^+, x_{i,\mathbf{s}}^-] = \delta_{ij} \frac{k_{i,\mathbf{s}} k_{i+1,\mathbf{s}}^{-1} - k_{i,\mathbf{s}}^{-1} k_{i+1,\mathbf{s}}}{q_i-q_i^{-1}}, \\ \label{DJ:fin-4}
& [x_{i,\mathbf{s}}^{\pm}, x_{j,\mathbf{s}}^{\pm}] = 0, \quad \text{if}~~ (\alpha_{i,\mathbf{s}}|\alpha_{j,\mathbf{s}}) = 0, \\ \label{DJ:fin-5}
& \big[ x_{i,\mathbf{s}}^{\pm}, [ x_{i,\mathbf{s}}^{\pm}, x_{\ell,\mathbf{s}}^{\pm} ]_{q_i}\big]_{q_i^{-1}} = 0, \quad \text{if}~~(\alpha_{i,\mathbf{s}}|\alpha_{i,\mathbf{s}}) \neq 0,~~\ell = i \pm 1, \\ \label{DJ:fin-6}
& \Big[\big[[ x_{i-1,\mathbf{s}}^{\pm}, x_{i,\mathbf{s}}^{\pm} ]_{q_i}, x_{i+1,\mathbf{s}}^{\pm} \big]_{q_{i+1}}, x_{i,\mathbf{s}}^{\pm}\Big] = 0, \quad \text{if}~~ (\alpha_{i,\mathbf{s}}|\alpha_{i,\mathbf{s}}) = 0.
\end{align}
\end{definition}

It is clear that the superalgebra $\mathcal{U}_q\big(\mathfrak{g}_{\mathbf{s}}\big)$ admits a Hopf superalgebra structure with the following comultiplication 
\begin{equation}\label{comul:fin:DJ}
    \begin{split}
        \triangle^{\mathbf{DJ}}(x_{i,\mathbf{s}}^+) &= 1 \otimes x_{i,\mathbf{s}}^+ + x_{i,\mathbf{s}}^+ \otimes k_{i,\mathbf{s}}^{-1} k_{i+1,\mathbf{s}}, \\
 \triangle^{\mathbf{DJ}}(x_{i,\mathbf{s}}^-) &= k_{i,\mathbf{s}} k_{i+1,\mathbf{s}}^{-1} \otimes x_{i,\mathbf{s}}^- + x_{i,\mathbf{s}}^- \otimes 1, \\
 \triangle^{\mathbf{DJ}}(k_{a,\mathbf{s}}^{\pm 1}) &= k_{a,\mathbf{s}}^{\pm 1} \otimes k_{a,\mathbf{s}}^{\pm 1}.
    \end{split}
\end{equation}

\begin{remark}
We can characterize the \textit{classical limit} of $\mathcal{U}_q\big(\mathfrak{g}_{\mathbf{s}}\big)$ analogously to how the standard case is treated in \cite{Zrb93}. When $q\rightarrow 1$, $\mathcal{U}_q\big(\mathfrak{g}_{\mathbf{s}}\big)$ coincides with the universal enveloping superalgebra $\mathcal{U}\big(\mathfrak{g}_{\mathbf{s}}\big)$ which is obtained by the following limiting processes: 
\begin{align*}
    \operatorname{lim}_{q\rightarrow 1}x_{i,\mathbf{s}}^+=E_{i,i+1,\mathbf{s}},\quad \operatorname{lim}_{q\rightarrow 1}x_{i,\mathbf{s}}^-=E_{i+1,i,\mathbf{s}},\quad 
    \operatorname{lim}_{q\rightarrow 1}\frac{k_{a,\mathbf{s}}-k_{a,\mathbf{s}}^{-1}}{q_a-q_a^{-1}}=E_{aa,,\mathbf{s}}. 
\end{align*}
\end{remark}

\vspace{1em}
Building upon the work of \cite{Zhf16}, the author established an equivalent R-matrix presentation for the quantum general linear superalgebra at the standard parity sequence, we now extend this framework to an arbitrary parity sequence $\mathbf{s} \in \mathcal{S}(m|n)$. The construction proceeds by considering the R-matrix defined by
\begin{gather*}
\mathcal{R}_{q,\mathbf{s}} = \sum_{i,j} q_i^{\delta_{ij}} E_{ii} \otimes E_{jj} + \sum_{i<j} (q_i - q_i^{-1}) E_{ji} \otimes E_{ij} \in \operatorname{End} \mathcal{V}_{\mathbf{s}}^{\otimes 2}.
\end{gather*}
The R-matrix $\mathcal{R}_{q,\mathbf{s}}$ is the $\mathbb{Z}_2$-graded solution of the following quantum Yang-Baxter equation
\begin{gather}\label{YB:solution}
	\mathcal{R}_{q,\mathbf{s}}^{12}\mathcal{R}_{q,\mathbf{s}}^{13}\mathcal{R}_{q,\mathbf{s}}^{23} = \mathcal{R}_{q,\mathbf{s}}^{23}\mathcal{R}_{q,\mathbf{s}}^{13}\mathcal{R}_{q,\mathbf{s}}^{12}.
\end{gather}

\begin{definition}
For a given $\mathbf{s} \in \mathcal{S}(m|n)$, the super R-matrix algebra associated to $\mathbf{s}$ is an associative superalgebra denoted by $\mathrm{U}q(\mathfrak{g}{\mathbf{s}})$. Its generators are $t_{ji,\mathbf{s}}$ and $\bar{t}{ij,\mathbf{s}}$ for $1 \leqslant i \leqslant j \leqslant N$, with parities given by $|t{ji,\mathbf{s}}| = |\bar{t}_{ij,\mathbf{s}}| = |i| + |j|$.
The defining relations are given as follows,
\begin{align}\label{RTT:fin-1}
&t_{ii,\mathbf{s}}\bar{t}_{ii,\mathbf{s}}=\bar{t}_{ii,\mathbf{s}}t_{ii,\mathbf{s}}=1,\quad \text{for }i\in I_{\mathbf{s}}, \\ \label{RTT:fin-2}
&\mathcal{R}_{q,\mathbf{s}}^{23}T_{\mathbf{s}}^1T_{\mathbf{s}}^2=T_{\mathbf{s}}^2T_{\mathbf{s}}^1\mathcal{R}_{q,\mathbf{s}}^{23}, \\ \label{RTT:fin-3}
&\mathcal{R}_{q,\mathbf{s}}^{23}\bar{T}_{\mathbf{s}}^1\bar{T}_{\mathbf{s}}^2=\bar{T}_{\mathbf{s}}^2\bar{T}_{\mathbf{s}}^1\mathcal{R}_{q,\mathbf{s}}^{23}, \\ \label{RTT:fin-4}
&\mathcal{R}_{q,\mathbf{s}}^{23}T_{\mathbf{s}}^1\bar{T}_{\mathbf{s}}^2=\bar{T}_{\mathbf{s}}^2T_{\mathbf{s}}^1\mathcal{R}_{q,\mathbf{s}}^{23},
\end{align}
where the matrices $T_{\mathbf{s}}$ and $\bar{T}_{\mathbf{s}}$ have the form
\begin{gather*}
T_{\mathbf{s}}=\sum_{1\leqslant i\leqslant j\leqslant N} E_{ji,\mathbf{s}}\otimes t_{ji,\mathbf{s}},\qquad \bar{T}_{\mathbf{s}}=\sum_{1\leqslant i\leqslant j\leqslant N} E_{ij,\mathbf{s}}\otimes \bar{t}_{ij,\mathbf{s}},
\end{gather*}
respectively.
\end{definition}

The superalgebra $\mathrm{U}_q\big(\mathfrak{g}_{\mathbf{s}}\big)$ possesses a Hopf superalgebra structure endowed with the comultiplication defined as
\begin{gather}\label{comul:fin:RTT}
	\triangle^{\mathbf{R}}(t_{ji,\mathbf{s}}) = \sum_{i \leqslant k \leqslant j} \varsigma_{ik;kj} t_{jk,\mathbf{s}} \otimes t_{ki,\mathbf{s}},\quad 
	\triangle^{\mathbf{R}}(\bar{t}_{ij,\mathbf{s}}) = \sum_{i \leqslant k \leqslant j} \varsigma_{ik;kj} \bar{t}_{ik,\mathbf{s}} \otimes \bar{t}_{kj,\mathbf{s}}
\end{gather}
where $\varsigma_{ab;cd}=(-1)^{(|a|+|b|)(|c|+|d|)}$ ($a,b,c,d\in I_{\mathbf{s}}$).

In terms of the generators $t_{ji,\mathbf{s}}$ and $t_{ij,\mathbf{s}}$, we are able to restate relations \eqref{RTT:fin-2}$-$\eqref{RTT:fin-4} in a more explicit form,
\begin{align}\label{RTT:fin-ex2}
    &q_i^{\delta_{ik}}t_{ij,\mathbf{s}}t_{kl,\mathbf{s}}-\varsigma_{ij;kl}q_j^{\delta_{jl}}t_{kl,\mathbf{s}}t_{ij,\mathbf{s}}=\varsigma_{ik;kl}(q_k-q_k^{-1})\left(\delta_{j<l}-\delta_{k<i}\right)t_{kj,\mathbf{s}}t_{il,\mathbf{s}}, \\ \label{RTT:fin-ex3}
    &q_i^{\delta_{ik}}\bar{t}_{ij,\mathbf{s}}\bar{t}_{kl,\mathbf{s}}-\varsigma_{ij;kl}q_j^{\delta_{jl}}\bar{t}_{kl,\mathbf{s}}\bar{t}_{ij,\mathbf{s}}=\varsigma_{ik;kl}(q_k-q_k^{-1})\left(\delta_{j<l}-\delta_{k<i}\right)\bar{t}_{kj,\mathbf{s}}\bar{t}_{il,\mathbf{s}}, \\ \label{RTT:fin-ex4}
    &q_i^{\delta_{ik}}t_{ij,\mathbf{s}}\bar{t}_{kl,\mathbf{s}}-\varsigma_{ij;kl}q_j^{\delta_{jl}}\bar{t}_{kl,\mathbf{s}}t_{ij,\mathbf{s}}=\varsigma_{ik;kl}(q_k-q_k^{-1})\left(\delta_{j<l}\bar{t}_{kj,\mathbf{s}}t_{il,\mathbf{s}}-\delta_{k<i}t_{kj,\mathbf{s}}\bar{t}_{il,\mathbf{s}}\right).
\end{align}
Consider a non-zero diagonal matrix $D = \operatorname{diag}(\mathfrak{d} \epsilon_1, \dots, \mathfrak{d} \epsilon_N)$, with $\mathfrak{d} \in \mathbb{C}\setminus{0}$ and $\epsilon_i \in {\pm 1}$. Then the map
\begin{gather}\label{UR:fin-isom}
T_{\mathbf{s}} \mapsto D T_{\mathbf{s}}, \quad \bar{T}{\mathbf{s}} \mapsto D^{-1} \bar{T}{\mathbf{s}}
\end{gather}
yields an automorphism of the superalgebra $\mathrm{U}q(\mathfrak{g}{\mathbf{s}})$, which is an immediate consequence of the defining relations \eqref{RTT:fin-1}, \eqref{RTT:fin-ex2}--\eqref{RTT:fin-ex4}.

\vspace{1em}
For our purpose, we introduce another $R$-matrix defined by
\begin{gather*}
\widetilde{\mathcal{R}}_{q,\mathbf{s}} = \mathcal{P}_{\mathbf{s}} \mathcal{R}_{q,\mathbf{s}}^{-1} \mathcal{P}_{\mathbf{s}},
\end{gather*}
where
\begin{gather*}
\mathcal{P}_{\mathbf{s}} =\sum_{ij \in I_{\mathbf{s}}} (-1)^{|j|} E_{ij,\mathbf{s}} \otimes E_{ji,\mathbf{s}}
\end{gather*}
is the $\mathbb{Z}_2$-graded permutation operator over $\mathcal{V}_{\mathbf{s}}^{\otimes 2}$ defined by $\mathcal{P}_{\mathbf{s}}(v \otimes w) = (-1)^{|v||w|} w \otimes v$ for homogeneous elements $v, w \in \mathcal{V}_{\mathbf{s}}$.
A direct calculation yields the identity:
\begin{gather}\label{Rq-P}
\widetilde{\mathcal{R}}_{q,\mathbf{s}} = \mathcal{R}_{q,\mathbf{s}} - (q - q^{-1}) \mathcal{P}_{\mathbf{s}}.
\end{gather}
Therefore, the relations \eqref{RTT:fin-2}$-$\eqref{RTT:fin-4} can be equivalently replaced by
\begin{gather}\label{RTT:fin-5}
\widetilde{\mathcal{R}}_{q,\mathbf{s}}^{23}T_{\mathbf{s}}^1T_{\mathbf{s}}^2=T_{\mathbf{s}}^2T_{\mathbf{s}}^1\widetilde{\mathcal{R}}_{q,\mathbf{s}}^{23},\quad \widetilde{\mathcal{R}}_{q,\mathbf{s}}^{23}\bar{T}_{\mathbf{s}}^1\bar{T}_{\mathbf{s}}^2=\bar{T}_{\mathbf{s}}^2\bar{T}_{\mathbf{s}}^1\widetilde{\mathcal{R}}_{q,\mathbf{s}}^{23},\quad \widetilde{\mathcal{R}}_{q,\mathbf{s}}^{23}\bar{T}_{\mathbf{s}}^1T_{\mathbf{s}}^2=T_{\mathbf{s}}^2\bar{T}_{\mathbf{s}}^1\widetilde{\mathcal{R}}_{q,\mathbf{s}}^{23}. 
\end{gather}
\begin{remark}
For a fixed $\mathbf{s} \in \mathcal{S}(m|n)$, the $R$-matrix $\widetilde{\mathcal{R}}{q,\mathbf{s}}$ takes the explicit form:
\begin{gather*}
\widetilde{\mathcal{R}}{q,\mathbf{s}}= \sum_{i,j} q_i^{-\delta_{ij}} E_{ii,\mathbf{s}} \otimes E_{jj,\mathbf{s}} - \sum_{i<j} (q_j - q_j^{-1}) E_{ij,\mathbf{s}} \otimes E_{ji,\mathbf{s}} \in \operatorname{End} \mathcal{V}_{\mathbf{s}}^{\otimes 2}.
\end{gather*}
In the purely even limit $n \to 0$, this matrix reduces to the standard trigonometric $R$-matrix for the quantum group $\mathrm{U}_q(\mathfrak{gl}_m)$; further details can be found in \cite{GM10,MR08,MRS03}.
\end{remark}

The following proposition is the generalization of \cite[Proposition 3.3(3)]{Zhf16} for arbitrary parity sequences, which will be proved  in Section \ref{se:braidPBWUqg}. 

\begin{proposition}\label{isom:fin}
The assignment
\begin{gather*}
\bar{t}{i,i+1,\mathbf{s}} \mapsto (q_i - q_i^{-1}) x{i,\mathbf{s}}^+ k_{i,\mathbf{s}}, \quad
t_{i+1,i,\mathbf{s}} \mapsto -(q_i - q_i^{-1}) k_{i,\mathbf{s}}^{-1} x_{i,\mathbf{s}}^-, \quad
\bar{t}{aa,\mathbf{s}} = t{aa,\mathbf{s}}^{-1} \mapsto k_{a,\mathbf{s}}
\end{gather*}
extends to a Hopf superalgebra isomorphism $\iota_{\mathbf{s}}: \mathrm{U}q(\mathfrak{g}_{\mathbf{s}}) \to \mathcal{U}q(\mathfrak{g}_{\mathbf{s}})$.
\end{proposition}

In view of this isomorphism, we call $\mathrm{U}q(\mathfrak{g}_{\mathbf{s}})$ the \textit{RTT presentation} of the quantum general linear superalgebra associated to the parity sequence $\mathbf{s}$.

\subsection{Braid group actions on $\mathrm{U}_q\big(\mathfrak{g}_{\mathbf{s}}\big)$}\label{se:braidPBWUqg}
The braid group serves as a principal tool for building root vectors and the PBW basis in the theory of standard quantum algebras. However, this approach does not apply to the superalgebra case due to its different root structure. To address this, we generalize the work of Molev and Ragoucy \cite[Section 2]{MR08}, thereby obtaining a systematic description of the braid group action on the root vectors of quantum general linear superalgebras.

Let $\mathfrak{S}_N$ be the symmetric group of degree $N$ and $\sigma_i$ the 2-cycle $(i, i+1)$. Recall that the \textit{brain group} of type $\mathfrak{gl}_N$, denoted by $\mathfrak{B}_{N}$, is generated by elements $\beta_{i}$ for $i=1,\ldots,N-1$ with relations
\begin{align*}%\label{braid1}
&\beta_{i} \beta_{i+1} \beta_{i} = \beta_{i+1} \beta_{i} \beta_{i+1},\quad i=1,\ldots,N-2, \\ %\label{braid2}
&\beta_i \beta_j = \beta_j \beta_i, \quad \text{if } j \neq i \pm 1.
\end{align*}
There is a surjective group homomorphism
$$\pi:~~\mathfrak{B}_N\rightarrow \mathfrak{S}_N\qquad \beta_i\mapsto \sigma_i,~~~~\text{for}~~i=1,\ldots,N-1.$$
The braid group $\mathfrak{B}_N$ acts on a parity sequence $\mathbf{s}=s_1 \cdots s_N \in \mathcal{S}(m|n)$ by
$$\beta.\mathbf{s} = \sigma(\mathbf{s})=s_{\sigma^{-1}(1)}\cdots s_{\sigma^{-1}(N)},\quad \text{for}~~\beta\in\mathfrak{B}_N,$$
where $\sigma=\pi(\beta)$.
If $\mathbf{s}$ contains a subsequence $s_is_{i+1} = 00$ or $11$, then $\mathbf{s}$ is invariant under the action of $\beta_i$; otherwise, $\beta_i$ is called an \textit{odd reflection}.

The elements of $\mathfrak{B}N$ can be interpreted as a family of isomorphisms between quantum general linear superalgebras in the following way. Fix $\mathbf{s} \in \mathcal{S}(m|n)$ and $i \in I{\mathbf{s}} \setminus {N}$.
Denote $\mathbf{s}' =s_1'\cdots s_N':= \sigma_i(\mathbf{s})$ and $d_i' = (-1)^{s_i'}$.
Following \cite{Ya94, Ya99}, we have
\begin{proposition}\label{odd-rel:1}
There exists an isomorphism  $\beta_{i,\mathbf{s}}:\mathcal{U}_q\big(\mathfrak{g}_{\mathbf{s}}\big)\rightarrow \mathcal{U}_q\big(\mathfrak{g}_{\mathbf{s}'}\big)$ given by
\begin{equation*}
\begin{aligned}
&k_{i,\mathbf{s}}\mapsto d_i'k_{i+1,\mathbf{s}'},\hspace{3em} k_{i+1,\mathbf{s}}\mapsto d_{i+1}'k_{i,\mathbf{s}'}, \\
&x_{i,\mathbf{s}}^+\mapsto -d_{i+1}'x_{i,\mathbf{s}'}^-k_{i,\mathbf{s}'}k_{i+1,\mathbf{s}'}^{-1}, \\
&x_{i-1,\mathbf{s}}^+\mapsto d_i'\,q^{-d_i'}\big[x_{i-1,\mathbf{s}'}^+,\,x_{i,\mathbf{s}'}^+\big]_{q^{d_i'}}, \\
&x_{i+1,\mathbf{s}}^+\mapsto - d_{i+1}'\,\big[x_{i,\mathbf{s}'}^+,\,x_{i+1,\mathbf{s}'}^+\big]_{q^{-d_{i+1}'}}, \\
&x_{r,\mathbf{s}}^+\mapsto x_{r,\mathbf{s}'}^+,\hspace{4em} 
x_{r,\mathbf{s}}^-\mapsto x_{r,\mathbf{s}'}^-,
\end{aligned}
\hspace{3em}
\begin{aligned}
&k_{a,\mathbf{s}}\mapsto k_{a,\mathbf{s}'},\hspace{3em} a\neq i,i+1, \\
&x_{i,\mathbf{s}}^-\mapsto -d_i'k_{i+1,\mathbf{s}'}k_{i,\mathbf{s}'}^{-1}x_{i,\mathbf{s}'}^+, \\
&x_{i-1,\mathbf{s}}^-\mapsto q^{d_i'}\big[x_{i,\mathbf{s}'}^-,\,x_{i-1,\mathbf{s}'}^-\big]_{q^{-d_i'}}, \\
&x_{i+1,\mathbf{s}}^-\mapsto -\big[x_{i+1,\mathbf{s}'}^-,\,x_{i,\mathbf{s}'}^-\big]_{q^{d_{i+1}'}}, \\
&r\neq i,i\pm 1.
\end{aligned}
\end{equation*}
\end{proposition}
Here, the subscript $\mathbf{s}$ of $\beta_{i,\mathbf{s}}$ indicates that the action of the braid generator $\beta_i$ ($i=1,\ldots,N-1$) as an isomorphism depends on the choice of $\mathbf{s}\in\mathcal{S}(m|n)$.

Next, we will show that
$\beta_{i,\mathbf{s}}$ can also act as an isomorphism of super R-matrix algebras. More specifically, we have
\begin{proposition}\label{odd-rel:2}
There exists an isomorphism $\beta_{i,\mathbf{s}}:\mathrm{U}_q\big(\mathfrak{g}_{\mathbf{s}}\big)\rightarrow \mathrm{U}_q\big(\mathfrak{g}_{\mathbf{s}'}\big)$ given by
\begin{align*}
&t_{ii,\mathbf{s}}\mapsto d_i't_{i+1,i+1,\mathbf{s}'},\quad t_{i+1,i+1,\mathbf{s}}\mapsto d_{i+1}'t_{ii,\mathbf{s}'},\quad t_{i+1,i,\mathbf{s}}\mapsto d_i'd_{i+1}'q^{-d_i'}\bar{t}_{i,i+1,\mathbf{s}'}\bar{t}_{ii,\mathbf{s}'}^{-2}, \\
&t_{ik,\mathbf{s}}\mapsto \varsigma_{i-1,i;i,i+1}'q^{-d_i'}t_{i+1,k,\mathbf{s}'}-\varsigma_{k,i-1;i,i+1}'t_{ii,\mathbf{s}'}^{-1}t_{i+1,i,\mathbf{s}'}t_{ik,\mathbf{s}'},\quad \text{if}~~k\leqslant i-1, \\
&t_{i+1,k,\mathbf{s}}\mapsto -\varsigma_{i-1,i;i,i+1}'d_{i+1}'t_{ik,\mathbf{s}'},\quad \text{if}~~ k\leqslant i-1, \\
&t_{li,\mathbf{s}}\mapsto \varsigma_{i,i+1;i,i+2}'q^{d_i'}t_{l,i+1,\mathbf{s}'}-\varsigma_{i,i+1;i+2,l}'t_{ii,\mathbf{s}'}t_{li,\mathbf{s}'}\bar{t}_{i,i+1,\mathbf{s}'},\quad \text{if}~~ l\geqslant i+2, \\
&t_{l,i+1,\mathbf{s}}\mapsto -\varsigma_{i,i+1;i+1,i+2}'d_{i+1}'t_{li,\mathbf{s}'},\quad \text{if}~~ l\geqslant i+2, \\
%&t_{lk}\mapsto (-1)^{[\alpha_i']([\alpha_{i-1}']+[\alpha_{i+1}'])}d_{i+1}'t_{lk},\quad \text{if } k\leqslant i-1,\ l\geqslant i+2, \\
&t_{lk,\mathbf{s}}\mapsto t_{lk,\mathbf{s}'},\quad \text{in all remaining cases},
\end{align*}
and
\begin{align*}
&\bar{t}_{ii,\mathbf{s}}\mapsto d_i'\bar{t}_{i+1,i+1,\mathbf{s}'},\quad \bar{t}_{i+1,i+1,\mathbf{s}}\mapsto d_{i+1}'\bar{t}_{ii,\mathbf{s}'},\quad \bar{t}_{i,i+1,\mathbf{s}}\mapsto q^{d_i'}t_{ii,\mathbf{s}'}^{-2}t_{i+1,i,\mathbf{s}'}, \\
&\bar{t}_{ki,\mathbf{s}}\mapsto \varsigma_{i-1,i;i,i+1}'d_i'q^{d_i'}\bar{t}_{k,i+1,\mathbf{s}'}-\varsigma_{k,i-1;i,i+1}'d_i'\bar{t}_{ki,\mathbf{s}'}\bar{t}_{i,i+1,\mathbf{s}'}\bar{t}_{ii,\mathbf{s}'}^{-1},
  \quad\text{if}~~ k\leqslant i-1, \\
&\bar{t}_{k,i+1,\mathbf{s}}\mapsto -\varsigma_{i-1,i;i,i+1}'\bar{t}_{ki,\mathbf{s}'},\quad \text{if}~~ k\leqslant i-1, \\
&\bar{t}_{il,\mathbf{s}}\mapsto \varsigma_{i,i+1;i,i+2}'d_{i}'q^{-d_{i}'}\bar{t}_{i+1,l,\mathbf{s}'}-
\varsigma_{i,i+1;i+2,l}'d_{i}'t_{i+1,i,\mathbf{s}'}\bar{t}_{il,\mathbf{s}'}\bar{t}_{ii,\mathbf{s}'},\quad \text{if}~~ l\geqslant i+2, \\
&\bar{t}_{i+1,l,\mathbf{s}}\mapsto -\varsigma_{i,i+1;i+1,i+2}'\bar{t}_{il,\mathbf{s}'},\quad \text{if}~~ l\geqslant i+2, \\
%&\bar{t}_{kl}\mapsto (-1)^{[\alpha_i']([\alpha_{i-1}']+[\alpha_{i+1}'])}d_{i+1}'\bar{t}_{kl},\quad \text{if } k\leqslant i-1,\ l\geqslant i+2, \\
&\bar{t}_{kl,\mathbf{s}}\mapsto \bar{t}_{kl,\mathbf{s}'},\quad \text{in all remaining cases},
\end{align*}
wher $\varsigma_{ab;cd}'=(-1)^{(|a|+|b|)(|c|+|d|)}\,(a,b,c,d\in I_{\mathbf{s}'})$.
\end{proposition}

\begin{proof}
The following mapping defines an inverse for $\beta_{i,\mathbf{s}}$:
\begin{align*}
&t_{ii,\mathbf{s}'}\mapsto d_it_{i+1,i+1,\mathbf{s}},\quad t_{i+1,i+1,\mathbf{s}'}\mapsto d_{i+1}t_{ii,\mathbf{s}},\quad t_{i+1,i,\mathbf{s}'}\mapsto q^{-d_{i+1}}\bar{t}_{i+1,i+1,\mathbf{s}}^{-2}\bar{t}_{i,i+1,\mathbf{s}}, \\
&t_{ik,\mathbf{s}'}\mapsto -\varsigma_{i-1,i+1;i,i+1}d_it_{i+1,k,\mathbf{s}},,\quad \text{if}~~k\leqslant i-1, \\
&t_{i+1,k,\mathbf{s}'}\mapsto \varsigma_{i-1,i+1;i,i+1}q^{d_{i+1}}t_{i,k,\mathbf{s}}-\varsigma_{k,i-1;i,i+1}t_{i+1,i+1,\mathbf{s}}\bar{t}_{i,i+1,\mathbf{s}}t_{i+1,k,\mathbf{s}},\quad \text{if}~~k\leqslant i-1, \\
&t_{li,\mathbf{s}'}\mapsto -\varsigma_{i,i+1;i,i+2}d_i t_{l,i+1,\mathbf{s}},\quad \text{if}~~ l\geqslant i+2, \\ 
&t_{l,i+1,\mathbf{s}'}\mapsto \varsigma_{i,i+1;i+1,i+2}q^{-d_{i+1}}t_{li,\mathbf{s}}-\varsigma_{i,i+1;i+2,l}t_{i+1,i+1,\mathbf{s}}^{-1}t_{l,i+1,\mathbf{s}}t_{i+1,i,\mathbf{s}},\quad \text{if}~~ l\geqslant i+2, \\
&t_{lk,\mathbf{s}'}\mapsto t_{lk,\mathbf{s}},\quad \text{in all remaining cases},\end{align*}
and
\begin{align*}
&\bar{t}_{ii,\mathbf{s}'}\mapsto d_i\bar{t}_{i+1,i+1,\mathbf{s}},\quad \bar{t}_{i+1,i+1,\mathbf{s}'}\mapsto d_{i+1}\bar{t}_{ii,\mathbf{s}},\quad \bar{t}_{i,i+1,\mathbf{s}'}\mapsto d_id_{i+1}q^{d_{i+1}}t_{i+1,i,\mathbf{s}}t_{i+1,i+1,\mathbf{s}}^{-2}, \\
&\bar{t}_{ki,\mathbf{s}'}\mapsto -\varsigma_{i-1,i+1;i,i+1}\bar{t}_{k,i+1,\mathbf{s}},
  \quad\text{if}~~ k\leqslant i-1, \\
&\bar{t}_{k,i+1,\mathbf{s}'}\mapsto \varsigma_{i-1,i+1;i,i+1}d_{i+1}q^{-d_{i+1}}\bar{t}_{ki,\mathbf{s}}-\varsigma_{k,i-1;i,i+1}d_{i+1}\bar{t}_{k,i+1,\mathbf{s}}t_{i+1,i,\mathbf{s}}\bar{t}_{i+1,i+1,\mathbf{s}},
  \quad\text{if}~~ k\leqslant i-1, \\
&\bar{t}_{il,\mathbf{s}'}\mapsto -\varsigma_{i,i+1;i,i+2}\bar{t}_{i+1,l,\mathbf{s}},\quad \text{if}~~ l\geqslant i+2, \\
&\bar{t}_{i+1,l,\mathbf{s}'}\mapsto \varsigma_{i,i+1;i+1,i+2}d_{i+1}q^{-d_{i+1}}\bar{t}_{il,\mathbf{s}}-
\varsigma_{i,i+1;i+2,l}d_{i+1}\bar{t}_{i,i+1,\mathbf{s}}\bar{t}_{i+1,l,\mathbf{s}}\bar{t}_{i+1,i+1,\mathbf{s}}^{-1},\quad \text{if}~~ l\geqslant i+2, \\
&\bar{t}_{kl,\mathbf{s}'}\mapsto \bar{t}_{kl,\mathbf{s}},\quad \text{in all remaining cases},
\end{align*}
Moreover, direct computation shows that $\beta_{i,\mathbf{s}}$ and $\beta_{i,\mathbf{s}}^{-1}$ are mutually inverse superalgebra homomorphisms, which implies that $\beta_{i,\mathbf{s}}$ is an isomorphism.
\end{proof}

Using the braid group action on quantum general linear superalgebras both in their Drinfeld-Jimbo presentations and RTT presentations, we obtain the following commutative diagram
\begin{center}
   \begin{tikzcd}
       \mathrm{U}_q\big(\mathfrak{g}_{\mathbf{s}}\big)
		\arrow{r}{\iota_{\mathbf{s}} }\arrow{d}{\beta_{\mathbf{s}}} & \mathcal{U}_q\big(\mathfrak{g}_{\mathbf{s}}\big) \arrow{d}{\beta_{\mathbf{s}}} \\
		\mathrm{U}_q\big(\mathfrak{g}_{\mathbf{t}}\big) \arrow{r}{\iota_{\mathbf{t}} } & \mathcal{U}_q\big(\mathfrak{g}_{\mathbf{t}}\big),
   \end{tikzcd}
\end{center}
where the braid element $\beta$ acts on the parity sequence $\mathbf{s} \in \mathcal{S}(m|n)$ via $\sigma(\mathbf{s}) = \mathbf{t}$ for some permutation $\sigma \in \mathfrak{S}N$. The fact that $\iota{\mathbf{s}}$ is an isomorphism follows from the case of the standard parity sequence $\mathbf{s}^{\mathrm{st}}$ established in \cite[Proposition 3.3(3)]{Zhf16}. Furthermore, $\iota_{\mathbf{s}}$ preserves the Hopf superalgebra structure, as seen by comparing \eqref{comul:fin:DJ} and \eqref{comul:fin:RTT}. This concludes the proof of Theorem~\ref{isom:fin}.

\subsection{PBW basis of $\mathrm{U}_q\big(\mathfrak{g}_{\mathbf{s}}\big)$}\label{se:PBWtheoremUqgl}
Under the framework of Theorem~\ref{isom:fin}, we construct a PBW-type basis for $\mathrm{U}q(\mathfrak{g}{\mathbf{s}})$ for an arbitrary parity sequence $\mathbf{s} \in \mathcal{S}(m|n)$. We first recall the well-known PBW basis in terms of the Drinfeld–Jimbo generators for the quantum general linear superalgebra at the standard parity sequence $\mathbf{s}^{\mathrm{st}}$. Here, as introduced in Section~\ref{se:notations}, we omit the subscript $\mathbf{s}$ when $\mathbf{s} = \mathbf{s}^{\mathrm{st}}$.

Introduce the elements $e_{ij}$ for $i\neq j$, $i,j\in I$ in $\mathcal{U}_q\big(\mathfrak{g}\big)$:
\begin{equation*}
    \begin{aligned}
        e_{i,i+1}&:=x_{i}^+, \\
        e_{ij}&:=-q_k^{-1}\big[e_{kj},\,e_{ik}\big]_{q_k}, 
    \end{aligned}\qquad
    \begin{aligned}
        e_{i+1,i}&:=x_{i}^-, \\
        e_{ji}&:=-q_{i+1}\big[e_{ki},\,e_{jk}\big]_{q_k^{-1}},
    \end{aligned}\qquad
    \begin{aligned}
        &\\
        &i<k<j, 
    \end{aligned}
\end{equation*}
where the expressions of $e_{ij}$ and $e_{ji}$ are independent of the choice of $k$. 
Then we have (\cite{HZ20,Zrb93})
\begin{theorem}\label{PBW:fin-st:DJ}
The set of all ordered monomials
\begin{gather}\label{order:1}
\mathop{\overrightarrow{\prod}}\limits_{i\in I} e_{i,i-1}^{b_{i,i-1}} e_{i,i-2}^{b_{i,i-2}} \cdots e_{i,1}^{b_{i,1}} \times \mathop{\overrightarrow{\prod}}\limits_{i\in I} k_i^{b_{ii}} \times
\mathop{\overrightarrow{\prod}}\limits_{i\in I} e_{1,i}^{b_{1,i}} e_{2,i}^{b_{2,i}} \cdots e_{i-1,i}^{b_{i-1,i}}
\end{gather}
with the exponents
\begin{equation}\label{exponents}
b_{ij}\in\begin{cases}
    \mathbb{Z}_+, &\text{if}~~|i| + |j| = \bar{0}~~\text{and}~~i\neq j, \\
    \{0,1\}, &\text{if}~~|i| + |j| = \bar{1}, \\
    \mathbb{Z}, &\text{if}~~i=j
\end{cases}
\end{equation}
form a basis for $\mathcal{U}_q\big(\mathfrak{g}\big)$.
\end{theorem}

Now, we are ready to state the following PBW theorem of $\mathrm{U}_q\big(\mathfrak{g}_{\mathbf{s}}\big)$ at arbitrary $\mathbf{s}\in\mathcal{S}(m|n)$.
\begin{theorem}\label{PBW:fin:RTT}
For any fixed $\mathbf{s}\in\mathcal{S}(m|n)$, the set of all ordered monomials
\begin{gather}\label{order:2}
\mathop{\overrightarrow{\prod}}\limits_{i\in I_{\mathbf{s}}} t_{i,i-1,\mathbf{s}}^{b_{i,i-1}} t_{i,i-2,\mathbf{s}}^{b_{i,i-2}} \cdots t_{i,1,\mathbf{s}}^{b_{i,1}} \times \mathop{\overrightarrow{\prod}}\limits_{i\in I_{\mathbf{s}}} \bar{t}_{ii,\mathbf{s}}^{\,b_{ii}} \times \mathop{\overrightarrow{\prod}}\limits_{i\in I_{\mathbf{s}}} \bar{t}_{1,i,\mathbf{s}}^{\,b_{1,i}} \bar{t}_{2,i,\mathbf{s}}^{\,b_{2,i}} \cdots \bar{t}_{i-1,i,\mathbf{s}}^{\,b_{i-1,i}}
\end{gather}
with the exponents \eqref{exponents} forms a basis for $\mathrm{U}_q\big(\mathfrak{g}_{\mathbf{s}}\big)$.
\end{theorem}
\begin{proof}
    The relations \eqref{RTT:fin-ex2} and \eqref{RTT:fin-ex4} for $i=j$ give
    \begin{gather*}
        \bar{t}_{ii,\mathbf{s}}\gamma_{kl}=q_i^{\delta_{ik}-\delta_{il}}\gamma_{kl}\bar{t}_{ii,\mathbf{s}}
    \end{gather*}
    for $\gamma\in\{\,t_{\mathbf{s}},\bar{t}_{\mathbf{s}}\,\}$. It indicates that $\bar{t}_{ii,\mathbf{s}}$ for all $i\in I_{\mathbf{s}}$ commutes with each $\gamma_{kl,\mathbf{s}}$ up to a constant. Therefore, any monomial $X\in \mathrm{U}_q\big(\mathfrak{g}_{\mathbf{s}}\big)$ has the form 
    $$X=\prod_{i\in I_{\mathbf{s}}}\bar{t}_{ii,\mathbf{s}}^{\,c_{ii}}\times \gamma_{i_1j_1}\cdots \gamma_{i_lj_l}$$
    with $\gamma\in\{\,t_{\mathbf{s}},\bar{t}_{\mathbf{s}}\,\}$, $c_{ii}\in\mathbb{Z}_+$ and $i_a\neq j_a$ for each $a$. It is admissible to introduce a filtration on the generators of $\mathrm{U}_q\big(\mathfrak{g}_{\mathbf{s}}\big)$ by setting
    $\deg X=l. $
    Define the associated graded algebra $\operatorname{Gr}\mathrm{U}_q\big(\mathfrak{g}_{\mathbf{s}}\big)$ by means of the following construction: 
    \begin{align*}
        \mathrm{U}_q^{[p]}\big(\mathfrak{g}_{\mathbf{s}}\big)&:=\big\{X\in \mathrm{U}_q\big(\mathfrak{g}_{\mathbf{s}}\big)\big| \deg X\geqslant p \big\},\qquad p\geqslant 0, \\
        \operatorname{Gr}\mathrm{U}_q\big(\mathfrak{g}_{\mathbf{s}}\big)&:=\bigoplus_{p=0}^{\infty} \mathrm{U}_q^{[p]}\big(\mathfrak{g}_{\mathbf{s}}\big)/ \mathrm{U}_q^{[p+1]}\big(\mathfrak{g}_{\mathbf{s}}\big). 
    \end{align*}
    Observe that the component $\mathrm{U}_q^{[0]}\big(\mathfrak{g}_{\mathbf{s}}\big)/ \mathrm{U}_q^{[1]}\big(\mathfrak{g}_{\mathbf{s}}\big)$ is commutative and generated by the images of all $t_{ii,\mathbf{s}},\bar{t}_{ii,\mathbf{s}}$. This theorem can be checked immediately for the case of $\mathbf{s}=\mathbf{s}^{\rm st}$ by Theorem \ref{isom:fin} and \ref{PBW:fin-st:DJ}. Then the image of the ordered monomials with form
    \begin{gather}\label{order:3}
\mathop{\overrightarrow{\prod}}\limits_{i\in I_{\mathbf{s}}} t_{i,i-1,\mathbf{s}}^{b_{i,i-1}} t_{i,i-2,\mathbf{s}}^{b_{i,i-2}} \cdots t_{i,1,\mathbf{s}}^{b_{i,1}} \times \mathop{\overrightarrow{\prod}}\limits_{i\in I_{\mathbf{s}}} \bar{t}_{1,i,\mathbf{s}}^{\,b_{1,i}} \bar{t}_{2,i,\mathbf{s}}^{\,b_{2,i}} \cdots \bar{t}_{i-1,i,\mathbf{s}}^{\,b_{i-1,i}}
\end{gather}
    for $\mathbf{s}=\mathbf{s}^{\rm st}$, $\sum_{i\neq j} b_{ij}=p$ and \eqref{exponents} constitude a basis for $\mathrm{U}_q^{[p]}\big(\mathfrak{g}\big)/ \mathrm{U}_q^{[p+1]}\big(\mathfrak{g}\big)$ for $p>0$. 

    For any $\mathbf{s}\in\mathcal{S}(m|n)$, there is a $\sigma\in \mathfrak{S}_N$ such that $\mathbf{s}=\sigma_{i_r}\cdots\sigma_{i_1}\mathbf{s}^{\rm st}$. The action of $$\beta_{i_r,\sigma_{i_{r-1}}\cdots\sigma_{i_1}\mathbf{s}}\cdots\beta_{i_2,\sigma_{i_1}\mathbf{s}}\beta_{i_1,\mathbf{s}}$$ forces that, for each $\mathbf{s}\in\mathcal{S}(m|n)$, the image of ordered monomials with form \eqref{order:3} for $\sum_{i\neq j} b_{ij}=p$ and \eqref{exponents} constitude a basis for $\mathrm{U}_q^{[p]}\big(\mathfrak{g}_{\mathbf{s}}\big)/ \mathrm{U}_q^{[p+1]}\big(\mathfrak{g}_{\mathbf{s}}\big)$ . Consequently, the ordered monomials \eqref{order:2} with exponents \eqref{exponents} constitude a basis for $\mathrm{U}_q\big(\mathfrak{g}_{\mathbf{s}}\big)$. 
    
\end{proof}

\begin{remark}
The PBW basis in Theorem~\ref{PBW:fin:RTT} is formulated in terms of the RTT generators. Its construction exhibits a subtle but notable difference from the PBW basis built from the Drinfeld--Jimbo generators as given by Yamane~\cite[Section 5]{Ya94}.
\end{remark}

%\begin{corollary}
%Let $\mathcal{N}_{\mathbf{s}}^+$ (resp. $\mathcal{N}_{\mathbf{s}}^-$ and $\mathrm{U}_{\mathbf{s}}^0$) be the span of ordered monomials
%\begin{gather*}
%\mathop{\overrightarrow{\prod}}\limits_{i\in I_{\mathbf{s}}} t_{i,i-1,\mathbf{s}}^{b_{i,i-1}} t_{i,i-2,\mathbf{s}}^{b_{i,i-2}} \cdots t_{i,1,\mathbf{s}}^{b_{i,1}} \ \bigg(\text{resp. } \mathop{\overrightarrow{\prod}}\limits_{i\in I_{\mathbf{s}}} \bar{t}_{1,i,\mathbf{s}}^{\,b_{1,i}} \bar{t}_{2,i,\mathbf{s}}^{\,b_{2,i}} \cdots \bar{t}_{i-1,i,\mathbf{s}}^{\,b_{i-1,i}} \ \text{and} \ \mathop{\overrightarrow{\prod}}\limits_{i\in I_{\mathbf{s}}} \bar{t}_{ii,\mathbf{s}}^{\,b_{ii}} \bigg).
%\end{gather*}
%Define $\mathrm{U}_{\mathbf{s}}^+ = \mathrm{U}_{\mathbf{s}}^0 \mathcal{N}_{\mathbf{s}}^+$ and $\mathrm{U}_{\mathbf{s}}^- = \mathcal{N}_{\mathbf{s}}^- \mathrm{U}_{\mathbf{s}}^0$. Then, $\mathrm{U}_q\big(\mathfrak{g}_{\mathbf{s}}\big)$ has the following decomposition
%\begin{gather}\label{tricom:R}
%\mathrm{U}_q(\mathfrak{gl}_{m|n}^{\mathfrak{s}}) \xlongrightarrow{\simeq} \mathrm{U}^- \mathrm{U}^0 \mathrm{U}^+
%\end{gather}
%and each tensor component is a sub-superalgebra of $\mathrm{U}_q(\mathfrak{gl}_{m|n}^{\mathfrak{s}})$.
%\end{corollary}

\subsection{Finite-dimensional irreducible representations of $\mathrm{U}_q\big(\mathfrak{g}_{\mathbf{s}}\big)$}\label{se:fdirreducible representations-fin}

This subsection is devoted to classify the finite-dimensional irreducible representations of the quantum general linear superalgebra $\mathrm{U}_q\big(\mathfrak{g}_{\mathbf{s}}\big)$ for arbitrary $\mathbf{s}\in\mathcal{S}(m|n)$. 
Our approach follows Molev~\cite{Mo22}, relying on the technique of odd reflections. Crucially, the isomorphism established in Theorem~\ref{isom:fin} allows us to define the highest-weight representation on the RTT generators of $\mathrm{U}q(\mathfrak{g}{\mathbf{s}})$.
For any fixed parity sequence $\mathbf{s}$, let 
$\mathfrak{g}_{\mathbf{s}}= \mathfrak{g}_{\mathbf{s}}(\bar{0}) \oplus \mathfrak{g}_{\mathbf{s}}(\bar{1})$ denote its $\mathbb{Z}_2$-graded decomposition. 

\subsubsection{Kac module over $\mathrm{U}_q\big(\mathfrak{g}_{\mathbf{s}}\big)$}\label{se:Kacmodule}
\begin{definition}
A representation $V$ is called a \textit{highest weight representation} over  $\mathrm{U}_q\big(\mathfrak{g}_{\mathbf{s}}\big)$ if $V$ is generated by a non-zero vector $\zeta \in V$ such that
\begin{align*}
&\bar{t}_{ij,\mathbf{s}}\zeta = 0, \quad \forall\ 1\leqslant i < j\leqslant N, \\
&\bar{t}_{ii,\mathbf{s}}\zeta = \lambda_i \zeta, \quad \lambda_i \in \mathbb{C} \backslash \{0\}.
\end{align*}
Set $\Lambda = (\lambda_1, \ldots, \lambda_N)$.
The vector $\zeta$ and the $N$-tuple $\Lambda$ are referred to as the \textit{maximal vector} and the \textit{highest weight} of $V$, respectively.
\end{definition}

Recall that every finite-dimensional irreducible representation of $\mathrm{U}q(\mathfrak{g}{\mathbf{s}}(\bar{0}))$ is a highest-weight module. This fact enables us to construct a class of finite-dimensional representations over $\mathrm{U}_q\big(\mathfrak{g}_{\mathbf{s}}\big)$, which are termed \textit{Kac module}.

Let $\mathring{V}_{\mathbf{s}}(\Lambda)$ be the finite-dimensional irreducible representation of $\mathrm{U}_q\big(\mathfrak{g}_{\mathbf{s}}(\bar{0})\big)\subset \mathrm{U}_q\big(\mathfrak{g}_{\mathbf{s}}\big)$ with the highest weight $\Lambda$. It induces a representation $K_{\mathbf{s}}(\Lambda)$ over $\mathrm{U}_q\big(\mathfrak{g}_{\mathbf{s}}\big)$ by setting 
\begin{gather*}
    \bar{t}_{ij,\mathbf{s}}. \mathring{V}_{\mathbf{s}}(\Lambda)=0. 
\end{gather*}
The representation $K_{\mathbf{s}}(\Lambda)$ is finite-dimensional with the highest weight $\Lambda$, but not necessarily irreducible. As stated in \cite{Zrb93}, $K_{\mathbf{s}}(\Lambda)$ has a unique irreducible quotient $\overline{K}_{\mathbf{s}}(\Lambda)$. By definition, for any given $N$-tuple $\Lambda\in\left(\mathbb{C}\setminus\{0\}\right)^N$, there exists a unique irreducible representation $\overline{K}_{\mathbf{s}}(\Lambda)$ over $\mathrm{U}_q\big(\mathfrak{g}_{\mathbf{s}}\big)$ with highest weight $\Lambda$. 
Let $V_{\mathbf{s}}(\Lambda)$ be a highest weight irreducible representation $\mathrm{U}_q\big(\mathfrak{g}_{\mathbf{s}}\big)$ with highest weight $\Lambda$. We need to show the necessary and sufficient conditions for the finite-dimensionality of $V_{\mathbf{s}}(\Lambda)$, that is to say, $V_{\mathbf{s}}(\Lambda)\simeq \overline{K}_{\mathbf{s}}(\Lambda)$. 

%\begin{proposition}
%Every finite-dimensional irreducible $\mathrm{U}_q(\mathfrak{gl}_{m|n}^{\mathfrak{s}})$-module is a highest weight module.
%\end{proposition}

%\begin{remark}
%    Per convention, we sometimes shorten ``irreducible representation" to ``irreducible %representation". 
%\end{remark}

\subsubsection{Zhang's results for $\mathrm{U}_q\big(\mathfrak{g}_{\mathbf{s}}\big)$ at $\mathbf{s}=\mathbf{s}^{\rm st}$}
As in Section~\ref{se:notations}, we take $\mathbf{s} = \mathbf{s}^{\mathrm{st}}$ and omit the subscript $\mathbf{s}$. We now recall from \cite{Zrb93} the necessary and sufficient conditions for an irreducible representation $V(\Lambda)$ of $\mathrm{U}_q(\mathfrak{g})$ to be finite-dimensional.

\begin{theorem}\label{standard:rep:fin}
Consider the $N$-tuple
$\Lambda = (\lambda_1, \ldots, \lambda_L)$ $(\forall\,\lambda_i \in \mathbb{C}\setminus\{0\})$. 
The following conditions for the irreducible highest weight representation $V(\Lambda)$ of $\mathrm{U}_q\big(\mathfrak{g}\big)$ with highest weight $\Lambda$ are equivalent\,{\rm :}
\begin{enumerate}
    \item[{\rm (1)}] $\dim V(\Lambda)<\infty$\,{\rm ;} 
    \item[{\rm (2)}] there exist some $\ell_1, \ldots, \ell_{m-1}, \ell_{m+1}, \ldots, \ell_{N-1} \in\mathbb{Z}_+$ and $\ell\in\mathbb{C}\setminus\{0\}$ such that
    \begin{gather}\label{fd:cd:fin}
    \frac{\epsilon_i \lambda_i}{\epsilon_{i+1} \lambda_{i+1}} = q_i^{\ell_i}, \qquad \frac{\epsilon_m \lambda_m}{\epsilon_{m+1} \lambda_{m+1}} = q^{\ell},
    \end{gather}
    for some $N$-tuple $\epsilon = (\epsilon_1, \ldots, \epsilon_N)\ (\forall\,\epsilon_i \in \{\pm 1\})$.
\end{enumerate}
\end{theorem}

Following  \cite{Zrb93}, a $\mathrm{U}_q\big(\mathfrak{gl}_{m|n}\big)$-module $V(\Lambda)$ is called \textit{typical} if $V(\Lambda) \simeq K(\Lambda)$;  otherwise, it is referred to as  \textit{atypical}.
Under the finite-dimensionality conditions \eqref{fd:cd:fin}, the highest weight of a finite-dimensional irreducible representation $V(\Lambda)$, up to the isomorphism \eqref{UR:fin-isom}, uniquely corresponds to a function in $\mathfrak{h}^{\ast}$ (still denoted by $\Lambda$ for convenience) given by the equations
\begin{gather*}
(\Lambda|\varepsilon_i)=d_i\Lambda_i\quad  \text{with}\quad \Lambda_i-\Lambda_{i+1} = \ell_i \in \mathbb{Z}_+,~~i\neq m.
\end{gather*}

\begin{proposition}\label{Zrb:p1}\cite[Proposition 2]{Zrb93}
The irreducible representation $V(\Lambda)$ over $\mathrm{U}_q\big(\mathfrak{g}\big)$ is typical if and only if
$$\prod_{\alpha \in \Phi_{\bar{1}}^+} (\Lambda + \rho, \alpha) \neq 0,$$ 
where $\rho$ is the graded half-sum of all positive roots for $\mathfrak{g}$, given by
\begin{gather*}
2\rho = \sum_{i=1}^{m} (m-n-2i+1) \varepsilon_i + \sum_{a=m+1}^{m+n} (3m+n-2a+1) \varepsilon_a.
\end{gather*}
Moreover,
\begin{gather*}
\dim V(\Lambda) = 2^{mn} \dim \mathring{V}(\Lambda),
\end{gather*}
where $\mathring{V}(\Lambda)$ is the finite-dimensional irreducible representation of $\mathrm{U}_q\big(\mathfrak{g}(\bar{0})\big)$ as mentioned in Section \ref{se:Kacmodule}, and its dimension is given by the formula
\begin{gather*}
\dim \mathring{V}(\Lambda) = \prod_{\alpha \in \Phi_{\bar{0}}^+} \frac{(\Lambda + \rho, \alpha)}{(\rho, \alpha)}.
\end{gather*}
\end{proposition}
In what follows, we address the case of a non-standard parity sequence $\mathbf{s} \in \mathcal{S}(m|n)$. Our treatment is based on the theory of odd reflections for $\mathrm{U}_q(\mathfrak{g}_{\mathbf{s}})$, developed in Section~\ref{se:braidPBWUqg} for arbitrary 01-sequence $\mathbf{s}$.

\subsubsection{Transition rules via odd reflections}
We begin with the special case $(m,n)=(1,1)$, for which $\mathbf{s} \in {01, 10}$. 
 Starting from the known finite-dimensionality conditions for $\mathrm{U}_q\big(\mathfrak{gl}_{1|1,10}\big)$
 we derive the corresponding conditions for the irreducible representation $V_{10}(\Lambda)$ of $\mathrm{U}_q\big(\mathfrak{gl}_{1|1,10}\big)$ by applying the odd reflection $\beta_{1,01}$.

Let $\zeta$ be the maximal vector of $V(\Lambda)$ with highest weight $\Lambda = (\lambda_1, \lambda_2)$. By applying the automorphism \eqref{UR:fin-isom}, we define $\lambda_1 = q^{\Lambda_1}$ and $\lambda_2 = q^{-\Lambda_2}$. According to Proposition~\ref{Zrb:p1}, the module $V(\Lambda)$ is typical if and only if $\Lambda_1 \neq -\Lambda_2$, and is atypical otherwise. In the typical case, the vector $\omega = t_{21} \zeta$ is nonzero in $V(\Lambda)$.
Moreover,
\begin{gather*}
\bar{t}_{11} \omega = q^{\Lambda_1 - 1} \omega, \quad \bar{t}_{22} \omega = q^{-\Lambda_2 - 1} \omega, \quad \text{and} \quad t_{21} \omega = 0.
\end{gather*}
Notice that $\bar{t}_{12} \omega \neq 0$ owing to $\Lambda_1 + \Lambda_2 \neq 0$. By the action of the odd reflection $\beta_{1,01}$, we regard $V' = V(\Lambda)$ as a representation over $\mathrm{U}_q\big(\mathfrak{gl}_{1|1,10}\big)$ with $\Lambda' = (q^{-\Lambda_2 - 1}, q^{\Lambda_1 - 1})$. Then $V'$ is isomorphic to the finite-dimensional irreducible representation $\overline{K}_{10}(\Lambda')$. In addition, $V(\Lambda)$ and $V_{10}(\Lambda'')$ with $\Lambda'' = (q^{-\Lambda_2}, q^{\Lambda_1})$ are both one-dimensional atypical modules if $\Lambda_1 + \Lambda_2 = 0$, which also forces 
$V(\Lambda)\simeq V_{10}(\Lambda'')$ 
as representation over $\mathrm{U}_q\big(\mathfrak{gl}_{1|1,10}\big)$.

\vspace{1em}

Now we turn to the general case. Let $\mathbf{s}=s_1s_2\cdots s_N\in\mathcal{S}(m|n)$, and let $\zeta$ be the maximal vector of $V_{\mathbf{s}}(\Lambda)$ with highest weight $\Lambda = (\lambda_1, \lambda_2,\ldots,\lambda_N)$. As in the standard case $\mathbf{s}=\mathbf{s}^{\rm st}$, we regard $\Lambda$ as a function in $\mathfrak{h}_{\mathbf{s}}^{\ast}$ such that $\Lambda_i=d_i^{-1}(\Lambda|\varepsilon_{i,\mathbf{s}})$ for $i\in I_{\mathbf{s}}$. 

Should $\mathbf{s}$ comprise only $0$s or only $1$s, the finite-dimensionality conditions for $V_{\mathbf{s}}(\Lambda)$ coincide with those of the non-super quantum algebra. Otherwise, at any position $i$ where $s_i \ne s_{i+1}$, one of the following embeddings
\begin{equation}\label{embeddings:fin}
\mathrm{U}_q\big(\mathfrak{gl}_{1|1,01}\big) \rightarrow \mathrm{U}_q\big(\mathfrak{g}_{\mathbf{s}}\big), \quad
\mathrm{U}_q\big(\mathfrak{gl}_{1|1,10}\big) \rightarrow \mathrm{U}_q\big(\mathfrak{g}_{\mathbf{s}}\big)
\end{equation}
is realized, mapping the generators $t_{lk}$, $\bar{t}{kl}$ to $t{i-1+l,, i-1+k}$, $\bar{t}_{i-1+k,, i-1+l}$, respectively, for $1 \leqslant k,l \leqslant 2$.
Without loss of generality, we set $\lambda_i=q_i^{\Lambda_i}$ for $i\in I$,
then we have
\begin{proposition}\label{non-st:rep:fin}
If $\mathbf{s}$ has a subsequence $s_is_{i+1}=01$ or $10$ and $\Lambda_i+\Lambda_{i+1}\neq 0$, then the representation $V_{\mathbf{s}}(\Lambda)$ of $\mathrm{U}_q\big(\mathfrak{g}_{\mathbf{s}}\big)$ is isomorphic to the representation $V_{\sigma_i\mathbf{s}}\big(\Lambda^{[i]}\big)$ of $\mathrm{U}_q\big(\mathfrak{g}_{\sigma_i\mathbf{s}}\big)$, where
\begin{gather*}
\Lambda^{[i]}=(q_1^{\Lambda_1},\ldots,q_{i-1}^{\Lambda_{i-1}},q_{i+1}^{\Lambda_{i+1}+1},q_{i}^{\Lambda_{i}-1},
   q_{i+2}^{\Lambda_{i+2}},\ldots,q_N^{\Lambda_N}).
\end{gather*}
\end{proposition}

\begin{proof}
Set $\omega^{[i]}=t_{i+1,i,\mathbf{s}}\zeta$ and $W=V_{\mathbf{s}}(\Lambda)$. Due to the embeddings \eqref{embeddings:fin}, we obtain
\begin{gather*}
\bar{t}_{ii,\mathbf{s}}\omega^{[i]}=q_i^{\Lambda_i-1}\omega^{[i]},\quad \bar{t}_{i+1,i+1,\mathbf{s}}\omega^{[i]}=q_{i+1}^{\Lambda_{i+1}+1}\omega^{[i]}, \quad t_{i+1,i,\mathbf{s}}\omega^{[i]}=0\quad \text{and}\quad
\bar{t}_{i,i+1,\mathbf{s}}\omega^{[i]}\neq 0.
\end{gather*}
Furthermore,
\begin{align*}
\bar{t}_{kl,\mathbf{s}}\omega^{[i]} &= \bar{t}_{kl,\mathbf{s}}t_{i+1,i,\mathbf{s}}\zeta 
=q_{i+1}^{\delta_{k,i+1}}q_i^{-\delta_{il}}\varsigma_{i+1,i;k,l}\,t_{i+1,i,\mathbf{s}}\bar{t}_{kl,\mathbf{s}}\zeta \\
&\qquad+q_i^{-\delta_{il}}\varsigma_{i,k;k,l}(q_k-q_k^{-1})
    \left(\, \delta_{i+1<k} t_{ki,\mathbf{s}}\bar{t}_{i+1,l,\mathbf{s}}-\delta_{i>l}\bar{t}_{ki,\mathbf{s}}t_{i+1,l,\mathbf{s}}  \,\right)\zeta=0,
\end{align*}
for $k<l$ with $(k,l)\neq (i,i+1)$.
We regard $W$ as a representation over $\mathrm{U}_q\big(\mathfrak{g}_{\sigma_i\mathbf{s}}\big)$ by the action of $\beta_{i,\mathfrak{s}}$. Then by Proposition \ref{odd-rel:2}, $W$ is the highest weight representation associated with highest weight $\Lambda^{[i]}$ and maximal vector $\omega^{[i]}$, we complete the proof.

\end{proof}

Moreover, an analogous argument shows that

\begin{proposition}\label{non-st:rep:fin2}
If $\mathbf{s}$ has a subsequence $s_is_{i+1}=01$ or $10$ and $\Lambda_i+\Lambda_{i+1}= 0$, then the representation $V_{\mathbf{s}}(\Lambda)$ of $\mathrm{U}_q\big(\mathfrak{g}_{\mathbf{s}}\big)$ is isomorphic to the representation $V_{\sigma_i\mathbf{s}}\big(\Lambda^{[i]}\big)$ of $\mathrm{U}_q\big(\mathfrak{g}_{\sigma_i\mathbf{s}}\big)$, where
\begin{gather*}
\Lambda^{[i]} = (q_1^{\Lambda_1}, \ldots, q_{i-1}^{\Lambda_{i-1}}, q_{i+1}^{\Lambda_{i+1}}, q_{i}^{\Lambda_i},
   q_{i+2}^{\Lambda_{i+2}}, \ldots, q_N^{\Lambda_N}). 
\end{gather*}
In this case, $V_{\mathbf{s}}(\Lambda)$ and $V_{\sigma_i\mathbf{s}}\big(\Lambda^{[i]}\big)$ share the same maximal vector $\zeta$.
\end{proposition}

Combining with the transition rules from Propositions \ref{non-st:rep:fin} and \ref{non-st:rep:fin2}, it becomes feasible to derive the necessary and sufficient conditions to ensure that every irreducible representation $V_{\mathbf{s}}(\Lambda)$ is finite-dimensional. Specifically, we determine the finite-dimensionality of $V_{\mathbf{s}}(\Lambda)$ with $\Lambda = (\lambda_1, \ldots, \lambda_N)$ by the following steps:
\begin{itemize}
  \item[(1)]\ If $\mathbf{s}$ is standard, use Theorem \ref{standard:rep:fin}; otherwise, go to step (2).
  \item[(2)]\ Consider the ratio $\lambda_i/\lambda_{i+1}$ for any subsequence $s_i s_{i+1} = 00$ or $11$. If there is a $\lambda_i/\lambda_{i+1} = \pm q_i^{\ell}$ for $\ell < 0$, then $V_{\mathbf{s}}(\Lambda)$ is infinite-dimensional; otherwise, go to step (3).
  \item[(3)]\ Consider the ratio $\lambda_i/\lambda_{i+1}$ for some subsequence $\mathfrak{s}_i \mathfrak{s}_{i+1} = 01$ or $10$. If $\lambda_i/\lambda_{i+1} \neq \pm 1$, apply Proposition \ref{non-st:rep:fin}; otherwise, apply Proposition \ref{non-st:rep:fin2}. After that, set $\mathbf{s} := \sigma_i\mathbf{s}$ and $\Lambda := \Lambda^{[i]}$, and return to step (1).
\end{itemize}

\subsubsection{Typical and atypical irreducible representations of $\mathrm{U}_q\big(\mathfrak{g}_{\mathbf{s}}\big)$}\label{se:typicalirreducible representationsUq}
We now distinguish between typical and atypical finite-dimensional irreducible representations of  $\mathrm{U}_q\big(\mathfrak{g}_{\mathbf{s}}\big)$ for an arbitrary $\mathbf{s}\in\mathcal{S}(m|n)$.In what follows, we will concentrate on the properties of representations in the typical case. 
\begin{definition}
A finite-dimensional irreducible representation $V_{\mathbf{s}}(\Lambda)$ over $\mathrm{U}_q\big(\mathfrak{g}_{\mathbf{s}}\big)$ is said to be \textit{typical} if there exists a typical irreducible representation $V(\Lambda')$ over $\mathrm{U}_q\big(\mathfrak{g}\big)$ (in the $\mathbf{s}^{\rm st}$ case) such that $V_{\mathbf{s}}(\Lambda)\simeq V(\Lambda')$ for some $N$-tuple $\Lambda'$.
If not, $V_{\mathbf{s}}(\Lambda)$ is called \textit{atypical}.
\end{definition}

\begin{lemma}\label{rho}
Let $\rho_{\mathbf{s}} $ be the graded half-sum of all positive roots for $\mathfrak{g}_{\mathbf{s}}$. Then
\begin{gather*}
2\rho_{\mathbf{s}}=\sum_{|i|=\bar{0},i\in I_{\mathbf{s}}}(m-n-2\tau^{-1}(i)+1)\varepsilon_{i,\mathbf{s}} +\sum_{|a|=\bar{1},a\in I_{\mathbf{s}}}(3m+n-2\tau^{-1}(a)+1)\varepsilon_{a,\mathbf{s}},
\end{gather*}
where $\tau\in\mathfrak{S}_N$ such that $\mathbf{s}=\tau\mathbf{s}^{\rm st}$.
\end{lemma}
\begin{proof}
    It can be easily obtained from the action of a series of odd reflections on $\rho$ given in \cite[Appendix B]{Zrb93}. 
    
\end{proof}

\begin{proposition}\label{non-st:rep:fin3}
A finite-dimensional irreducible representation $V_{\mathbf{s}}(\Lambda)$ over $\mathrm{U}_q\big(\mathfrak{g}_{\mathbf{s}}\big)$ is typical if and only if
$\prod_{\alpha\in\Phi_{\bar{1},\mathbf{s}}^+}(\Lambda+\rho_{\mathbf{s}}|\alpha)\neq 0$.
\end{proposition}
\begin{proof}
Put $\Lambda^{(0)}=\Lambda$ and $I_{\mathbf{s}}'=I_{\mathbf{s}}\setminus\{N\}$. By definition, there exists a series of indices
$$i_1\in I',~~i_2\in I'_{\sigma_{i_1}\mathbf{s}^{\rm st}},~~\ldots,~~i_r\in I'_{\sigma_{i_{r-1}}\cdots\sigma_{i_1}\mathbf{s}^{\rm st}}$$
with each $|i_p|+|i_p+1|=\bar{1}$ ($1\leqslant p\leqslant r$) such that
\begin{gather*}
V_{\mathbf{s}}\big(\Lambda^{(0)}\big)\simeq V_{\sigma_{i_r}\mathbf{s}}\big(\Lambda^{(1)}\big)\simeq\cdots \simeq V_{\sigma_{i_2}\cdots\sigma_{i_r}\mathbf{s}}\big(\Lambda^{(r-1)}\big)\simeq V\big(\Lambda^{(r)}\big)
\end{gather*}
with $\mathbf{s}=\sigma_{i_r}\cdots\sigma_{i_1}\mathbf{s}^{\rm st}$ and
\begin{align*}
    \Lambda_{\sigma_{i_r}(i)}^{(1)} &=\Lambda_i^{(0)}-\delta_{i,i_r}+\delta_{i,i_r+1}, \\
    \Lambda_{\sigma_{i_{r-p+1}}\cdots\sigma_{i_r}(i)}^{(p)} &=\Lambda_{\sigma_{i_r}(i)}^{(1)}-\sum_{k=2}^p\left(\delta_{\sigma_{i_{r-k+2}}\cdots\sigma_{i_r}(i),i_{r-k+1}}-\delta_{\sigma_{i_{r-k+2}}\cdots\sigma_{i_r}(i),i_{r-k+1}+1}\right),~~2\leqslant p\leqslant r,
\end{align*}
where each irreducible representation $V_{\sigma_{i_{p}}\cdots\sigma_{i_1}\mathbf{s}}\big(\Lambda^{(p)}\big)$ is typical over $\mathrm{U}_q\big(\mathfrak{g}_{\sigma_{i_{p}}\cdots \sigma_{i_1}\mathbf{s}}\big)$.

We proceed by induction on $r$. The case $r = 0$ is clear by Proposition \ref{Zrb:p1}. Assume that this proposition holds for $r = p$. We now consider the case of $r = p + 1$. 
%Let $\varepsilon_1^{\mathfrak{s}}, \ldots, \varepsilon_L^{\mathfrak{s}}$ be the basis of $\mathfrak{h}^{\ast}[\mathfrak{s}]$ such that $\varepsilon_i^{\mathfrak{s}} = \varepsilon_{\tau(i)}$ for $\tau = \sigma_{i_r} \cdots \sigma_{i_1}$. 
Set $\tau=\sigma_{i_1}\cdots\sigma_{i_{p+1}}$ and $\mathbf{s}'=\sigma_{i_{p+1}}\mathbf{s}$. Then, by Lemma \ref{rho}, we have
\begin{align*}
2\rho_{\mathbf{s}} & = \sum_{i=1}^{m} (m-n-2i+1) \varepsilon_{\tau(i),\mathbf{s}} + \sum_{a=m+1}^{m+n} (3m+n-2a+1) \varepsilon_{\tau(a),\mathbf{s}}, \\
2\rho_{\mathbf{s}'} & = \sum_{i=1}^{m} (m-n-2i+1) \varepsilon_{\tau\sigma_{i_{p+1}}(i),\mathbf{s}'} + \sum_{a=m+1}^{m+n} (3m+n-2a+1) \varepsilon_{\tau\sigma_{i_{p+1}}(a),\mathbf{s}'}.
\end{align*}
Therefore,
\begin{gather*}
(\Lambda^{(0)}+\rho_{\mathbf{s}}|\varepsilon_{k,\mathbf{s}}-\varepsilon_{l,\mathbf{s}})
=(\Lambda^{(1)}+\rho_{\mathbf{s}'}|\varepsilon_{\sigma_{i_{p+1}}(k),\mathbf{s}'}
  -\varepsilon_{\sigma_{i_{p+1}}(l),\mathbf{s}'})\neq 0,\quad \forall |k|+|l|=\bar{1},
\end{gather*}
by the induction hypothesis.

Conversely, suppose that $\Lambda$ satisfies $(\Lambda+\rho_{\mathbf{s}}|\alpha)\neq 0$ for each $\alpha\in\Phi_{\bar{1},\mathbf{s}}^+$. By Proposition \ref{non-st:rep:fin}, we find $V_{\mathbf{s}}(\Lambda)\simeq V\big(\Lambda^{(r)}\big)$, which is typical over $\mathrm{U}_q\big(\mathfrak{g}\big)$.

\end{proof}

\begin{corollary}
For any given typical finite-dimensional irreducible representation $V_{\mathbf{s}}(\Lambda)$ over $\mathrm{U}_q\big(\mathfrak{g}_{\mathbf{s}}\big)$, there exists a corresponding typical finite-dimensional irreducible representation $\overline{V}_{\mathbf{s}}(\Lambda)$ with the same highest weight $\Lambda$ over $\mathfrak{g}_{\mathbf{s}}$ (see \cite{Ka77-1}). Furthermore, $V_{\mathbf{s}}(\Lambda)$ specializes to $\overline{V}_{\mathbf{s}}(\Lambda)$ as $q$ approaches $1$.
\end{corollary}

To be more specific, we have
\begin{theorem}\label{non-st:rep:fin4}
Given $\mathbf{s} \in \mathcal{S}(m|n)$, 
%let
%\begin{align*}
%&I^+_{\mathbf{s}} = \{\,i_1 < i_2 < \cdots < i_m \,|\, s_{i_a} = 0~~\text{for}~~1 \leqslant a \leqslant m\,\}, \\
%&I^-_{\mathbf{s}}= \{\,j_1 < j_2 < \cdots < j_n \,|\, s_{j_b} = 1~~\text{for}~~1 \leqslant b \leqslant n\,\}
%\end{align*}
%be the subsets such that $I_{\mathbf{s}} = I^+_{\mathbf{s}} \cup I^-_{\mathbf{s}}$.
we consider the $N$-tuple
$\Lambda = (\lambda_1, \ldots, \lambda_L)$ $(\forall\,\lambda_i \in \mathbb{C}\setminus\{0\})$. 
The following conditions for the typical irreducible highest weight representation $V_{\mathbf{s}}(\Lambda)$ of $\mathrm{U}_q\big(\mathfrak{g}_{\mathbf{s}}\big)$ with highest weight $\Lambda$ are equivalent\,{\rm :}
\begin{enumerate}
    \item[{\rm (1)}] $\dim V_{\mathbf{s}}(\Lambda)<\infty$\,{\rm ;} 
    \item[{\rm (2)}] there exists a series of nonnegative integers 
%$l_1^{\circ}, \ldots, l_m^{\circ}$ and $l_1^{\bullet}, \ldots, l_n^{\bullet}$
$l_{ij}$ for $1\leqslant i<j\leqslant N$ 
such that
\begin{gather}\label{fd:cd:fin1}
\frac{\epsilon_{i} \lambda_{i}}{\epsilon_{j} \lambda_{j}} = q_i^{l_{ij}+\#_{(i,j)} }, \quad \text{if}~~|i|+|j|=\bar{0}, 
\end{gather}
%\begin{gather}\label{fd:cd:fin1}
%\frac{\epsilon_{i_a} \lambda_{i_a}}{\epsilon_{i_{a+1}} \lambda_{i_{a+1}}} = q^{l_a^{\circ} + i_{a+1} - i_a - 1}, \quad
%\frac{\epsilon_{j_b} \lambda_{j_b}}{\epsilon_{j_{b+1}} \lambda_{j_{b+1}}} = q^{-(l_b^{\bullet} + j_{b+1} - j_b - 1)}
%\end{gather}
and
\begin{gather}\label{fd:cd:fin2}
\frac{\epsilon_{i} \lambda_{i}}{\epsilon_{j} \lambda_{j}} = q^{l_{ij}}\neq q^{-(\rho_{\mathbf{s}}|\varepsilon_{i,\mathbf{s}}-\varepsilon_{j,\mathbf{s}})}, \quad \text{if}~~|i|+|j|=\bar{1}
\end{gather}
%\begin{gather}\label{fd:cd:fin2}
%\frac{\epsilon_{i_a} \lambda_{i_a}}{\epsilon_{j_{b}} \lambda_{j_{b}}} = q^{\ell_{i_a j_b}},\quad \ell_{i_a,j_b}\neq -(\rho_{\mathbf{s}}|\varepsilon_{i_a,\mathbf{s}}-\varepsilon_{j_b,\mathbf{s}})
%\end{gather}
for some $\epsilon = (\epsilon_1, \ldots, \epsilon_L) \ (\forall\,\epsilon_i \in \{\pm 1\})$.
\end{enumerate}
Here, 
the notation $\#_{(i,j)}$ denotes the number of 1s {\rm (}resp. 0s{\rm)} appearing in the subsequence $s_{i+1}\cdots s_{j-1}$ if $|i|=|j|=\bar{0}$ {\rm (}resp. $|i|=|j|=\bar{1}${\rm)}. 
\end{theorem}

\begin{corollary}\label{non-st:rep:fin5}
    Under the hypotheses of Theorem \ref{non-st:rep:fin4}, 
    the following conditions for the atypical irreducible highest weight representation $V_{\mathbf{s}}(\Lambda)$ are equivalent\,{\rm :}
    \begin{enumerate}
    \item[{\rm (1)}] $\dim V_{\mathbf{s}}(\Lambda)<\infty$\,{\rm ;} 
    \item[{\rm (2)}] condition \eqref{fd:cd:fin1} still holds, along with condition \eqref{fd:cd:fin2} is replaced by
    \begin{gather}\label{fd:cd:fin3}
    \frac{\epsilon_{i} \lambda_{i}}{\epsilon_{j} \lambda_{j}} = q^{l_{ij}}= q^{-(\rho_{\mathbf{s}}|\varepsilon_{i,\mathbf{s}}-\varepsilon_{j,\mathbf{s}})}, \quad \text{if}~~|i|+|j|=\bar{1}
    \end{gather}
    %\begin{gather}\label{fd:cd:fin3}
%\frac{\epsilon_{i_a} \lambda_{i_a}}{\epsilon_{j_{b}} \lambda_{j_{b}}} = q^{\ell_{i_a j_b}},\quad \ell_{i_a,j_b}=-(\rho_{\mathbf{s}}|\varepsilon_{i_a,\mathbf{s}}-\varepsilon_{j_b,\mathbf{s}})
%\end{gather}
    for some $\epsilon = (\epsilon_1, \ldots, \epsilon_L) \ (\forall\,\epsilon_i \in \{\pm 1\})$.
    \end{enumerate}
\end{corollary}

\begin{remark}
Unlike Theorem~\ref{standard:rep:fin} (which concerns the standard case), Theorem~\ref{non-st:rep:fin4} and Corollary~\ref{non-st:rep:fin5} are stated in terms of the ratios $\lambda_i/\lambda_j$ rather than $\lambda_i/\lambda_{i+1}$. This modification is essential because conditions \eqref{fd:cd:fin1}--\eqref{fd:cd:fin3}, if restricted only to the case $j = i+1$, do not suffice to ensure that $\dim V_{\mathbf{s}}(\Lambda) < \infty$.
\end{remark}

\subsubsection{Nonstandard Young-like diagram}
For a more intuitive illustration of the finite-dimensional irreducible representations $V_{\mathbf{s}}(\Lambda)$ of $\mathrm{U}_q\big(\mathfrak{g}_{\mathbf{s}}\big)$, the graphical notion below will be employed. We need a block strip $\mathsf{box}V_{\mathbf{s}}(\Lambda)$ made of three types of boxes to correspond the module $V_{\mathbf{s}}(\Lambda)$. The $(i,j)$-th box $\mathsf{box}(i,j)$ is determined by the following rule:
\begin{enumerate}
    \item[(1)] $\mathsf{box}(i,j)$=\ \begin{tikzpicture}[baseline={(0, -0.1)}]
    % 绘制正方形边框（四个顶点坐标：(-0.3,-0.3) 左下，(0.3,-0.3) 右下，(0.3,0.3) 右上，(-0.3,0.3) 左上）
    \draw (-0.3,-0.3) rectangle (0.3,0.3);
    \node at (0,0) {$p$};
    \end{tikzpicture}\,, if $|i|+|j|=\bar{0}$ and the radio $\lambda_i/\lambda_{j}=\pm q_i^{l_{ij}+\#_{(i,j)}}$ for $l_{ij}\in \mathbb{Z}_+$, where $p=\bar{d}_{i}k_{ij}$ for 
    \begin{equation*}
        \bar{d}_i=\begin{cases}
            +, &\text{if}~~d_i=1, \\
            -, &\text{if}~~d_i=-1;
        \end{cases}
    \end{equation*} 

    \item[(2)] $\mathsf{box}(i,j)$=\ \begin{tikzpicture}[baseline={(0, -0.1)}]
    \draw (-0.3,-0.3) rectangle (0.3,0.3);
    \fill[black] (-0.3,-0.3) -- (0.3,-0.3) -- (0.3,0.3) -- cycle; % 填充右下三角区域（从左下到右上的对角线下方）
    \end{tikzpicture}\,, if $|i|+|j|=\bar{1}$ and the radio $\lambda_i/\lambda_{j}=\pm q_i^{l_{ij}}$ for $l_{ij}\neq-(\rho_{\mathbf{s}}|\varepsilon_{i,\mathbf{s}}-\varepsilon_{j,\mathbf{s}})$;

    \item[(3)] $\mathsf{box}(i,j)$=\ \begin{tikzpicture}[baseline={(0, -0.1)}]
    \draw (-0.3,-0.3) rectangle (0.3,0.3);
    \fill[black] (0,0) circle (0.2); % (0,0)是圆心坐标，0.1是半径
    \end{tikzpicture}\,, if $|i|+|j|=\bar{1}$ and the radio $\lambda_i/\lambda_{j}=\pm q_i^{l_{ij}}$ for $l_{ij}=-(\rho_{\mathbf{s}}|\varepsilon_{i,\mathbf{s}}-\varepsilon_{j,\mathbf{s}})$.
\end{enumerate}
If $\mathsf{box}V_{\mathbf{s}}(\Lambda)$ does not contain \begin{tikzpicture}[baseline={(0, -0.1)}]
    \draw (-0.3,-0.3) rectangle (0.3,0.3);
    \fill[black] (0,0) circle (0.2); % (0,0)是圆心坐标，0.1是半径
    \end{tikzpicture}\,, $V_{\mathbf{s}}(\Lambda)$ is typical by Theorem \ref{non-st:rep:fin4}. Moreover, if $\mathsf{box}(i,i+1)$=\ \begin{tikzpicture}[baseline={(0, -0.1)}]
    \draw (-0.3,-0.3) rectangle (0.3,0.3);
    \fill[black] (0,0) circle (0.2); % (0,0)是圆心坐标，0.1是半径
    \end{tikzpicture}\,, we have $\lambda_i/\lambda_{i+1}=\pm 1$. Therefore, the trivial representation of $\mathrm{U}_q\big(\mathfrak{g}_{\mathbf{s}}\big)$ is atypical when $m\neq 0$, $n\neq 0$.  
\begin{example}
    Consider the special case $(m,n) = (2, 1)$. Table \ref{table:fin_rep_example_21} provides all possible non-standard Young-like diagrams for the superalgebra $\mathrm{U}_q\big(\mathfrak{sl}_{2|1,\mathbf{s}}\big)$. 
\end{example}
 
\begin{table} % [h!] 控制表格位置为当前位置
    \centering
    \footnotesize
    \caption{Diagrams of finite-dimensional irreducible representations for $\mathrm{U}_q\big(\mathfrak{sl}_{2|1,\mathbf{s}}\big)$} % 表格标题
    \label{table:fin_rep_example_21}
      \vspace{2pt}
    \begin{tabular}{|c|c|c|c|} % 4列，均居中对齐，带竖线
        \hline
        % 第1行（表头）
        ~&~&~&~\\[-0.6em]
        $\mathbf{s}\in\mathcal{S}(2|1)$ & $\mathsf{box}V_{\mathbf{s}}(\Lambda)$ & Typical/Atypical &  $\dim V_{\mathbf{s}}(\Lambda)$  \\
        ~&~&~&~\\[-0.6em]
        \hline
        % 第2-4行：第1列合并3行
        ~&~&~&~\\[-0.6em]
        \multirow{12}{*}{$\mathbf{s}=001$} & \begin{tikzpicture}[baseline={(0, -0.1)}]
        \draw (-0.6,0) rectangle (0,0.6);
        \node at (-0.3,0.3) {$+p$};
         \draw (-0.6,-0.6) rectangle (0,0);
         \fill[black] (-0.6,-0.6) -- (0,-0.6) -- (0,0) -- cycle;
         \draw (0,0) rectangle (0.6,0.6);
         \fill[black] (0,0) -- (0.6,0) -- (0.6,0.6) -- cycle;
        \end{tikzpicture} &Typical  & $4(p+1)$ \\
        ~&~&~&~\\[-0.6em]
        \cline{2-4} % 仅分隔第2-4列
        ~&~&~&~\\[-0.6em]
        & \begin{tikzpicture}[baseline={(0, -0.1)}]
        \draw (-0.6,0) rectangle (0,0.6);
        \node at (-0.3,0.3) {$+p$};
         \draw (-0.6,-0.6) rectangle (0,0);
         \fill[black] (-0.6,-0.6) -- (0,-0.6) -- (0,0) -- cycle;
         \draw (0,0) rectangle (0.6,0.6);
         \fill[black] (0.3,0.3) circle (0.2);
        \end{tikzpicture} & Atypical  & $2p+1$ \\
        ~&~&~&~\\[-0.6em]
        \cline{2-4}
        ~&~&~&~\\[-0.6em]
        & \begin{tikzpicture}[baseline={(0, -0.1)}]
        \draw (-0.6,0) rectangle (0,0.6);
        \node at (-0.3,0.3) {$+p$};
         \draw (-0.6,-0.6) rectangle (0,0);
         \fill[black] (-0.3,-0.3) circle (0.2);
         \draw (0,0) rectangle (0.6,0.6);
         \fill[black] (0,0) -- (0.6,0) -- (0.6,0.6) -- cycle;
        \end{tikzpicture} & Atypical &  $2p+3$ \\
        ~&~&~&~\\[-0.6em]
        \cline{2-4}
        ~&~&~&~\\[-0.6em]
        & \begin{tikzpicture}[baseline={(0, -0.1)}]
        \draw (-0.6,0) rectangle (0,0.6);
        \node at (-0.3,0.3) {$+1$};
         \draw (-0.6,-0.6) rectangle (0,0);
         \fill[black] (-0.3,-0.3) circle (0.2);
         \draw (0,0) rectangle (0.6,0.6);
         \fill[black] (0.3,0.3) circle (0.2);
        \end{tikzpicture} & Atypical  & $2$ \\
        ~&~&~&~\\[-0.6em]
        \hline
        ~&~&~&~\\[-0.6em]
        % 第5-7行：第1列合并3行
        \multirow{12}{*}{$\mathbf{s}=010$} & \begin{tikzpicture}[baseline={(0, -0.1)}]
        \draw (-0.6,0) rectangle (0,0.6);
        \fill[black] (-0.6,0) -- (0,0) -- (0,0.6) -- cycle;
         \draw (-0.6,-0.6) rectangle (0,0);
         \node at (-0.3,-0.3) {$+p$};
         \draw (0,0) rectangle (0.6,0.6);
         \fill[black] (0,0) -- (0.6,0) -- (0.6,0.6) -- cycle;
        \end{tikzpicture} & Typical  & $4(p+1)$ \\
        ~&~&~&~\\[-0.6em]
        \cline{2-4}
        ~&~&~&~\\[-0.6em]
        & \begin{tikzpicture}[baseline={(0, -0.1)}]
        \draw (-0.6,0) rectangle (0,0.6);
        \fill[black] (-0.6,0) -- (0,0) -- (0,0.6) -- cycle;
         \draw (-0.6,-0.6) rectangle (0,0);
         \node at (-0.3,-0.3) {$+p$};
         \draw (0,0) rectangle (0.6,0.6);
         \fill[black] (0.3,0.3) circle (0.2);
        \end{tikzpicture} & Atypical  & $2p+1$ \\
        ~&~&~&~\\[-0.6em]
        \cline{2-4}
        ~&~&~&~\\[-0.6em]
        & \begin{tikzpicture}[baseline={(0, -0.1)}]
        \draw (-0.6,0) rectangle (0,0.6);
        \fill[black] (-0.3,0.3) circle (0.2);
         \draw (-0.6,-0.6) rectangle (0,0);
         \node at (-0.3,-0.3) {$+p$};
         \draw (0,0) rectangle (0.6,0.6);
         \fill[black] (0,0) -- (0.6,0) -- (0.6,0.6) -- cycle;
        \end{tikzpicture} & Atypical  & $2p+3$ \\
        ~&~&~&~\\[-0.6em]
        \cline{2-4}
        ~&~&~&~\\[-0.6em]
        & \begin{tikzpicture}[baseline={(0, -0.1)}]
        \draw (-0.6,0) rectangle (0,0.6);
        \fill[black] (-0.3,0.3) circle (0.2);
         \draw (-0.6,-0.6) rectangle (0,0);
         \node at (-0.3,-0.3) {$+1$};
         \draw (0,0) rectangle (0.6,0.6);
         \fill[black] (0.3,0.3) circle (0.2);
        \end{tikzpicture} & Atypical  & $2$ \\
        ~&~&~&~\\[-0.6em]
        \hline
        ~&~&~&~\\[-0.6em]
        % 第8-10行：第1列合并3行（最后一行多1行，共3行）
        \multirow{12}{*}{$\mathbf{s}=100$} & \begin{tikzpicture}[baseline={(0, -0.1)}]
        \draw (-0.6,0) rectangle (0,0.6);
        \fill[black] (-0.6,0) -- (0,0) -- (0,0.6) -- cycle;
         \draw (-0.6,-0.6) rectangle (0,0);
         \fill[black] (-0.6,-0.6) -- (0,-0.6) -- (0,0) -- cycle;
         \draw (0,0) rectangle (0.6,0.6);
         \node at (0.3,0.3) {$+p$};
        \end{tikzpicture} & Typical  & $4(p+1)$ \\
        ~&~&~&~\\[-0.6em]
        \cline{2-4}
        ~&~&~&~\\[-0.6em]
        & \begin{tikzpicture}[baseline={(0, -0.1)}]
        \draw (-0.6,0) rectangle (0,0.6);
        \fill[black] (-0.6,0) -- (0,0) -- (0,0.6) -- cycle;
         \draw (-0.6,-0.6) rectangle (0,0);
         \fill[black] (-0.3,-0.3) circle (0.2);
         \draw (0,0) rectangle (0.6,0.6);
         \node at (0.3,0.3) {$+p$};
        \end{tikzpicture} & Atypical  & $2p+1$ \\
        ~&~&~&~\\[-0.6em]
        \cline{2-4}
        ~&~&~&~\\[-0.6em]
        & \begin{tikzpicture}[baseline={(0, -0.1)}]
        \draw (-0.6,0) rectangle (0,0.6);
        \fill[black] (-0.3,0.3) circle (0.2);
         \draw (-0.6,-0.6) rectangle (0,0);
         \fill[black] (-0.6,-0.6) -- (0,-0.6) -- (0,0) -- cycle;
         \draw (0,0) rectangle (0.6,0.6);
         \node at (0.3,0.3) {$+p$};
        \end{tikzpicture} & Atypical  & $2p+3$ \\
        ~&~&~&~\\[-0.6em]
        \cline{2-4}
        ~&~&~&~\\[-0.6em]
        & \begin{tikzpicture}[baseline={(0, -0.1)}]
        \draw (-0.6,0) rectangle (0,0.6);
        \fill[black] (-0.3,0.3) circle (0.2);
         \draw (-0.6,-0.6) rectangle (0,0);
         \fill[black] (-0.3,-0.3) circle (0.2);
         \draw (0,0) rectangle (0.6,0.6);
         \node at (0.3,0.3) {$+1$};
        \end{tikzpicture} & Atypical  & $2$ \\[1.6em]
        \hline
    \end{tabular}
\end{table}

\section{Quantum affine general linear superalgebras}\label{se:qafsuperalgebra}
In this section, we review the definition of quantum affine general linear superalgebras in terms of RTT presentations. Additionally, we investigate their PBW basis for arbitrary parity sequence $\mathbf{s}$, based on \cite[Section 2.2.2]{LWZ24}. 

\subsection{RTT presentation $\mathrm{U}_q\big(\widehat{\mathfrak{g}}_{\mathbf{s}}\big)$}\label{se:RTTqafsuperalgebra}
Fix $\mathbf{s}\in\mathcal{S}(m|n)$, we introduce the quantum affine super R-matrix
\begin{align*}
\mathcal{R}_{q,\mathbf{s}}(u,v)&=\sum_{i,j\in I_{\mathbf{s}}}\left(uq_i^{\delta_{ij}}-vq_i^{-\delta_{ij}}\right)E_{ii,\mathbf{s}}\otimes E_{jj,\mathbf{s}}+u\sum_{i>j}\left(q_j-q_j^{-1}\right)E_{ij,\mathbf{s}}\otimes E_{ji,\mathbf{s}}   \\
  &\hspace{3em}+v\sum_{i<j}\left(q_j-q_j^{-1}\right)E_{ij,\mathbf{s}}\otimes E_{ji,\mathbf{s}}\ \in \ \operatorname{End} \mathcal{V}_{\mathbf{s}}^{\otimes 2}[u,v],
\end{align*}
which covers the standard case given in \cite{Zhf16}. Notably, $\mathcal{R}_{q,\mathbf{s}}(u,v)$ satisfies
\begin{gather}\label{Rquv}
    \mathcal{R}_{q,\mathbf{s}}(u,v)=\mathcal{R}_{q,\mathbf{s}}u-\widetilde{\mathcal{R}}_{q,\mathbf{s}}v.
\end{gather}
\begin{lemma}
    The R-matrix $\mathcal{R}_{q,\mathbf{s}}(u,v)$ is a solution of the following quantum Yang-Baxter equation
\begin{gather*}
\mathcal{R}_{q,\mathbf{s}}^{12}(u,v)\mathcal{R}_{q,\mathbf{s}}^{13}(u,w)\mathcal{R}_{q,\mathbf{s}}^{23}(v,w)=\mathcal{R}_{q,\mathbf{s}}^{23}(v,w)\mathcal{R}_{q,\mathbf{s}}^{13}(u,w)\mathcal{R}_{q,\mathbf{s}}^{12}(u,v).
\end{gather*}
\end{lemma}
\begin{proof}
    It can be directly deduced from \eqref{YB:solution}, \eqref{Rq-P} and \eqref{Rquv}. 
    
\end{proof}
The \textit{quantum affine general linear superalgebra} $\mathrm{U}_q\big(\widehat{\mathfrak{g}}_{\mathbf{s}}\big)$ (with trivial central charge) is defined via the Faddeev-Reshetikhin-Takhtajan's presentation as follows.
\begin{definition}
Introduce the formal power series
\begin{gather*}
    t_{ij,\mathbf{s}}(u)=\sum_{r\geqslant 0}t_{ij,\mathbf{s}}^{(r)}u^{-r}\in \mathrm{U}_q\big(\widehat{\mathfrak{g}}_{\mathbf{s}}\big)[[u^{-1}]],\quad 
    \bar{t}_{ij,\mathbf{s}}(u)=\sum_{r\geqslant 0}\bar{t}_{ij,\mathbf{s}}^{(r)}u^{r}\in \mathrm{U}_q\big(\widehat{\mathfrak{g}}_{\mathbf{s}}\big)[[u]],
\end{gather*}
Put
\begin{gather*}
    T_{\mathbf{s}}(u)=\sum_{i,j\in I_{\mathbf{s}}}t_{ij,\mathbf{s}}(u)\otimes E_{ij,\mathbf{s}},\quad \bar{T}_{\mathbf{s}}(u)=\sum_{i,j\in I_{\mathbf{s}}}\bar{t}_{ij,\mathbf{s}}(u)\otimes E_{ij,\mathbf{s}}. 
\end{gather*}
The quantum affine general linear superalgebra $\mathrm{U}_q\big(\widehat{\mathfrak{g}}_{\mathbf{s}}\big)$ is an associative superalgebra with the set of generators 
$$t_{ij,\mathbf{s}}^{(r)},\ \bar{t}_{ij,\mathbf{s}}^{(r)},~~i,j\in I_{\mathbf{s}},\ r\in\mathbb{Z}_+,\quad \text{with parities}\quad \big|t_{ij,\mathbf{s}}^{(r)}\big|=\big|\bar{t}_{ij,\mathbf{s}}^{(r)}\big|=|i|+|j|.$$
These generators satisfy the defining relations with respect to the R-matrix $\mathcal{R}_{q,\mathbf{s}}(u,v)$: 
\begin{align}\label{RTT:af1}
t_{ij,\mathbf{s}}^{(0)}=\bar{t}_{ij,\mathbf{s}}^{(0)}&=0,\quad \textrm{if}~~1\leqslant i<j\leqslant N, \\ \label{RTT:af2}
t_{ii,\mathbf{s}}^{(0)}\bar{t}_{ii,\mathbf{s}}^{(0)}=\bar{t}_{ii,\mathbf{s}}^{(0)}t_{ii,\mathbf{s}}^{(0)}&=1,\quad\textrm{if}~~i\in I_{\mathbf{s}}, \\ \label{RTT:af3}
\mathcal{R}_{q,\mathbf{s}}^{23}(u,v)T^{1}_{\mathbf{s}}(u)T^{2}_{\mathbf{s}}(v)&=T^{2}_{\mathbf{s}}(v)T^{1}_{\mathbf{s}}(u)\mathcal{R}_{q,\mathbf{s}}^{23}(u,v), \\ \label{RTT:af4}
\mathcal{R}_{q,\mathbf{s}}^{23}(u,v)\bar{T}^{1}_{\mathbf{s}}(u)\bar{T}^{2}_{\mathbf{s}}(v)&=\bar{T}^{2}_{\mathbf{s}}(v)\bar{T}^{1}_{\mathbf{s}}(u)\mathcal{R}_{q,\mathbf{s}}^{23}(u,v), \\ \label{RTT:af5}
\mathcal{R}_{q,\mathbf{s}}^{23}(u,v)T^{1}_{\mathbf{s}}(u)\bar{T}^{2}_{\mathbf{s}}(v)&=\bar{T}^{2}_{\mathbf{s}}(v)T^{1}_{\mathbf{s}}(u)\mathcal{R}_{q,\mathbf{s}}^{23}(u,v). 
\end{align}
\end{definition}

The superalgebra $\mathrm{U}_q\big(\widehat{\mathfrak{g}}_{\mathbf{s}}\big)$ is a Hopf superalgebra endowed with the comultiplication $\widehat{\triangle}_{\mathbf{s}}$
%, the counit $\varepsilon$ and the antipode $S$ 
given by: 
\begin{align*}
t_{ij,\mathbf{s}}(u)\mapsto\sum_{k\in I_{\mathbf{s}}} \varsigma_{ik;kj}t_{ik,\mathbf{s}}(u)\otimes t_{kj,\mathbf{s}}(u),\quad \bar{t}_{ij,\mathbf{s}}(u)\mapsto\sum_{k\in I_{\mathbf{s}} } \varsigma_{ik;kj}\bar{t}_{ik,\mathbf{s}}(u)\otimes \bar{t}_{kj,\mathbf{s}}(u).
%\varepsilon&:\ T(u)\mapsto 1,\quad \overline{T}(u)\mapsto 1, \\
%S&:\ T(u)\mapsto T(u)^{-1},\ \overline{T}(u)\mapsto \overline{T}(u)^{-1},
\end{align*}

In the present paper, we aim to express the RTT relations
\eqref{RTT:af3}$-$\eqref{RTT:af5}
more explicitly in terms of generator series. To achieve this, we rewrite the defining relation \eqref{RTT:af5} as
\begin{equation}\label{RTT:af5-ex}
\begin{split}
&\big(q_i^{-\delta_{ik}}v-q_i^{\delta_{ik}}u\big)t_{ij,\mathbf{s}}(u)\bar{t}_{kl,\mathbf{s}}(v)
-\varsigma_{ij;kl}\big(q_j^{-\delta_{jl}}v-q_j^{\delta_{jl}}u\big)\bar{t}_{kl,\mathbf{s}}(v)t_{ij,\mathbf{s}}(u) \\
&\qquad=\varsigma_{ik;kl}\left(q_k-q_k^{-1}\right)\Big(\big(\delta_{k<i}u+\delta_{i<k}v\big)t_{kj,\mathbf{s}}(u)\bar{t}_{il,\mathbf{s}}(v)
-\big(\delta_{j<l}u+\delta_{l<j}v\big)\bar{t}_{kj,\mathbf{s}}(v)t_{il,\mathbf{s}}(u)\Big).
\end{split}
\end{equation}
The defining relations in terms of $t_{ij,\mathbf{s}}(u)$ are obtained from \eqref{RTT:af5-ex} by replacing $\bar{t}$ by $t$,
\begin{equation}\label{RTT:af3-ex}
\begin{split}
&\big(q_i^{-\delta_{ik}}v-q_i^{\delta_{ik}}u\big)t_{ij,\mathbf{s}}(u)t_{kl,\mathbf{s}}(v)
-\varsigma_{ij;kl}\big(q_j^{-\delta_{jl}}v-q_j^{\delta_{jl}}u\big)t_{kl,\mathbf{s}}(v)t_{ij,\mathbf{s}}(u) \\
&\qquad=\varsigma_{ik;kl}\left(q_k-q_k^{-1}\right)\Big(\big(\delta_{k<i}u+\delta_{i<k}v\big)t_{kj,\mathbf{s}}(u)t_{il,\mathbf{s}}(v)
-\big(\delta_{j<l}u+\delta_{l<j}v\big)t_{kj,\mathbf{s}}(v)t_{il,\mathbf{s}}(u)\Big),
\end{split}
\end{equation}
and the defining relations in terms of $\bar{t}_{ij,\mathbf{s}}(u)$ are obtained from \eqref{RTT:af5-ex} by replacing $t$ by $\bar{t}$,
\begin{equation}\label{RTT:af4-ex}
\begin{split}
&\big(q_i^{-\delta_{ik}}v-q_i^{\delta_{ik}}u\big)\bar{t}_{ij,\mathbf{s}}(u)\bar{t}_{kl,\mathbf{s}}(v)
-\varsigma_{ij;kl}\big(q_j^{-\delta_{jl}}v-q_j^{\delta_{jl}}u\big)\bar{t}_{kl,\mathbf{s}}(v)\bar{t}_{ij,\mathbf{s}}(u) \\
&\qquad=\varsigma_{ik;kl}\left(q_k-q_k^{-1}\right)\Big(\big(\delta_{k<i}u+\delta_{i<k}v\big)\bar{t}_{kj,\mathbf{s}}(u)\bar{t}_{il,\mathbf{s}}(v)
-\big(\delta_{j<l}u+\delta_{l<j}v\big)\bar{t}_{kj,\mathbf{s}}(v)\bar{t}_{il,\mathbf{s}}(u)\Big).
\end{split}
\end{equation}

\begin{remark}\label{usual0}
    %The R-matrix $R_{q^{-1},\mathbf{s}}(u,v)$ can be expressed as
    %\begin{align*}
%\mathcal{R}^{q^{-1}}(u,v)&=\sum_{i,j\in I}\left(uq_i^{-\delta_{ij}}-vq_i^{\delta_{ij}}\right)E_{ii}\otimes E_{jj}-u\sum_{i>j}\left(q_j-q_j^{-1}\right)E_{ij}\otimes E_{ji}   \\
 % &\hspace{3em}-v\sum_{i<j}\left(q_j-q_j^{-1}\right)E_{ij}\otimes E_{ji}\ \in \ \operatorname{End} V^{\otimes 2}[u,v].
%\end{align*}
 When $n$ is equal to 0, $\mathcal{R}_{q^{-1},\mathbf{s}}(u,v)$ coincides with the trigonometric solution $R(u,v)$ for the quantum affine algebra $\mathrm{U}_q\big(\widehat{\mathfrak{gl}}_m\big)$, as proposed by Molev-Ragoucy-Sorba \cite[Section 3]{MRS03}. 
\end{remark}

\vspace{1em}
\begin{lemma}
In $\mathrm{U}_q\big(\widehat{\mathfrak{g}}_{\mathbf{s}}\big)\otimes \left(\operatorname{End} \mathcal{V}_{\mathbf{s}}^{\otimes 2}\right)[[u^{\pm 1},v^{\pm 1}]]$, the following equation holds,
\begin{gather}\label{RTT:af6}
\mathcal{R}_{q,\mathbf{s}}^{23}(u,v)\bar{T}^{1}_{\mathbf{s}}(u)T^{2}_{\mathbf{s}}(v)=T^2_{\mathbf{s}}(v)\bar{T}^{1}_{\mathbf{s}}(u)\mathcal{R}_{q,\mathbf{s}}^{23}(u,v).
\end{gather}
\end{lemma}

\begin{proof}
    A straightforward calculation yields
    \begin{gather*}
        \mathcal{R}_{q,\mathbf{s}}(u,v)\mathcal{R}_{q^{-1},\mathbf{s}}(u,v)=\mathcal{R}_{q^{-1},\mathbf{s}}(u,v)\mathcal{R}_{q,\mathbf{s}}(u,v)=\Big((u-v)^2-(q-q^{-1})^2uv\Big)1\otimes 1. 
    \end{gather*}
    Multiplying both sides of relation \eqref{RTT:af5} on the left and right by $\mathcal{R}_{q^{-1},\mathbf{s}}(u,v)$ simultaneously, we obtain
    \begin{gather*}
        T^1_{\mathbf{s}}(u)\bar{T}^{2}_{\mathbf{s}}(v)\mathcal{R}_{q^{-1},\mathbf{s}}^{23}(u,v)=\mathcal{R}_{q^{-1},\mathbf{s}}^{23}(u,v)\bar{T}^{2}_{\mathbf{s}}(v)T^1_{\mathbf{s}}(u).
    \end{gather*}
    Then, applying $\mathcal{P}_{\mathbf{s}}(\,\cdot\,)\mathcal{P}_{\mathbf{s}}$ to this equation and swapping $u\longleftrightarrow v$, we get \eqref{RTT:af6} due to 
    $$\mathcal{R}_{q,\mathbf{s}}(u,v)=-\mathcal{P}_{\mathbf{s}}\mathcal{R}_{q^{-1},\mathbf{s}}(v,u)\mathcal{P}_{\mathbf{s}}.$$
\end{proof}

 We also can rewrite \eqref{RTT:af6} as
\begin{equation}\label{RTT:af6-ex}
\begin{split}
&\big(q_i^{-\delta_{ik}}v-q_i^{\delta_{ik}}u\big)\bar{t}_{ij,\mathbf{s}}(u)t_{kl,\mathbf{s}}(v)
-\varsigma_{ij;kl}\big(q_j^{-\delta_{jl}}v-q_j^{\delta_{jl}}u\big)t_{kl,\mathbf{s}}(v)\bar{t}_{ij,\mathbf{s}}(u) \\
&\qquad=\varsigma_{ik;kl}\left(q_k-q_k^{-1}\right)\Big(\big(\delta_{k<i}u+\delta_{i<k}v\big)\bar{t}_{kj,\mathbf{s}}(u)t_{il,\mathbf{s}}(v)
-\big(\delta_{j<l}u+\delta_{l<j}v\big)t_{kj,\mathbf{s}}(v)\bar{t}_{il,\mathbf{s}}(u)\Big).
\end{split}
\end{equation}

\begin{remark}
    The superalgebra $\mathrm{U}_{q^{-1}}\big(\widehat{\mathfrak{g}}_{\mathbf{s}}\big)$ for $\mathbf{s}=\mathbf{s}^{\rm st}$---defined by generator matrices
    $$L^{+}(u)=\bar{T}(u^{-1}),\quad L^-(u)=T(u^{-1}),$$
    along with defining relations \eqref{RTT:af1}$-$\eqref{RTT:af4} and \eqref{RTT:af6}---is identical to the presentation proposed by Jing, Li, and Zhang \cite{JLZ25} when $q^c=1$. 
\end{remark}

\vspace{1em}
Let $f(u),g(u)$ be the formal series
\begin{gather*}
    f(u)=\sum_{r=0}^{\infty}f^{(r)}u^{-r},\quad g(u)=\sum_{r=0}^{\infty}g^{(r)}u^r\ \in\ \mathbb{C}[[u,u^{-1}]]
\end{gather*}
such that $f^{(0)}g^{(0)}=1$, and let $\mathfrak{d}$ be a nonzero complex number. By its defining relations, $\mathrm{U}_q\big(\widehat{\mathfrak{g}}_{\mathbf{s}}\big)$ has many natural superalgebraic automorphisms given by
\begin{align}\label{auto:af:fu}
    &T(u)\mapsto f(u)T(u),\qquad &&\bar{T}(u)\mapsto g(u)\bar{T}(u), \\ \label{auto:af:c}
    &T(u)\mapsto T(\mathfrak{d} u),&& \bar{T}(u)\mapsto \bar{T}(\mathfrak{d} u).
\end{align}

\vspace{1em}

\subsection{PBW basis of $\mathrm{U}_q\big(\widehat{\mathfrak{g}}_{\mathbf{s}}\big)$}\label{se:PBWbasis}

In the recent work \cite{LWZ24}, the authors presented a RTT-type of PBW basis for the quntum affine superalgebra $\mathrm{U}_q\big(\widehat{\mathfrak{g}}_{\mathbf{s}}\big)$ at at the standard parity sequence $\mathbf{s}=\mathbf{s}^{\rm st}$ in some fixed order. However, two main difficulties arise::  first, no ordered basis has been established for non-standard $\mathbf{s}\in\mathcal{S}(m|n)$; second, even in the standard case, the existing PBW basis is not well-adapted to studying finite-dimensional representations of $\mathrm{U}_q\big(\widehat{\mathfrak{g}}_{\mathbf{s}^{\rm st}}\big)$. In response, we will show that $\mathrm{U}_q\big(\widehat{\mathfrak{g}}_{\mathbf{s}}\big)$ admits an ordered basis with respect to another appropriate order for any parity sequence $\mathbf{s}$. 

Let $\prec$ be the lexicographical order of the countable set $I_{\mathbf{s}}\times I_{\mathbf{s}}\times \mathbb{Z}_+$, we introduce an ordering on the generators of the superalgebra $\mathrm{U}_q\big(\widehat{\mathfrak{g}}_{\mathbf{s}}\big)$: 
\begin{align*}
    \gamma_{i_1,j_1}^{(r_1)} &\prec \gamma_{i_2,j_2}^{(r_2)}\quad \text{if and only if}\quad (j_1-i_1,i_1,r_1)\prec (j_2-i_2,i_2,r_2), \\
    t_{ij,\mathbf{s}}^{(r)} &\prec \bar{t}_{ij,\mathbf{s}}^{(r)}\quad~~ \text{for any triples}\quad (i,j,r)
\end{align*}
for $\gamma\in\{\,t_{\mathbf{s}},\bar{t}_{\mathbf{s}}\,\}$.
The following theorem establishes a PBW basis of $\mathrm{U}_q\big(\widehat{\mathfrak{g}}_{\mathbf{s}}\big)$ with respect to the aforementioned order. 

%\begin{proposition}\label{base:JA}
%Let $\mathcal{B}$ be the set of all ordered monomials
%\begin{equation}\label{base:1}
%\begin{split}
%&\mathop{\overrightarrow{\prod}}\limits_{1\leqslant i\leqslant N}\bigg\{  \mathop{\overrightarrow{\prod}}\limits_{j<i}\left((\bar{t}_{j,i}^{(0)})^{\nu_{j,i,0}}(\bar{t}_{j,i}^{(1)})^{\nu_{j,i,1}}(t_{j,i}^{(1)})^{\mu_{j,i,1}}\cdots\right)  \\
%&\hspace{6em}\times \left((\bar{t}_{i,i}^{(0)})^{\nu_{i,i,0}}(t_{i,i}^{(0)})^{\mu_{i,i,0}}(\bar{t}_{i,i}^{(1)})^{\nu_{i,i,1}}(t_{i,i}^{(1)})^{\mu_{i,i,1}}\cdots \right) \\
%&\hspace{6em} \times \mathop{\overrightarrow{\prod}}\limits_{j>i}\left((t_{j,i}^{(0)})^{\mu_{j,i,0}}(\bar{t}_{j,i}^{(1)})^{\nu_{j,i,1}}(t_{j,i}^{(1)})^{\mu_{j,i,1}}\cdots\right) \bigg\}
%\end{split}
%\end{equation}
%with the exponents
%
%Then $\mathfrak{B}$ forms a basis of $\mathrm{U}_q\big(\widehat{\mathfrak{gl}}_{m|n}\big)$ over $\mathbb{C}(q)$.
%\end{proposition}

\begin{theorem}\label{base:rep-af}
Let $\mathcal{B}_{\mathbf{s}}$ be the set of all ordered monomials
\begin{equation}\label{base:af}
\begin{split}
&\mathop{\overrightarrow{\prod}}\limits_{1-N\leqslant k\leqslant 1}\mathop{\overrightarrow{\prod}}\limits_{1-k\leqslant i\leqslant N} \bigg\{ \Big(t_{i,i+k,\mathbf{s}}^{(0)}\Big)^{b_{i,i+k,0}}\Big(t_{i,i+k,\mathbf{s}}^{(1)}\Big)^{b_{i,i+k,1}}
     \Big(\bar{t}_{i,i+k,\mathbf{s}}^{(1)}\Big)^{\bar{b}_{i,i+k,1}} \cdots\bigg\} \\
&\hspace{5em}\times \mathop{\overrightarrow{\prod}}\limits_{1\leqslant i\leqslant N}\bigg\{ \Big(t_{ii,\mathbf{s}}^{(0)}\Big)^{b_{i,i,0}}\Big(\bar{t}_{ii,\mathbf{s}}^{(0)}\Big)^{\bar{b}_{i,i,0}}\Big(t_{ii,\mathbf{s}}^{(1)}\Big)^{b_{i,i,0}}\Big(\bar{t}_{ii,\mathbf{s}}^{(1)}\Big)^{\bar{b}_{i,i,1}}\cdots\bigg\} \\
&\hspace{5em} \times \mathop{\overrightarrow{\prod}}\limits_{1\leqslant k\leqslant N-1}\mathop{\overrightarrow{\prod}}\limits_{1\leqslant i\leqslant k} \bigg\{ \Big(\bar{t}_{i,i+k,\mathbf{s}}^{(0)}\Big)^{\bar{b}_{i,i+k,0}}\Big(t_{i,i+k,\mathbf{s}}^{(1)}\Big)^{b_{i,i+k,1}}
     \Big(\bar{t}_{i,i+k,\mathbf{s}}^{(1)}\Big)^{\bar{b}_{i,i+k,1}} \cdots\bigg\}
\end{split}
\end{equation}
with the exponents 
\begin{align}\label{base:af-co1}
&b_{i,j,r},\,\bar{b}_{i,j,r}\in\mathbb{Z}_+,\quad \text{if}~~|i|+|j|=\bar{0}, \\ \label{base:af-co2}
&b_{i,j,r},\,\bar{b}_{i,j,r}\in\{0,1\},\quad \text{if}~~|i|+|j|=\bar{1}, \\ \label{base:af-co3}
&b_{i,i,0}\times \bar{b}_{i,i,0}=0\quad\text{for }i\in I_{\mathbf{s}}.
\end{align} 
Then the monomial set $\mathcal{B}_{\mathbf{s}}$ forms an ordered basis of $\mathrm{U}_q\big(\widehat{\mathfrak{g}}_{\mathbf{s}}\big)$.
\end{theorem}

\begin{proof}
  We first prove the set $\mathcal{B}_{\mathbf{s}}$ spans the whole superalgebra $\mathrm{U}_q\big(\widehat{\mathfrak{gl}}_{\mathbf{s}}\big)$. By relation \eqref{RTT:af5-ex}, we have
  \begin{align*}
      t_{ij,\mathbf{s}}(u)\bar{t}_{kl,\mathbf{s}}(v)&=\varsigma_{ij;kl}\frac{q_j^{-\delta_{jl}}v-q_j^{\delta_{jl}}u}{q_i^{-\delta_{ik}}v-q_i^{\delta_{ik}}u}\bar{t}_{kl,\mathbf{s}}(v)t_{ij,\mathbf{s}}(u) \\
      &+\frac{\varsigma_{ik;kl}(q_k-q_k^{-1})}{q_i^{-\delta_{ik}}v-q_i^{\delta_{ik}}u}\Big\{(\delta_{k<i}u+\delta_{i<k}v)t_{kj,\mathbf{s}}(u)\bar{t}_{il,\mathbf{s}}(v)-(\delta_{j<l}u+\delta_{l<j}v)\bar{t}_{kj,\mathbf{s}}(v)t_{il,\mathbf{s}}(u)\Big\},
  \end{align*}
  which forces
  \begin{gather}\label{Uqspan:1}
      t_{ij,\mathbf{s}}^{(r)}\bar{t}_{kl,\mathbf{s}}^{(s)}\in\operatorname{span}_{\mathbb{C}}\Big\{\,\bar{t}_{kl,\mathbf{s}}^{(a_1)}t_{ij,\mathbf{s}}^{(b_1)},\ t_{kj,\mathbf{s}}^{(a_2)}\bar{t}_{il,\mathbf{s}}^{(b_2)},\ \bar{t}_{kj,\mathbf{s}}^{(a_2)}t_{il,\mathbf{s}}^{(b_3)}\,\Big|\,\text{ each }a_i,b_i\in\mathbb{Z}_+\,\Big\} 
  \end{gather}
  for $r,s\in\mathbb{Z}_+$. 
  Similarly, we also obtain by relation \eqref{RTT:af6-ex}
  \begin{gather}\label{Uqspan:2}
      t_{ij,\mathbf{s}}^{(r)}\bar{t}_{kl,\mathbf{s}}^{(s)}\in\operatorname{span}_{\mathbb{C}}\Big\{\,\bar{t}_{kl,\mathbf{s}}^{(a_1)}t_{ij,\mathbf{s}}^{(b_1)},\ t_{il,\mathbf{s}}^{(a_2)}\bar{t}_{kj,\mathbf{s}}^{(b_2)},\ \bar{t}_{il,\mathbf{s}}^{(a_2)}t_{kj,\mathbf{s}}^{(b_3)}\,\Big|\,\text{ each }a_i,b_i\in\mathbb{Z}_+\,\Big\} 
  \end{gather}
  for $r,s\in\mathbb{Z}_+$. 

  Suppose that $j-i>l-k$. It implies that the generator $\bar{t}_{kl,\mathbf{s}}^{(s)}$ precedes $t_{ij,\mathbf{s}}^{(r)}$. Either the condition $j-k<l-i$ or $j-k>l-i$ allows us to deduce that
  \begin{gather}\label{Uqspan:3}
      t_{ij,\mathbf{s}}^{(r)}\bar{t}_{kl,\mathbf{s}}^{(s)}\in \operatorname{span}_{\mathbb{C}}\mathcal{B}_{\mathbf{s}}
  \end{gather}
  due to \eqref{Uqspan:1} and \eqref{Uqspan:2}. 
  As for the case $j-k=l-i$, we find that
  \begin{gather*}
      j-k=l-i  \quad \Rightarrow \quad j+i=k+l \quad \Rightarrow \quad j-i=2(k-i)+l-k,
  \end{gather*}
  hence, $i<k$. According to \eqref{Uqspan:2}, we still give \eqref{Uqspan:3}. 

  Now taking $j-i=l-k$ and $i>k$, then \eqref{Uqspan:3} still holds by \eqref{Uqspan:1}. Finally, if $i=k$, $j=l$, we have
  \begin{gather*}
      [t_{ij,\mathbf{s}}(u),\,\bar{t}_{ij,\mathbf{s}}(v)]=0,
  \end{gather*}
  which is equal to $t_{ij,\mathbf{s}}^{(r)}t_{ij,\mathbf{s}}^{(s)}=(-1)^{|i|+|j|}t_{ij,\mathbf{s}}^{(s)}t_{ij,\mathbf{s}}^{(r)}$ for $r,s\in\mathbb{Z}_+$. 

  The above arguments can also be used to get
  \begin{gather*}
      t_{ij,\mathbf{s}}^{(r)}t_{kl,\mathbf{s}}^{(s)},\ \bar{t}_{ij,\mathbf{s}}^{(r)}t_{kl,\mathbf{s}}^{(s)},\ \bar{t}_{ij,\mathbf{s}}^{(r)}\bar{t}_{kl,\mathbf{s}}^{(s)}\in \operatorname{span}_{\mathbb{C}}\mathcal{B}_{\mathbf{s}}. 
  \end{gather*}

  Moreover, the proof of "linear independence" part in \cite[Proposition 2.10]{LWZ24} is independent of both the ordering of generators and the specific parity sequence. Consequently, the same argument also establishes the linear independence of $\mathcal{B}_{\mathbf{s}}$.
  
\end{proof}

Theorem \ref{base:rep-af} allows us to define the following $\mathbb{Z}_2$-graded subspaces. Let $\mathrm{N}^+$ (resp. $\mathrm{N}^+$, or $\mathrm{U}^0$) be the $\mathbb{Z}_2$-graded subspace spanned by all ordered monomials in $t_{ij,\mathbf{s}}^{(r)},\bar{t}_{ij,\mathbf{s}}^{(r)}$ for $j-i<0$ (resp. $j-i>0$, or $j-i=0$). Set
\begin{gather*}
    \mathrm{U}^+:=\mathrm{U}^0\mathrm{N}^+,\qquad \mathrm{U}^-:=\mathrm{N}^-\mathrm{U}^0.
\end{gather*}
It follows that $\mathrm{U}_q\big(\widehat{\mathfrak{g}}_{\mathbf{s}}\big)$ has the decomposition
\begin{gather*}
    \mathrm{U}_q\big(\widehat{\mathfrak{g}}_{\mathbf{s}}\big)\simeq \mathrm{U}^- \mathrm{U}^0 \mathrm{U}^+. 
\end{gather*}
Clearly, we find that
\begin{gather*}
     \mathrm{U}_q\big(\widehat{\mathfrak{g}}_{\mathbf{s}}\big)\mathrm{N}^-\subset \mathrm{U}_q\big(\widehat{\mathfrak{g}}_{\mathbf{s}}\big)+\mathrm{N}^-\mathrm{N}^++\mathrm{U}^-. 
\end{gather*}
Note that these $\mathbb{Z}_2$-graded subspaces do not form the sub-superalgebras of $\mathrm{U}_q\big(\widehat{\mathfrak{g}}_{\mathbf{s}}\big)$.

\subsection{$q$-Super Yangian $\mathrm{Y}_q\big(\mathfrak{g}_{\mathbf{s}}\big)$}\label{se:qsuperYangian}
 
\begin{definition}
    The $q$-super Yangian $\mathrm{Y}_q\big(\mathfrak{g}_{\mathbf{s}}\big)$ is an associative superalgebra generated by elements $\left(\bar{t}_{ii,\mathbf{s}}^{(0)}\right)^{-1}$, $\bar{t}_{ij,\mathbf{s}}^{(r)}$ for $i,j\in I_{\mathbf{s}}$ and $r\in\mathbb{Z}_+$ subject to the defining relations \eqref{RTT:af1} and \eqref{RTT:af4-ex}.
\end{definition}
From Theorem \ref{base:rep-af}, there exists an embedding of superalgebras from $\mathrm{Y}_q\big(\mathfrak{g}_{\mathbf{s}}\big)$ to $\mathrm{U}_q\big(\widehat{\mathfrak{g}}_{\mathbf{s}}\big)$ given by
\begin{gather}\label{qSY-embed}
    \left(\bar{t}_{ii,\mathbf{s}}^{(0)}\right)^{-1}\mapsto t_{ii,\mathbf{s}}^{(0)},\qquad \bar{t}_{ij,\mathbf{s}}^{(r)}\mapsto \bar{t}_{ij,\mathbf{s}}^{(r)}
\end{gather}
for $i,j\in I_{\mathbf{s}}$ and $r\in\mathbb{Z}_+$. Then we may regard $\mathrm{Y}_q\big(\mathfrak{g}_{\mathbf{s}}\big)$ as a sub-superalgebra of $\mathrm{U}_q\big(\widehat{\mathfrak{g}}_{\mathbf{s}}\big)$ 
with the ordered basis consisting of all ordered monomials of the form 
\begin{equation}\label{base:af-qSY}
\begin{split}
&\mathop{\overrightarrow{\prod}}\limits_{1-N\leqslant k\leqslant 1}\mathop{\overrightarrow{\prod}}\limits_{1-k\leqslant i\leqslant N} \bigg\{ 
     \Big(\bar{t}_{i,i+k,\mathbf{s}}^{(1)}\Big)^{\bar{b}_{i,i+k,1}} \Big(\bar{t}_{i,i+k,\mathbf{s}}^{(2)}\Big)^{\bar{b}_{i,i+k,2}}\cdots\bigg\} \times \mathop{\overrightarrow{\prod}}\limits_{1\leqslant i\leqslant N}\mathop{\overrightarrow{\prod}}\limits_{r\geqslant 0}\bigg\{ \Big(\bar{t}_{ii,\mathbf{s}}^{(r)}\Big)^{c_{i,r}}\bigg\} \\
&\hspace{5em} \times \mathop{\overrightarrow{\prod}}\limits_{1\leqslant k\leqslant N-1}\mathop{\overrightarrow{\prod}}\limits_{1\leqslant i\leqslant k} \bigg\{ \Big(\bar{t}_{i,i+k,\mathbf{s}}^{(0)}\Big)^{\bar{b}_{i,i+k,0}}
     \Big(\bar{t}_{i,i+k,\mathbf{s}}^{(1)}\Big)^{\bar{b}_{i,i+k,1}} \cdots\bigg\}
\end{split}
\end{equation}
whose exponents satisfy \eqref{base:af-co1}, \eqref{base:af-co2} and $c_{i,r}\in\mathbb{Z}$. 

\begin{remark}
    There exists another version of the $q$-super Yangian, defined as the subsuperalgebra of $\mathrm{U}_q\big(\widehat{\mathfrak{g}}_{\mathbf{s}}\big)$ generated by elements $\left(t_{ii,\mathbf{s}}^{(0)}\right)^{-1}$, $t_{ij,\mathbf{s}}^{(r)}$ for $i,j\in I_{\mathbf{s}}$ and $r\in\mathbb{Z}_+$. This subsuperalgebra is isomorphic to $\mathrm{Y}_q\big(\mathfrak{g}_{n|m,\mathbf{s}}\big)$. 
\end{remark}

\vspace{1em}
\section{Highest weight representations of $\mathrm{U}_q\big(\widehat{\mathfrak{g}}_{\mathbf{s}}\big)$}\label{se:highestweightreps}
Given $\mathbf{s} \in \mathcal{S}(m|n)$, we first develop some necessary structural results on highest-weight representations of the quantum affine superalgebra $\mathrm{U}_q\big(\widehat{\mathfrak{g}}_{\mathbf{s}}\big)$, before studying its finite-dimensional irreducible representations. In particular, we construct two fundamental classes of such representations: Verma modules and evaluation representations.
Motivated by \cite{GM10,JLM20-3,Mo07,Mo22,Zrb95,Zrb96}, we adopt the formal series to describe the representations of variety superalgebras. 
We generalize the definition of the highest weight representation for $\mathrm{U}_q\big(\widehat{\mathfrak{gl}}_N\big)$ to the super case as follows. 

\begin{definition}\label{hw:rep}
A representation $V$ is called a \textit{highest weight representation} over $\mathrm{U}_q\big(\widehat{\mathfrak{g}}_{\mathbf{s}}\big)$ if $V$ is generated by a non-zero vector $\zeta\in V$ such that
\begin{align*}
&t_{ij,\mathbf{s}}(u)\zeta=\bar{t}_{ij,\mathbf{s}}(u)\zeta=0,\qquad \text{for }~1\leqslant i<j\leqslant N, \\
&t_{ii,\mathbf{s}}(u)\zeta=\lambda_i(u)\zeta,\qquad \bar{t}_{ii,\mathbf{s}}(u)\zeta=\bar{\lambda}_i(u)\zeta,\qquad \text{for }~i\in I_{\mathbf{s}},
\end{align*}
where $\lambda_i(u),\bar{\lambda}_i(u)$ are the formal power series given by
\begin{gather}\label{formal_lambda}
\lambda_i(u)=\sum_{r=0}^{\infty}\lambda_i^{(r)}u^{-r},\qquad \bar{\lambda}_i(u)=\sum_{r=0}^{\infty}\bar{\lambda}_i^{(r)}u^{r},
\end{gather}
for all coefficients 
\begin{gather}\label{formal_lambda2}
    \lambda_i^{(r)},\bar{\lambda}_i^{(r)}\in \mathbb{C},\quad\text{and}\quad \lambda_i^{(0)}\times\bar{\lambda}_i^{(0)}=1,\quad \forall\, i\in I_{\mathbf{s}}.
\end{gather}
    Set the $N$-tuples
\begin{gather*}
\lambda(u)=(\lambda_1(u),\ldots,\lambda_{N}(u)),\qquad \bar{\lambda}(u)=(\bar{\lambda}_1(u),\ldots,\bar{\lambda}_{N}(u)).
\end{gather*}
The vector $\zeta$ and the pair $(\lambda(u);\bar{\lambda}(u))$ are referred to as the \textit{maximal vector} and the \textit{highest weights} of $V$, respectively. 
\end{definition}

\begin{proposition}\label{fd_irreducible representations:af}
Every finite-dimensional irreducible representation for $\mathrm{U}_q\big(\widehat{\mathfrak{g}}_{\mathbf{s}}\big)$ is of highest weight type.
\end{proposition}
\begin{proof}
Let $W$ be a finite-dimensional irreducible representation of $\mathrm{U}_q\big(\widehat{\mathfrak{g}}_{\mathbf{s}}\big)$.  Set
\begin{gather*}
W_0:=\big\{\,\omega\in W\,\big|\,t_{ij,\mathbf{s}}(u)\omega=\bar{t}_{ij,\mathbf{s}}(u)\omega=0~\text{ for }~1\leqslant i<j\leqslant N\,\big\}.
\end{gather*}
We claim that the $\mathbb{Z}_2$-graded subspace $W_0\neq 0$. Let $0\neq \omega_0\in W$ be a joint eigenvector of $t_{ii,\mathbf{s}}^{(0)}$, $\bar{t}_{ii,\mathbf{s}}^{(0)}$ for all $i\in I_{\mathbf{s}}$ such that
\begin{gather*}
t_{ii,\mathbf{s}}^{(0)}\omega_0=\mu_i\omega_0,\quad \bar{t}_{ii,\mathbf{s}}^{(0)}\omega_0=\mu_i^{-1}\omega_0
\end{gather*}
for $\mu_i\in\mathbb{C}\setminus\{0\}$. 
Following relation \eqref{RTT:af5-ex}$-$\eqref{RTT:af4-ex} and \eqref{RTT:af6-ex}, we have for $k<l$, 
\begin{align*}
&t_{ii,\mathbf{s}}^{(0)}\bar{t}_{kl,\mathbf{s}}(v)\omega_0=q_i^{\delta_{il}-\delta_{ik}}\bar{t}_{kl,\mathbf{s}}(v)t_{ii,\mathbf{s}}^{(0)}\omega_0=q_i^{\delta_{il}-\delta_{ik}}\mu_i \bar{t}_{kl,\mathbf{s}}(v)\omega_0, \\
&t_{ii,\mathbf{s}}^{(0)}t_{kl,\mathbf{s}}(v)\omega_0=q_i^{\delta_{il}-\delta_{ik}}t_{kl,\mathbf{s}}(v)t_{ii,\mathbf{s}}^{(0)}\omega_0=q_i^{\delta_{il}-\delta_{ik}}\mu_i t_{kl,\mathbf{s}}(v)\omega_0, \\
&\bar{t}_{ii,\mathbf{s}}^{(0)}t_{kl,\mathbf{s}}(v)\omega_0=q_i^{\delta_{ik}-\delta_{il}}t_{kl,\mathbf{s}}(v)\bar{t}_{ii,\mathbf{s}}^{(0)}\omega_0=q_i^{\delta_{ik}-\delta_{il}}\mu_i^{-1} t_{kl,\mathbf{s}}(v)\omega_0, \\
&\bar{t}_{ii,\mathbf{s}}^{(0)}\bar{t}_{kl,\mathbf{s}}(v)\omega_0=q_i^{\delta_{ik}-\delta_{il}}\bar{t}_{kl,\mathbf{s}}(v)\bar{t}_{ii,\mathbf{s}}^{(0)}\omega_0=q_i^{\delta_{ik}-\delta_{il}}\mu_i^{-1} \bar{t}_{kl,\mathbf{s}}(v)\omega_0.
\end{align*}
Suppose that $W_0=0$. For any $\omega\in W$, there exists some pair $(k_1,l_1)$ with $k_1<l_1$ such that $t_{k_1,l_1}(v)\omega\neq 0$ or $\bar{t}_{k_1,l_1}(v)\omega\neq 0$.
Let us assume that $\omega_1=t_{k_1,l_1}^{(r_1)}\omega\neq 0$ for some $r_1\in\mathbb{Z}_+$. According to the hypothesis, there also exists a pair $(k_2,l_2)$ with $k_2<l_2$ such that $t_{k_2,l_2}(v)\omega_1\neq 0$ or $\bar{t}_{k_2,l_2}(v)\omega_1\neq 0$. Either, let $\omega_2=\bar{t}_{k_2,l_2}^{(r_2)}\omega_1\neq 0$ for some $r_2\in\mathbb{Z}_+$. And so on, we obtain an infinite set $\Pi$ of vectors $\omega,\omega_1,\omega_2,\ldots$.
The eigenvalues of the action of the sets $\{t_{ii,\mathbf{s}}^{(0)}|i\in I_{\mathbf{s}}\}$ and $\{\bar{t}_{ii,\mathbf{s}}^{(0)}|i\in I_{\mathbf{s}}\}$ on the elements of $\Pi$ are pairwise distinct, hence, $\Pi$ is linearly independent. This contradicts the finite dimensionality of $W$.

Next, we need to show that $W_0$ is invariant under all $t_{ii,\mathbf{s}}(u),\bar{t}_{ii,\mathbf{s}}(u)$. Choose a nonzero vector $\omega\in W$, we argue with
\begin{gather*}
t_{kl,\mathbf{s}}(v)t_{ii,\mathbf{s}}(u)\omega,\quad t_{kl,\mathbf{s}}(v)\bar{t}_{ii,\mathbf{s}}(u)\omega,\quad \bar{t}_{kl,\mathbf{s}}(v)t_{ii,\mathbf{s}}(u)\omega,\quad \bar{t}_{kl,\mathbf{s}}(v)\bar{t}_{ii,\mathbf{s}}(u)\omega,\quad k<l
\end{gather*}
for the case of $i<l$ or $i\geqslant l>k$. Consider the pair $(\gamma,\gamma')\in\{(t_{\mathbf{s}},t_{\mathbf{s}}),(t_{\mathbf{s}},\bar{t}_{\mathbf{s}}),(\bar{t}_{\mathbf{s}},t_{\mathbf{s}}),(\bar{t}_{\mathbf{s}},\bar{t}_{\mathbf{s}})\}$. If $i<l$,
\begin{equation*}
\begin{split}
\gamma_{kl}(v)\gamma'_{ii}(u)\omega&=\gamma_{kl}(v)\gamma'_{ii}(u)\omega-\frac{q_i^{\delta_{ik}}u-q_i^{-\delta_{ik}}v}{u-v}\gamma'_{ii}(u)\gamma_{kl}(v)\omega \\
&=\frac{\varsigma_{ik;kl}\left(q_k-q_k^{-1}\right)}{u-v}\bigg\{\left(\delta_{k<i}u+\delta_{i<k}v\right)\gamma'_{ki}(u)\gamma_{il}(v)
  -u\gamma_{ki}(v)\gamma'_{il}(u)\bigg\}\omega =0;
\end{split}
\end{equation*}
Otherwise,
\begin{equation*}
\begin{split}
\gamma_{kl}(u)\gamma'_{ii}(v)\omega&=\gamma_{kl}(u)\gamma'_{ii}(v)\omega-\frac{q_i^{\delta_{il}}u-q_i^{-\delta_{il}}v}{u-v}\gamma'_{ii}(v)\gamma_{kl}(u)\omega \\
&=\frac{\varsigma_{ik;kl}\left(q_k-q_k^{-1}\right)}{v-u}\bigg\{v\gamma_{il}(u)\gamma'_{ki}(v)-\left(\delta_{l<i}u+\delta_{i<l}v\right)\gamma_{il}(u)\gamma'_{ki}(v)
  \bigg\}\omega =0.
\end{split}
\end{equation*}

Finally, we have in $\mathrm{U}_q\big(\widehat{\mathfrak{g}}_{\mathbf{s}}\big)$,
\begin{equation*}
\begin{split}
&[\gamma_{ii}(u),\,\gamma'_{ii}(v)]=0, \\
&[\gamma_{ii}(u),\,\gamma'_{kk}(v)]=\frac{q_k-q_k^{-1}}{v-u}\left(v\gamma_{ki}(u)\gamma'_{ik}(v)-u\gamma'_{ki}(v)\gamma_{ik}(u)\right)
\end{split}
\end{equation*}
for the pair $(\gamma,\gamma')$ mentioned as above and $i<k$. That is to say, for all $\omega\in W_0$ and $i<k$,
\begin{gather*}
[\gamma_{ii}(u),\,\gamma'_{kk}(v)]\zeta=0. 
\end{gather*}
If $i>k$, we have
\begin{gather}\label{hirreducible representations:1}
(v-u)\Big(\gamma_{ii}(u)\gamma'_{kk}(v)-\gamma'_{kk}(v)\gamma_{ii}(u) \Big)\omega=\left(q_k-q_k^{-1}\right)\Big(u\gamma_{ki}(u)\gamma'_{ik}(v)-v\gamma_{ki}(v)\gamma'_{ik}(u)\Big)\omega.
\end{gather}
We substitute
\begin{align*}
\gamma_{ki}(u)\gamma'_{ik}(v)&=(-1)^{|i|+|k|}\gamma'_{ik}(v)\gamma_{ki}(u)+\left(q_k-q_k^{-1}\right)v\left(\gamma_{ii}(u)\gamma'_{kk}(v)-\gamma_{ii}(v)\gamma'_{kk}(u)\right) \\
\gamma_{ki}(v)\gamma'_{ik}(u)&=(-1)^{|i|+|k|}\gamma'_{ik}(u)\gamma_{ki}(v)+\left(q_k-q_k^{-1}\right)u\left(\gamma'_{kk}(v)\gamma_{ii}(u)-\gamma'_{kk}(u)\gamma_{ii}(v)\right)
\end{align*}
into \eqref{hirreducible representations:1} to obtain
\begin{gather*}
\Big((v-u)-(q-q^{-1})^2uv\Big)[\gamma_{ii}(u),\,\gamma'_{kk}(v)]\omega=\left(q_i-q_i^{-1}\right)\Big(u\gamma_{ik}'(v)\gamma_{ki}(u)-v\gamma_{ik}'(u)\gamma_{ki}(v)\Big)\omega=0.
\end{gather*}
These calculations imply that $t_{ii,\mathbf{s}}(u),\bar{t}_{ii,\mathbf{s}}(u)$ for all $i\in I_{\mathbf{s}}$ act on $W_0$ as pairwise commuting operators,
then there exists at least one joint eigenvector of all $t_{ii,\mathbf{s}}(u),\bar{t}_{ii,\mathbf{s}}(u)$ in $W_0$. Comparing to Definition \ref{hw:rep}, $W_0$ is a highest weight representation of $\mathrm{U}_q\big(\widehat{\mathfrak{g}}_{\mathbf{s}}\big)$. Using the irreducibility of $W$, we have $W=W_0$.

\end{proof}

\subsection{Construction for Verma modules}\label{se:afVermamodule}
Now we proceed to construct a class of highest weight irreducible representations. Given a nonzero vector $\zeta$, we consider the one-dimensional vector space $\mathbb{C}$-spanned by $\zeta$. Define the action of the $\mathbb{Z}_2$-graded subspace $\mathrm{U}^+$ on $\mathbb{C}\zeta$ by
\begin{align*}
    &t_{ij,\mathbf{s}}(u)\zeta=\bar{t}_{ij,\mathbf{s}}(u)\zeta=0,\qquad 1\leqslant i<j\leqslant N, \\
    &t_{ii,\mathbf{s}}(u)\zeta=\lambda_i(u)\zeta,\quad \bar{t}_{ii,\mathbf{s}}(u)\zeta=\bar{\lambda}_i(u)\zeta,\qquad i\in I_{\mathbf{s}}, \\
    &t_{ii,\mathbf{s}}(u)X\zeta=\lambda_i(u)X\zeta,\quad \bar{t}_{ii,\mathbf{s}}(u)X\zeta=\bar{\lambda}_i(u)X\zeta,\qquad X\in \mathrm{U}^0,\ i\in I_{\mathbf{s}},
\end{align*}
where $\lambda_i(u)$, $\bar{\lambda}_i(u)$ are the formal series satisfying \eqref{formal_lambda} and \eqref{formal_lambda2}. According to the last part of the proof of Proposition \ref{fd_irreducible representations:af}, every homogeneous element $X\in \mathrm{U}^0$ satisfies $[X,\, t_{ii}(u)]\zeta=[X,\,\bar{t}_{ii}(u)]\zeta=0$, thus, the above definition is well-defined. 

Introduce 
\begin{gather*}
    M(\lambda(u);\bar{\lambda}(u)):=\mathrm{U}_q\big(\widehat{\mathfrak{g}}_{\mathbf{s}}\big)\otimes_{\mathrm{U}^+}\zeta.
\end{gather*}
Here, we use the previous notations for $\lambda(u)$ and $\bar{\lambda}(u)$. Due to Theorem \ref{base:rep-af}, $M(\lambda(u);\bar{\lambda}(u))\simeq \mathrm{N}^-\otimes_{\mathrm{U}^+}\zeta$. It serves as a representation on $\mathrm{U}_q\big(\widehat{\mathfrak{g}}_{\mathbf{s}}\big)$ in the following sense. For all $X\in \mathrm{U}_q\big(\widehat{\mathfrak{g}}_{\mathbf{s}}\big)$ and $Y\in\mathrm{N}^-$, if $XY$ has the expression
\begin{gather*}
    XY=\sum a_{\alpha,\beta,\gamma}Y_-^{[\alpha]}Y_0^{[\beta]}Y_+^{[\gamma]}+\sum b_{\mu,\nu}Y_-^{[\mu]}Y_+^{[\nu]}+\sum c_{\sigma,\varsigma}Y_-^{[\sigma]}Y_0^{[\varsigma]},
\end{gather*}
then 
\begin{gather*}
    X(Y\otimes \zeta)=\sum c_{\sigma,\varsigma}Y_-^{[\sigma]}Y_0^{[\varsigma]} \zeta.
\end{gather*}
We call $M(\lambda(u);\bar{\lambda}(u))$ the \textit{Verma module} over $\mathrm{U}_q\big(\widehat{\mathfrak{g}}_{\mathbf{s}}\big)$. 

It is easy to see that $M(\lambda(u);\bar{\lambda}(u))$ is a highest weight representation of $\mathrm{U}_q\big(\widehat{\mathfrak{g}}_{\mathbf{s}}\big)$ with highest weight $(\lambda(u);\bar{\lambda}(u))$. It may not be finite-dimensional. Standard classical argument implies that $\mathrm{M}(\lambda(u);\bar{\lambda}(u))$ is indecomposable and has a maximal proper submodule $Y(\lambda(u);\bar{\lambda}(u))$. Define
\begin{gather*}
    V(\lambda(u);\bar{\lambda}(u)):=M(\lambda(u);\bar{\lambda}(u))/Y(\lambda(u);\bar{\lambda}(u))
\end{gather*}
Thus, $V(\lambda(u);\bar{\lambda}(u))$ is irreducible and of type highest weight. Moreover, for given weights $\lambda(u)$ and $\bar{\lambda}(u)$, up to isomorphism, there is a unique highest weight irreducible representation $V(\lambda(u);\bar{\lambda}(u))$.

\subsection{Evaluation representation}\label{se:evairreducible representations}
The superalgebra $\mathrm{U}_q\big(\widehat{\mathfrak{g}}_{\mathbf{s}}\big)$ admits a family of simple examples of finite-dimensional representations extending those of   $\mathrm{U}_q\big(\mathfrak{g}_{\mathbf{s}}\big)$ over the same superspace. This extension relies on a superalgebraic homomorphism, which is commonly known as an \textit{evaluation homomorphism}. 

\begin{proposition}\label{eva-map}
For any $a\in\mathbb{C}\setminus\{0\}$ and $\mathbf{s}\in\mathcal{S}(m|n)$, there exists a surjective homomorphism of superalgebras $\mathsf{ev}_{a,\mathbf{s}}: \, \mathrm{U}_q\big(\widehat{\mathfrak{g}}_{\mathbf{s}}\big)\rightarrow \mathrm{U}_q\big(\mathfrak{g}_{\mathbf{s}}\big)$ such that
\begin{gather}\label{eva-hom}
    T_{\mathbf{s}}(u)\mapsto T_{\mathbf{s}}-\bar{T}_{\mathbf{s}}a^{-1}u^{-1},\qquad \bar{T}_{\mathbf{s}}(u)\mapsto \bar{T}_{\mathbf{s}}-T_{\mathbf{s}}au. 
\end{gather}
\end{proposition}
\begin{proof}
    It suffices to verify that the map $\mathsf{ev}_{a,\mathbf{s}}$ preserves relations \eqref{RTT:af1}$-$\eqref{RTT:af5}. Relations \eqref{RTT:af1} and \eqref{RTT:af2} clearly hold. For the remaining relations, we only need to check for \eqref{RTT:af5} when we substitute the right-hand side of \eqref{Rquv} and \eqref{eva-hom} for $\mathcal{R}_{q,\mathbf{s}}(u,v)$, $T_{\mathbf{s}}(u)$ and $\bar{T}_{\mathbf{s}}(u)$ as an example. Indeed, we need to check
    \begin{equation}\label{RTT:af-ev1}
        \begin{split}
            &\left(\mathcal{R}^{23}_{q,\mathbf{s}}u-\widetilde{\mathcal{R}}^{23}_{q,\mathbf{s}}v\right)\left(T^1_{\mathbf{s}}-\bar{T}^1_{\mathbf{s}}a^{-1}u^{-1}\right)\left(\bar{T}^2_{\mathbf{s}}-T^2_{\mathbf{s}}av\right) \\
        &\hspace{5em}=\left(\bar{T}^2_{\mathbf{s}}-T^2_{\mathbf{s}}av\right)\left(T^1_{\mathbf{s}}-\bar{T}^1_{\mathbf{s}}a^{-1}u^{-1}\right)\left(\mathcal{R}^{23}_{q,\mathbf{s}}u-\widetilde{\mathcal{R}}^{23}_{q,\mathbf{s}}v\right).
        \end{split}
    \end{equation}
    By using \eqref{RTT:fin-2}$-$\eqref{RTT:fin-4}, \eqref{RTT:fin-5}, and \eqref{Rq-P}, we find that \eqref{RTT:af-ev1} is equivalent to 
    \begin{gather*}
        \mathcal{P}^{23}_{\mathbf{s}}\bar{T}_{1,\mathbf{s}}T_{2,\mathbf{s}}-\mathcal{P}^{23}_{\mathbf{s}}T_{1,\mathbf{s}}\bar{T}_{2,\mathbf{s}}=T_{2,\mathbf{s}}\bar{T}_{1,\mathbf{s}}\mathcal{P}^{23}_{\mathbf{s}}-\bar{T}_{2,\mathbf{s}}T_{1,\mathbf{s}}\mathcal{P}^{23}_{\mathbf{s}},
    \end{gather*}
    which can be easily shown by $\mathcal{P}^{23}_{\mathbf{s}}\bar{T}_{1,\mathbf{s}}=\bar{T}_{2,\mathbf{s}}\mathcal{P}^{23}_{\mathbf{s}}$ and $\mathcal{P}^{23}_{\mathbf{s}}T_{1,\mathbf{s}}=T_{2,\mathbf{s}}\mathcal{P}^{23}_{\mathbf{s}}$. 
    
\end{proof}
Proposition \ref{eva-map} is a generalization of \cite[Section 2.1]{Zhf17} to arbitrary parity sequences. 
The map $\mathsf{ev}_{a,\mathbf{s}}$ serves as such an evaluation homomorphism for $\mathrm{U}_q\big(\widehat{\mathfrak{g}}_{\mathbf{s}}\big)$. 

Let $V_{\mathbf{s}}(\mathcal{M})$ for $\mathcal{M}=(\mu_1,\ldots,\mu_N)\in\big(\mathbb{C}\setminus\{0\}\big)^N$ be a finite-dimensional irreducible representations for $\mathrm{U}_q\big(\mathfrak{g}_{\mathbf{s}}\big)$ established in Section \ref{se:fdirreducible representations-fin}. 
Under pullback by the evaluation homomorphism $\mathsf{ev}_{a,\mathbf{s}}$, the representation $V_{\mathbf{s}}(\mathcal{M})$ induces a family of finite-dimensional representations over the superalgebra $\mathrm{U}_q\big(\widehat{\mathfrak{g}}_{\mathbf{s}}\big)$. 
Consequently, these representations are highest weight representations of $\mathrm{U}_q\big(\widehat{\mathfrak{g}}_{\mathbf{s}}\big)$ with highest weights $(\mu(u);\bar{\mu}(u))$ given by
\begin{equation*}
    \begin{aligned}
        &\mu_i(u)=\mu_i^{-1}-\mu_ia^{-1}u^{-1}, \\
        &\mu(u)=(\mu_1(u),\mu_2(u),\ldots,\mu_N(u)),
    \end{aligned}\qquad 
    \begin{aligned}
        &\bar{\mu}_i(u)=\mu_i-\mu_i^{-1}au, \\
        &\bar{\mu}(u)=(\bar{\mu}_1(u),\bar{\mu}_2(u),\ldots,\bar{\mu}_N(u)). 
    \end{aligned}
\end{equation*}
We denote them by $V_{a,\mathbf{s}}(\mathcal{M})$ for each $a$; these $V_{a,\mathbf{s}}(\mathcal{M})$ are called the \textit{evaluation representation} of $\mathrm{U}_q\big(\widehat{\mathfrak{g}}_{\mathbf{s}}\big)$, and each $V_{a,\mathbf{s}}(\mathcal{M})$ is an irreducible representation over $\mathrm{U}_q\big(\widehat{\mathfrak{g}}_{\mathbf{s}}\big)$. We call $V_{a,\mathbf{s}}(\mathcal{M})$ typical (resp. atypical) if it is a typical (resp. atypical) irreducible representations of $\mathrm{U}_q\big(\mathfrak{g}_{\mathbf{s}}\big)$. 
Condition \eqref{fd:cd:fin1} implies that the formal series $\mu_i(u)$, $\bar{\mu}_i(u)$ satisfy the radios for $|i|+|j|=\bar{0}$, 
\begin{gather}\label{af:fd_eva}
    \frac{\mu_i(u)}{\mu_j(u)}=q_i^{l_{ij}+\#_{(i,j)}}\frac{P_{ij}(q_i^{-2}u)}{P_{ij}(u)}=\frac{\bar{\mu}_i(u)}{\bar{\mu}_j(u)},
\end{gather}
where
\begin{gather*}
    P_{ij}(u)=\left(\mu_j-\mu_j^{-1}au\right)\left(\mu_j-q_i^{-2}\mu_j^{-1}au\right)\cdots \left(\mu_j-q_i^{-2(l_{ij}+\#_{(i,j)}-1)}\mu_j^{-1}au\right).
\end{gather*}

The evaluation representations defined above is essential for establishing our main result in Section \ref{se:fdirreducible representations:af}.

\vspace{1em}
\section{Finite-dimensional irreducible representations of $\mathrm{U}_q\big(\widehat{\mathfrak{g}}_{\mathbf{s}}\big)$}
\label{se:fdirreducible representations:af}
In order to classify the finite-dimensional irreducible representations of $\mathrm{U}_q\big(\widehat{\mathfrak{g}}_{\mathbf{s}}\big)$, we first determine necessary and sufficient conditions for finite-dimensionality of its highest weight irreducible representations $V(\lambda(u);\bar{\lambda}(u))$  using Proposition \ref{fd_irreducible representations:af}. 
As shown in Section~\ref{se:evairreducible representations}, there exists a family of nontrivial finite-dimensional representations of $\mathrm{U}_q\big(\widehat{\mathfrak{g}}_{\mathbf{s}}\big)$ that satisfy conditions \eqref{af:fd_eva} with the exception of $i=m$. We therefore begin by analyzing the special case where $m=n=1$.

%\subsection{Finite-dimensional irreducible representations of $\mathrm{U}_q\big(\widehat{\mathfrak{gl}}_{1|1,\mathbf{s}}\big)$}
\subsection{Conditions for finite-dimensionality of $\mathrm{U}_q\big(\widehat{\mathfrak{gl}}_{1|1,\mathbf{s}}\big)$}\label{se:equivcondition:fd1}
\begin{theorem}\label{rep:T1}
Given $\mathbf{s}\in\mathcal{S}(1|1)=\{01,10\}$. Consider the $2$-tuples 
\begin{gather*}
    \lambda(u)=(\lambda_1(u),\lambda_2(u)),\quad \bar{\lambda}(u)=(\bar{\lambda}_1(u),\bar{\lambda}_2(u))
\end{gather*}
for each series 
$\lambda_i(u),\bar{\lambda}_i(u)$ satisfying \eqref{formal_lambda} and \eqref{formal_lambda2}. 
The following conditions for the irreducible highest weight representation $V_{\mathbf{s}}(\lambda(u);\bar{\lambda}(u))$ of $\mathrm{U}_q\big(\widehat{\mathfrak{gl}}_{1|1,\mathbf{s}}\big)$ are equivalent\,{\rm :}
\begin{itemize}
  \item[{\rm (1)}] $\dim V_{\mathbf{s}}(\lambda(u);\bar{\lambda}(u))<\infty$\,{\rm ;}
  \item[{\rm (2)}] there exist polynomials $Q(u),\widetilde{Q}(u)\in \mathbb{C}[u]$ of degree $K$ together with the products of the leading coefficient and the constant term equal to 1, 
  %of degree $\leqslant K$, both possessing nonzero constant terms,  
  such that
    \begin{gather}\label{fd:cd-af:11}
    \frac{\lambda_1(u)}{\lambda_2(u)}=\frac{Q(u)}{\widetilde{Q}(u)}=\frac{\bar{\lambda}_1(u)}{\bar{\lambda}_2(u)}.
    \end{gather}
\end{itemize}
\end{theorem}

\begin{proof}
Let $\dim V_{\mathbf{s}}(\lambda_1(u),\lambda_2(u);\bar{\lambda}_1(u),\bar{\lambda}_2(u))<\infty$. Twisting by \eqref{auto:af:fu}, we may set $\lambda_2(u)=\bar{\lambda}_2(u)=1$ without loss of generality. Let $\zeta$ be the maximal vector of $V_{\mathbf{s}}(\lambda_1(u),1;\bar{\lambda}_1(u),1)$, and $W$ the $\mathbb{Z}_2$-graded subspace of $V_{\mathbf{s}}(\lambda_1(u),1;\bar{\lambda}_1(u),1)$ spanned by all vectors $t_{21,\mathbf{s}}^{(b)}\zeta$, $\bar{t}_{21,\mathbf{s}}^{(c)}\zeta$. Then $W$ must be finite-dimensional. It follows that there exists some sufficiently large integers $k,l$ such that
\begin{gather}\label{fd:af-eq:1}
\sum_{b=0}^{l} \tau_b t_{21,\mathbf{s}}^{(b)}\zeta+\sum_{c=1}^{k} \sigma_c \bar{t}_{21,\mathbf{s}}^{(c)}\zeta=0, \qquad \text{for}~~\sigma_k,\ \tau_l\neq 0.
\end{gather}

By the defining relation \eqref{RTT:af5-ex}, we have
\begin{gather}\label{fd:af-eq:2}
t_{12,\mathbf{s}}(u)\bar{t}_{21,\mathbf{s}}(v)+\bar{t}_{21,\mathbf{s}}(v)t_{12,\mathbf{s}}(u)=(q-q^{-1})\frac{v}{v-u}\left(t_{22,\mathbf{s}}(u)\bar{t}_{11,\mathbf{s}}(v)-\bar{t}_{22,\mathbf{s}}(v)t_{11,\mathbf{s}}(u)\right).
\end{gather}
Divide both sides by $(v-u)$ to get
\begin{gather*}
t_{12,\mathbf{s}}^{(p)}\bar{t}_{21,\mathbf{s}}^{(c)}+\bar{t}_{21,\mathbf{s}}^{(c)}t_{12,\mathbf{s}}^{(p)}=-\left(q-q^{-1}\right)\sum_{r=1}^{\min\{p,c\}}
\left(t_{22,\mathbf{s}}^{(p-r)}\bar{t}_{11,\mathbf{s}}^{(c-r)}-\bar{t}_{22,\mathbf{s}}^{(c-r)}t_{11,\mathbf{s}}^{(p-r)}\right)
\end{gather*}
for all $p,c\geqslant 1$, owing to the formal power series
\begin{gather*}
\frac{v}{v-u}=-\frac{v}{u}\cdot \frac{1}{1-u^{-1}v}=-\sum_{r=1}^{\infty}u^{-r}v^r.
\end{gather*}
Similarly, we also have
\begin{gather*}
t_{12,\mathbf{s}}^{(p)}t_{21,\mathbf{s}}^{(b)}+t_{21,\mathbf{s}}^{(b)}t_{12,\mathbf{s}}^{(p)}=-\left(q-q^{-1}\right)\sum_{r=1}^{p}
\left(t_{22,\mathbf{s}}^{(p-r)}t_{11,\mathbf{s}}^{(b+r)}-t_{22,\mathbf{s}}^{(b+r)}t_{11,\mathbf{s}}^{(p-r)}\right)
\end{gather*}
for all $p\geqslant 1,\ b\geqslant 0$. As we have set $\lambda_2(u)=\bar{\lambda}_2(u)=1$, it follows that
\begin{align*}
t_{12,\mathbf{s}}^{(p)}\bar{t}_{21,\mathbf{s}}^{(c)}\zeta&=\left(q-q^{-1}\right)\left(\delta_{a\leqslant p}\lambda_1^{(p-a)}-\delta_{a\geqslant p}\bar{\lambda}_1^{(a-p)}\right)\zeta, \\
t_{12,\mathbf{s}}^{(p)}t_{21,\mathbf{s}}^{(b)}\zeta&=-\left(q-q^{-1}\right)\lambda_1^{(b+p)}\zeta.
\end{align*}
Applying $t_{12,\mathbf{s}}^{(p)}$ for $p\geqslant 1$ to \eqref{fd:af-eq:1}, one immediately gets
\begin{equation*}
\sum_{c=1}^{k} \sigma_c\left(\delta_{c\leqslant p}\lambda_1^{(p-c)}-\delta_{c\geqslant p}\bar{\lambda}_1^{(c-p)}\right)=\sum_{b=0}^{l} \tau_b \lambda_1^{(b+p)}.
\end{equation*}
Summing over all $p\geqslant 1$, we obtain
\begin{equation*}
\begin{split}
\sum_{b=0}^{l}\sum_{p=1}^{\infty} \tau_b \lambda_1^{(b+p)}u^{-p}-\sum_{c=1}^{k}\sum_{p=c}^{\infty} \sigma_c\lambda_1^{(p-c)}u^{-p}&=
-\sum_{c=1}^k\sum_{p=1}^c\sigma_c\bar{\lambda}_1^{(c-p)}u^{-p} \\
\sum_{b=0}^{l}\tau_bu^b\sum_{p=b+1}^{\infty}\lambda_1^{(p)}u^{-p}-\sum_{c=1}^{k}\sigma_c u^{-c}\sum_{p=0}^{k}\lambda_1^{(p)}u^{-p}
&=-\sum_{p=0}^{k-1}\bar{\lambda}_1^{(p)}u^p\sum_{c=p+1}^{k}\sigma_cu^{-c} \\
\lambda_1(u)\left(\sum_{b=0}^{l}\tau_bu^b-\sum_{c=1}^{k}\sigma_c u^{-c}\right)&=\sum_{p=0}^{l}\lambda_1^{(p)}u^{-p}\sum_{b=p}^{l}\tau_bu^b
-\sum_{p=0}^{k-1}\bar{\lambda}_1^{(p)}u^p\sum_{c=p+1}^{k}\sigma_cu^{-c}
\end{split}
\end{equation*}
Set
\begin{align*}
    Q(u)&= \left(\sum_{p=0}^{l}\lambda_1^{(p)}u^{-p}\sum_{b=p}^{l}\tau_bu^b
-\sum_{p=0}^{k-1}\bar{\lambda}_1^{(p)}u^p\sum_{c=p+1}^{k}\sigma_cu^{-c}\right)u^k, \\
  \widetilde{Q}(u)&=\left(\sum_{b=0}^{l}\tau_bu^b-\sum_{c=1}^{k}\sigma_c u^{-c}\right)u^k.
\end{align*}
This forces 
\begin{gather*}
    \lambda_1(u)=\frac{Q(u)}{\widetilde{Q}(u)}.
\end{gather*}
The molecular and denominator parts of the ratio are both polynomials of degree $k+l$ such that the products of the leading coefficient and the constant term are both equal to $\sigma_k\tau_l\neq 0$, 
satisfying the first equation of \eqref{fd:cd-af:11}. The second equation follows from the action of $\bar{t}_{12,\mathbf{s}}^{(p)}$ for $p\geqslant 1$ on \eqref{fd:af-eq:1}.

Conversely, let $Q(u),\widetilde{Q}(u)$ be polynomials
\begin{align*}
    Q(u)&=Q_0+Q_1u+\cdots +Q_Ku^K\  \in\ \mathbb{C}[u], \\
    \widetilde{Q}(u)&=\widetilde{Q}_0+\widetilde{Q}_1u+\cdots +\widetilde{Q}_Ku^K\  \in\ \mathbb{C}[u]
\end{align*}
such that $Q_0Q_L=\widetilde{Q}_0\widetilde{Q}_K=1$, 
and let $\lambda(u)=(\lambda_1(u),\lambda_2(u))$, $\bar{\lambda}(u)=(\bar{\lambda}_1(u),\bar{\lambda}_2(u))$ satisfy the equations given in \eqref{fd:cd-af:11}. For generality, we may assume that $Q(u)$ and $\widetilde{Q}(u)$ do not have common factors. 

Let $\zeta$ be the maximal vector of $V_{\mathbf{s}}(\lambda(u);\bar{\lambda}(u))$. Applying both sides of \eqref{fd:af-eq:2} to $\zeta$, we have
\begin{gather}\label{fd:af-eq:3}
t_{12,\mathbf{s}}(u)\bar{t}_{21,\mathbf{s}}(v)\xi=(q-q^{-1})\frac{v}{v-u}\left(\lambda_2(u)\bar{\lambda}_1(v)-\bar{\lambda}_2(v)\lambda_1(u)\right)\xi.
\end{gather}
Using the isomorphism \eqref{auto:af:fu} for 
\begin{gather*}
f(u)=\frac{\widetilde{Q}_0\lambda_2(u)}{\widetilde{Q}(u)},\qquad g(u)=\frac{\widetilde{Q}_0\bar{\lambda}_2(u)}{\widetilde{Q}(u)},
\end{gather*}
it implies that \eqref{fd:af-eq:3} is equivalent to
\begin{align*}
t_{12,\mathbf{s}}(u)\bar{t}_{21,\mathbf{s}}(v)\zeta&=\left(q-q^{-1}\right)(Q'_0)^{-2}\frac{v}{v-u}\left(\widetilde{Q}(u)Q(v)-Q(u)\widetilde{Q}(v)\right)\zeta \\
&=\left(q-q^{-1}\right)(\widetilde{Q}_0)^{-2}\frac{v}{v-u}\sum_{r,s=0}^{K}\widetilde{Q}_rQ_s\left(u^rv^s-u^sv^r\right)\zeta, \\
&=\left(q-q^{-1}\right)(\widetilde{Q}_0)^{-2}\sum_{r=1}^{K}h_r(u)v^{r}\zeta
\end{align*}
for a family of polynomials $h_r(u)\in\mathbb{C}[u]$.
It follows that
\begin{gather*}
\bar{t}_{21,\mathbf{s}}^{(p)}\zeta=0\quad\text{for }~~p>K.
\end{gather*}
Similarly, by \eqref{RTT:af6-ex}, we have
\begin{gather*}
t_{21,\mathbf{s}}^{(p)}\zeta=0\quad\text{for }~~p>K.
\end{gather*}
Hence, the representation $V_{\mathbf{s}}(\lambda(u);\bar{\lambda}(u))$ is finite-dimensional. 

\end{proof}

In Theorem \ref{rep:T1}, we have the following decompositions:
\begin{gather*}
    Q(u)=\epsilon_1(\eta_1+\eta_1^{-1}u)\cdots (\eta_N+\eta_N^{-1}u),\qquad 
    \widetilde{Q}(u)=\epsilon_2(\widetilde{\eta}_1+\widetilde{\eta}_1^{-1}u)\cdots (\widetilde{\eta}_N+\widetilde{\eta}_N^{-1}u)
\end{gather*}
for some nonzero complex numbers $\eta_i$, $\widetilde{\eta}_i$ and $\epsilon_i\in\{\pm 1\}$. 

%From Theorem \ref{rep:T1}'s proof, we deduce that if $Q(u)$, $\widetilde{Q}(u)$ are both degree-$K$ polynomials,
%\begin{gather*}
%    \dim V_{\mathbf{s}}(\lambda(u);\bar{\lambda}(u))\leqslant 2^{2K+1}.
%\end{gather*}

\subsection{Finite-dimensional irreducible representations of $\mathrm{U}_q\big(\widehat{\mathfrak{gl}}_{2|0}\big)$ and $\mathrm{U}_q\big(\widehat{\mathfrak{gl}}_{0|2}\big)$}\label{se:equivcondition:fd2}
Recall the algebra $\mathrm{U}_q\big(\widehat{\mathfrak{gl}}_N\big)$ defined in \cite[Section 3]{MRS03} (see also \cite[Section 2.3]{GM10}),
denote the generator series of $\mathrm{U}_{q^{-1}}\big(\widehat{\mathfrak{gl}}_2\big)$ by $t^{\star}_{ij}(u)$, $\bar{t}^{\star}_{ij}(u)$ for $1\leqslant i,j\leqslant 2$. In accordance with Remark \ref{usual0}, it is easy to check that we have the following isomorphisms
\begin{gather}\label{iso:af1}
\mathrm{U}_{q^{-1}}\big(\widehat{\mathfrak{gl}}_2\big)\rightarrow \mathrm{U}_q\big(\widehat{\mathfrak{gl}}_{2|0}\big)\qquad t_{ij}^{\star}(u)\mapsto t_{ij}(u),\quad \bar{t}_{ij}^{\star}(u)\mapsto \bar{t}_{ij}(u), 
\end{gather}
and 
\begin{gather}\label{iso:af2}
\mathrm{U}_{q^{-1}}\big(\widehat{\mathfrak{gl}}_2\big)\rightarrow \mathrm{U}_q\big(\widehat{\mathfrak{gl}}_{0|2}\big)\qquad t_{ij}^{\star}(u)\mapsto \bar{t}_{3-i,3-j}(u^{-1}),\quad \bar{t}_{ij}^{\star}(u)\mapsto t_{3-i,3-j}(u^{-1}).
\end{gather}
Theorem \ref{rep:T2} follows the isomorphisms \eqref{iso:af1}$-$\eqref{iso:af2}, and \cite[Theorem 3.6]{GM10}. 

\begin{theorem}\label{rep:T2}
Consider the $2$-tuples 
\begin{gather*}
    \lambda(u)=(\lambda_1(u),\lambda_2(u)),\quad \bar{\lambda}(u)=(\bar{\lambda}_1(u),\bar{\lambda}_2(u))
\end{gather*}
for each series 
$\lambda_i(u),\bar{\lambda}_i(u)$ satisfying \eqref{formal_lambda} and \eqref{formal_lambda2}. 
\begin{enumerate}
    \item[{\rm (1)}] The following conditions for the irreducible highest weight representation $V_{00}(\lambda(u);\bar{\lambda}(u))$ of $\mathrm{U}_q\big(\widehat{\mathfrak{gl}}_{2|0}\big)$ are equivalent\,{\rm :}
        \begin{enumerate}
            \item[{\rm (\romannumeral1)}] $\dim V_{00}(\lambda(u);\bar{\lambda}(u))<\infty$\,{\rm ;}
            \item[{\rm (\romannumeral2)}] there exists a polynomial $P(u)\in 1+u\mathbb{C}[u]$ such that
            \begin{gather*}
            \frac{\epsilon_1\lambda_1(u)}{\epsilon_2\lambda_2(u)}=q^{\deg P(u)}\cdot\frac{P(q^{-2}u)}{P(u)}=\frac{\epsilon_1\bar{\lambda}_1(u)}{\epsilon_2\bar{\lambda}_2(u)}
            \end{gather*}
            for some $\epsilon_1,\epsilon_2\in\{\pm 1\}$. The polynomial $P(u)$ is uniquely determined up to $\{\pm 1\}$.
        \end{enumerate}
    \item[{\rm (2)}] The following conditions for the irreducible highest weight representation $V_{11}(\lambda(u);\bar{\lambda}(u))$ of $\mathrm{U}_q\big(\widehat{\mathfrak{gl}}_{0|2}\big)$ are equivalent\,{\rm :}
        \begin{enumerate}
            \item[{\rm (\romannumeral3)}] $\dim V_{11}(\lambda(u);\bar{\lambda}(u))<\infty$\,{\rm ;}
            \item[{\rm (\romannumeral4)}] there exist a polynomial $P(u)\in 1+u\mathbb{C}[u]$ such that
            \begin{gather*}
            \frac{\epsilon_1\lambda_1(u)}{\epsilon_2\lambda_2(u)}=q^{-\deg P(u)}\cdot\frac{P(q^2u)}{P(u)}=\frac{\epsilon_1\bar{\lambda}_1(u)}{\epsilon_2\bar{\lambda}_2(u)}
            \end{gather*}
            for some $\epsilon_1,\epsilon_2\in\{\pm 1\}$. The polynomial $P(u)$ is uniquely determined up to $\{\pm 1\}$. 
        \end{enumerate}
\end{enumerate}
\end{theorem}

\subsection{Finite-dimensional irreducible representations in general cases}\label{se:fdirreducible representations:af-general}
Within the framework of Section \ref{se:equivcondition:fd1} and \ref{se:equivcondition:fd2}, we are ready to classify the finite-dimensional irreducible representations of the superalgebra $\mathrm{U}_q\big(\widehat{\mathfrak{g}}_{\mathbf{s}}\big)$ for general pairs $(m,n)$. For any pair $(k,l)\in I_{\mathbf{s}}\times I_{\mathbf{s}}$ with $k\neq l$, let $\left[\mathrm{U}_q\big(\widehat{\mathfrak{g}}_{\mathbf{s}}\big)\right]_{k,l}$ be the sub-superalgebra of $\mathrm{U}_q\big(\widehat{\mathfrak{g}}_{\mathbf{s}}\big)$ generated by $t_{kl,\mathbf{s}}^{(r)}$, $\bar{t}_{kl,\mathbf{s}}^{(r)}$, $r\in\mathbb{Z}_+$. 
Depending on the values of $s_ks_l$, there are four distinct isomorphisms onto the superalgebra $\left[\mathrm{U}_q\big(\widehat{\mathfrak{g}}_{\mathbf{s}}\big)\right]_{k,l}$, as presented below: ($1\leqslant i,j\leqslant 2$)
\begin{enumerate}
    \item[(1)] If $s_ks_l=01$, the isomorphism is given by
    \begin{gather}\label{isom:af-kl:1}
        \mathrm{U}_q\big(\widehat{\mathfrak{gl}}_{1|1,01}\big)\rightarrow \left[\mathrm{U}_q\big(\widehat{\mathfrak{g}}_{\mathbf{s}}\big)\right]_{k,l}\qquad \gamma_{ij,01}(u)\mapsto \gamma_{o_{kl}(i),o_{kl}(j),\mathbf{s}}(u). 
    \end{gather}
    \item[(2)] If $s_ks_l=10$, the isomorphism is given by
    \begin{gather}\label{isom:af-kl:2}
        \mathrm{U}_q\big(\widehat{\mathfrak{gl}}_{1|1,10}\big)\rightarrow \left[\mathrm{U}_q\big(\widehat{\mathfrak{g}}_{\mathbf{s}}\big)\right]_{k,l}\qquad \gamma_{ij,10}(u)\mapsto \gamma_{o_{kl}(i),o_{kl}(j),\mathbf{s}}(u). 
    \end{gather}
    \item[(3)] If $s_ks_l=00$, the isomorphism is given by
    \begin{gather}\label{isom:af-kl:3}
        \mathrm{U}_q\big(\widehat{\mathfrak{gl}}_{2|0}\big)\rightarrow \left[\mathrm{U}_q\big(\widehat{\mathfrak{g}}_{\mathbf{s}}\big)\right]_{k,l}\qquad \gamma_{ij,00}(u)\mapsto \gamma_{o_{kl}(i),o_{kl}(j),\mathbf{s}}(u). 
    \end{gather}
    \item[(4)] If $s_is_j=11$, the isomorphism is given by
    \begin{gather}\label{isom:af-kl:4}
        \mathrm{U}_q\big(\widehat{\mathfrak{gl}}_{0|2}\big)\rightarrow \left[\mathrm{U}_q\big(\widehat{\mathfrak{g}}_{\mathbf{s}}\big)\right]_{k,l}\qquad \gamma_{ij,11}(u)\mapsto \gamma_{o_{kl}(i),o_{kl}(j),\mathbf{s}}(u). 
    \end{gather}
\end{enumerate}
Here, $\gamma\in\{\,t,\bar{t}\,\}$ and $o_{kl}: \{1,2\}\rightarrow \{k,l\}$ is the mapping such that $o_{kl}(1)=k$, $o_{kl}(2)=l$.

\begin{theorem}\label{rep:T3}
Consider the $N$-tuples 
\begin{gather*}
    \lambda(u)=(\lambda_1(u),\lambda_2(u),\ldots,\lambda_N(u)),\quad \bar{\lambda}(u)=(\bar{\lambda}_1(u),\bar{\lambda}_2(u),\ldots,\bar{\lambda}_N(u))
\end{gather*}
for each series 
$\lambda_i(u),\bar{\lambda}_i(u)$ satisfying \eqref{formal_lambda} and \eqref{formal_lambda2}. The following conditions for the irreducible highest weight representation $V_{\mathbf{s}}(\lambda(u);\bar{\lambda}(u))$ of $\mathrm{U}_q\big(\widehat{\mathfrak{g}}_{\mathbf{s}}\big)$ are equivalent\,{\rm :}
\begin{itemize}
  \item[{\rm (1)}]\ $\dim V_{\mathbf{s}}(\lambda(u);\bar{\lambda}(u))<\infty$\,{\rm ;}
  \item[{\rm (2)}]\ there exist a series of polynomials $P_{ij}(u)\in 1+u\mathbb{C}[u]$ {\rm (}$1\leqslant i<j\leqslant N$, $|i|+|j|=\bar{0}${\rm)}, and $Q_{bc}(u),\widetilde{Q}_{bc}(u)$ {\rm (}$1\leqslant b<c\leqslant N$, $|b|+|c|=\bar{1}${\rm)} with the products of the constant term and the leading coefficient equal to 1, such that
    \begin{gather}\label{rep:mn-af:1}
    \frac{\epsilon_i\lambda_i(u)}{\epsilon_j\lambda_j(u)}=q_i^{\deg P_{ij}(u)}\cdot\frac{P_{ij}(q_i^{-2}u)}{P_{ij}(u)}=\frac{\epsilon_i\bar{\lambda}_i(u)}{\epsilon_j\bar{\lambda}_j(u)}
    \end{gather}
    for some $\epsilon_i,\epsilon_j\in\{\pm 1\}$, and
    \begin{gather}\label{rep:mn-af:2}
    \frac{\lambda_b(u)}{\lambda_c(u)}=\frac{Q_{bc}(u)}{\widetilde{Q}_{bc}(u)}=\frac{\bar{\lambda}_b(u)}{\bar{\lambda}_c(u)}.
    \end{gather}
    Here, 
    \begin{align*}
        P_{ij}(u)&=P_{ih}(u)P_{hj}(u)&&\text{for}~~ |h|=|i|=|j|,~~i<h<j, \quad\\
        Q_{bc}(u)&=Q_{bh}(u)Q_{hc}(u)  &&\text{for}~~ |h|=|b|=|c|,~~b<h<c,\quad \\
        \widetilde{Q}_{bc}(u)&=\widetilde{Q}_{bh}(u)\widetilde{Q}_{hc}(u) &&\text{for}~~ |h|=|b|=|c|,~~b<h<c,\quad
    \end{align*}
    where the decompositions of $P_{ij}(u),Q_{bc}(u),\widetilde{Q}_{bc}(u)$ are independent of the choice of $h$. 
\end{itemize}
\end{theorem}

\begin{proof}
Suppose that $\dim V_{\mathbf{s}}(\lambda(u);\bar{\lambda}(u))<\infty$, and let $\zeta$ be the maximal vector of the representation $V_{\mathbf{s}}(\lambda(u);\bar{\lambda}(u))$. Through the isomorphisms \eqref{isom:af-kl:1}--\eqref{isom:af-kl:4},
 the cyclic span $\left[\mathrm{U}_q\big(\widehat{\mathfrak{g}}_{\mathbf{s}}\big)\right]_{k,l}\zeta$ can be respective considered as the finite-dimensional Verma module of $\mathrm{U}_q\big(\widehat{\mathfrak{gl}}_{1|1,01}\big)$, $\mathrm{U}_q\big(\widehat{\mathfrak{gl}}_{1|1,10}\big)$, $\mathrm{U}_q\big(\widehat{\mathfrak{gl}}_{2|0}\big)$ and $\mathrm{U}_q\big(\widehat{\mathfrak{gl}}_{0|2}\big)$ with a different value of $s_ks_l$. Examining its irreducible quotient, the conditions \eqref{rep:mn-af:1} and \eqref{rep:mn-af:2} hold owing to Proposition \ref{rep:T1}--\ref{rep:T2}.

Now, let the conditions \eqref{rep:mn-af:1} and \eqref{rep:mn-af:2} hold for the representation $V_{\mathbf{s}}(\lambda(u);\bar{\lambda}(u))$. For convenience, we take the $N$-tuple $(\epsilon_1,\ldots,\epsilon_{N})=(1,\ldots,1)$. The statements in Section \ref{se:evairreducible representations} imply that such a finite-dimensional representation $U^0_{\mathbf{s}}=V_{\mathbf{s}}(\mathcal{M})$ exists for all $\mathcal{M}$ that satisfies the conditions \eqref{fd:cd:fin1}, \eqref{fd:cd:fin2} (resp. \eqref{fd:cd:fin3}). 

Let $\mathbf{t}$ be a subsequence of $\mathbf{s}$ (at least length 2). Denote $\mathfrak{g}^{\flat}_{\mathbf{t}}$ by the sub-superalgebra of $\mathfrak{g}_{\mathbf{s}}$ corresponding to $\mathbf{t}$. Similarly, we set $I^{\flat}_{\mathbf{t}}\subset I_{\mathbf{s}}$. We use the notation $V_{\mathbf{t}}$ to denote the restriction of the highest weight  $\mathrm{U}_q\big(\widehat{\mathfrak{g}}_{\mathbf{s}}\big)$-module $V$ with trivial action of $t_{ii,\mathbf{s}}(u),\bar{t}_{ii,\mathbf{s}}(u)$ for $i\in I_{\mathbf{s}}\setminus I^{\flat}_{\mathbf{t}}$ to $\mathrm{U}_q\big(\widehat{\mathfrak{g}}^{\flat}_{\mathbf{t}}\big)$. 
Then the irreducibility of $V_{\mathbf{t}}$ implies that $V$ is also irreducible. 
For each $\mathbf{t}$, we initiate our discussion from the trivial representation of $\mathrm{U}_q\big(\widehat{\mathfrak{g}}^{\,\flat}_{\mathbf{t}}\big)$,
thereby ensuring its irreducibility and finite-dimensionality. 

Assume that the polynomials $P_{ij}(u),Q_{bc}(u),\widetilde{Q}_{bc}(u)$ satisfy the conditions \eqref{rep:mn-af:1} and \eqref{rep:mn-af:2}, and the associated irreducible representation $W_{\mathbf{t}}=V_{\mathbf{t}}(\lambda(u);\bar{\lambda}(u))$ is finite-dimensional. The comutiplication $\widehat{\triangle}_{\mathbf{t}}$ defined in Section \ref{se:RTTqafsuperalgebra} ensures that $\mathrm{U}_q\big(\widehat{\mathfrak{g}}^{\,\flat}_{\mathbf{t}}\big)$ acts on the tensor produce $W_{\mathbf{t}}^{\circ}=U^0_{\mathbf{t}}\otimes W_{\mathbf{t}}$ as a representation. Let $\xi_0$ and $\zeta$ be the maximal vectors of $U^0_{\mathbf{t}}$ and $W_{\mathbf{t}}$, respectively. Observe that the cyclic span $\mathrm{U}_q\big(\widehat{\mathfrak{g}}^{\,\flat}_{\mathbf{t}}\big)(\xi_0\otimes \zeta)$ is a finite-dimensional highest weight representation with highest weights
\begin{gather*}
(\lambda_1^{\circ}(u),\ldots,\lambda_{N}^{\circ}(u);\bar{\lambda}_1^{\circ}(u),\ldots,\bar{\lambda}_{N}^{\circ}(u))
\end{gather*}
such that
\begin{gather*}
\lambda_i^{\circ}(u)=\left(\mu_i^{-1}-\mu_i au^{-1}\right)\lambda_i(z),\quad \bar{\lambda}_i^{\circ}(z)=\left(\mu_i-\mu_i^{-1}au\right)\bar{\lambda}_i(z).
\end{gather*}
Then we have
\begin{align*}
\frac{\lambda_i^{\circ}(z)}{\lambda_j^{\circ}(z)}
&=q_i^{l_{ij}+\#_{(i,j)}+\deg P_{ij}(u)}\frac{\left(\mu_j-q_i^{-2}\mu_j^{-1}au\right)\cdots \left(\mu_j-q_i^{-2(l_{ij}+\#_{(i,j)})}\mu_j^{-1}au\right)}{\left(\mu_j-\mu_j^{-1}au\right)\cdots \left(\mu_j-q_i^{-2(l_{ij}+\#_{(i,j)}-1)}\mu_j^{-1}au\right)}\cdot\frac{P_i(q^{-2}u)}{P_{i}(u)} \\
&=\frac{\bar{\lambda}_i^{\circ}(u)}{\bar{\lambda}_{i+1}^{\circ}(u)}
\end{align*}
for $|i|+|j|=\bar{0}$, $i<j$; and
\begin{gather*}
\frac{\lambda_b^{\circ}(u)}{\lambda_c^{\circ}(u)}=\frac{\mu_b-\mu_b^{-1}au}{\mu_c-\mu_c^{-1}au}\cdot \frac{Q_{bc}(u)}{\widetilde{Q}_{bc}(u)}=\frac{\bar{\lambda}_b^{\circ}(u)}{\bar{\lambda}_c^{\circ}(u)}
\end{gather*}
for $|b|+|c|=\bar{1}$, $b<c$. 
This representation admit a unique irreducible quotient related to the new polynomials
\begin{align*}
P^{\circ}_{ij}(u)&=\left(1-\mu_j^{-2}au\right)\cdots \left(1-q_i^{-2(l_{ij}+\#_{(i,j)}-1)}\mu_j^{-2}au\right)P_{ij}(u), \\
Q^{\circ}_{bc}(u)&=\left(\mu_b a^{-\frac{1}{2}}-\mu_b^{-1}a^{\frac{1}{2}}u\right)Q_{bc}(u),\qquad  
\widetilde{Q}^{\circ}_{bc}(u)=\left(\mu_ca^{-\frac{1}{2}}-\mu_c^{-1}a^{\frac{1}{2}}u\right)\widetilde{Q}_{bc}(u),
\end{align*}
which also obviously satisfies the conditions \eqref{rep:mn-af:1} and \eqref{rep:mn-af:2}. 

Our construction above arounds all possible representation $W_{\mathbf{t}}^{\circ}$ subject to the polynomials 
$P_{ij}^{\circ}(u),Q_{ij}^{\circ}(u),\widetilde{Q}_{ij}^{\circ}(u)$ for the subsequence $\mathbf{t}$, when we adjust the choose of $\mathcal{M}$ in $U_{\mathbf{t}}^0$. 
By induction on the length of $\mathbf{t}$, we conclude that all $W_{\mathbf{s}}^{\circ}$ is finite-dimensional.

\end{proof}

\begin{remark}
    For a given $N$-tuple $(\epsilon_1,\ldots,\epsilon_{N})$, there exists a unique set of polynomials $P_{ij}(u)$ 
that satisfy the condition \eqref{rep:mn-af:1}. In contrast, this particular set of polynomials corresponds uniquely to a $N$-tuple $(\epsilon_1,\ldots,\epsilon_N)$. In particular, this specified condition remains valid if we simultaneously alter the signs of all $\epsilon_i$.
\end{remark}
\begin{remark}
    If we only consider pairs of coprime polynomials $(Q_{bc}(u),\widetilde{Q}_{bc}(u))$ for each $b,c$ in \eqref{rep:mn-af:2}, they are unique up to a factor $\pm 1$. 
\end{remark}

Within the framework of the proof of Theorem \ref{rep:T3}, we immediately have
\begin{corollary}\label{tensor-eva-rep}
    Every finite-dimensional irreducible representation of $\mathrm{U}_q\big(\widehat{\mathfrak{g}}_{\mathbf{s}}\big)$ is isomorphic to a subquotient of a tensor product of evaluation representations. 
\end{corollary}

\subsection{Tensor product of evaluation representations for $\mathrm{U}_q\big(\widehat{\mathfrak{gl}}_{1|1,\mathbf{s}}\big)$}

The final subsection of this paper  provides a more precise result of Corollary \ref{tensor-eva-rep} for the special case $m=n=1$. Our analysis relies essentially on the $q$-super Yangian $\mathrm{Y}_q\big(\mathfrak{gl}_{1|1,\mathbf{s}}\big)$ introduced in Section \ref{se:qsuperYangian}. 

Define the irreducible highest weight module $\check{V}_{\mathbf{s}}(\bar{\lambda}(u))$ of the $q$-super Yangian $\mathrm{Y}_q\big(\mathfrak{gl}_{1|1,\mathbf{s}}\big)$ obtained by  restricting  $V_{\mathbf{s}}(\lambda(u);\bar{\lambda}(u))$. According to the "if" part of the proof of Theorem \ref{rep:T1} with the essential modification of replacing $t_{21,\mathbf{s}}(u)$ by $\bar{t}_{21,\mathbf{s}}(u)$ in equation \eqref{fd:af-eq:3}, we deduce that if the second equation in \eqref{fd:cd-af:11} is satisfied, then
\begin{gather*}
    \bar{t}_{21}^{(p)}\zeta=0,~~p>K,\quad \text{and}\quad \dim \check{V}_{\mathbf{s}}(\bar{\lambda}(u))\leqslant 2^K.
\end{gather*}

Comparing with the embedding \eqref{qSY-embed} and Proposition \ref{eva-map}, $\mathrm{Y}_q\big(\mathfrak{gl}_{1|1,\mathbf{s}}\big)$ inherits the evaluation homomorphism $\mathsf{ev}_{a,\mathbf{s}}$. Thus, the restriction of $V_{a,\mathbf{s}}(\mathcal{M})$ to $\mathrm{Y}_q\big(\mathfrak{gl}_{1|1,\mathbf{s}}\big)$ is an evaluation representation of $\mathrm{Y}_q\big(\mathfrak{gl}_{1|1,\mathbf{s}}\big)$, which is still irreducible. Denote it by $\check{V}_{a,\mathbf{s}}(\mathcal{M})$.  

The comultiplication $\widehat{\triangle}_{\mathbf{s}}$ enables us to regard the tensor products of evaluation representations 
\begin{gather}\label{tensorreps}
    \check{V}_{a_{1},\mathbf{s}}(\mathcal{M}_1)\otimes \check{V}_{a_{2},\mathbf{s}}(\mathcal{M}_2) \cdots \otimes \check{V}_{a_l,\mathbf{s}}(\mathcal{M}_l)
\end{gather}
as a highest weight representation of $\mathrm{Y}_q\big(\mathfrak{gl}_{1|1,\mathbf{s}}\big)$. 
Here, each $a_i\in \mathbb{C}\setminus\{0\}$ and $\mathcal{M}_i=(\mu_{i,1},\ldots,\mu_{i,N})\in \left(\mathbb{C}\setminus\{0\}\right)^N$.  
Clearly, it coincides with the restriction of the tensor product of the evaluation representations over $\mathrm{U}_q\big(\widehat{\mathfrak{gl}}_{1|1,\mathbf{s}}\big)$ given by 
\begin{gather}\label{tensorreps2}
    V_{a_{1},\mathbf{s}}(\mathcal{M}_1)\otimes V_{a_{2},\mathbf{s}}(\mathcal{M}_2) \cdots \otimes V_{a_l,\mathbf{s}}(\mathcal{M}_l). 
\end{gather}

As shown in Section \ref{se:typicalirreducible representationsUq}, every typical evaluation representation $\check{V}_{a,\mathbf{s}}(\mathcal{M})$ of $\mathrm{Y}_q\big(\mathfrak{gl}_{1|1,\mathbf{s}}\big)$ satisfies
\begin{gather*}
    \dim  \check{V}_{a,\mathbf{s}}(\mathcal{M})= 2. 
\end{gather*}
More precisely, 
\begin{gather*}
    \check{V}_{a,\mathbf{s}}(\mathcal{M})=\operatorname{span}_{\mathbb{C}}\big\{\,\zeta,\ \bar{t}_{21}^{(1)}\zeta\,\big\},
\end{gather*}
where $\zeta$ is the maximal vector of $\check{V}_{a,\mathbf{s}}(\mathcal{M})$. Additionally, $\check{V}_{a,\mathbf{s}}(\mathcal{M})$ is one-dimensional if it is atypical. Then we have
\begin{lemma}
    The $\mathrm{Y}_q\big(\mathfrak{gl}_{1|1,\mathbf{s}}\big)$-module \eqref{tensorreps} is irreducible if the last factor $\check{V}_{a_{l},\mathbf{s}}(\mathcal{M}_l)$ is atypical and the tensor product of the remaining factors
    \begin{gather}\label{tensorreps3}
    \check{V}_{a_{1},\mathbf{s}}(\mathcal{M}_1)\otimes \check{V}_{a_{2},\mathbf{s}}(\mathcal{M}_2) \cdots \otimes \check{V}_{a_{l-1},\mathbf{s}}(\mathcal{M}_{l-1})
    \end{gather}
    is irreducible. 
\end{lemma}

Next, we will consider the irreducibility of the tensor product \eqref{tensorreps} when each factor is typical. Let $\check{W}(\nu(u))$ be a irreducible highest weight representation of $\mathrm{Y}_q\big(\mathfrak{gl}_{1|1,\mathbf{s}}\big)$ endowed with highest weight $\nu(u)$
    of order $K$, and let $\xi$ be its maximal vector. We need the following lemma. 

\begin{lemma}\label{lemm:linearcombin}
    If any linear combination of the set of vectors
    \begin{gather*}
        \big\{\,\xi,\ \bar{t}_{21,\mathbf{s}}^{(r_1)}\bar{t}_{21,\mathbf{s}}^{(r_2)}\cdots \bar{t}_{21,\mathbf{s}}^{(r_{p})}\xi\,\big|\,1\leqslant r_1<r_2<\cdots<r_{p}\leqslant K~~~\text{for}~~~p=1,2,\ldots,K\,\big\}
    \end{gather*}
    is trivial, then we have
    \begin{gather*}
        \bar{t}_{21,\mathbf{s}}^{(1)}\bar{t}_{21,\mathbf{s}}^{(2)}\cdots \bar{t}_{21,\mathbf{s}}^{(K)}\xi=0.
    \end{gather*}
\end{lemma}

\begin{proof}
    This lemma can be proved as a similar argument as in \cite[Lemma 2]{Zrb95}. 
    
\end{proof}

Suppose that $\dim \check{W}(\nu(u))=2^K$ and $\nu(u)=(\nu_1(u),\nu_2(u))$ is the highest weights of $\check{W}(\nu(u))$ with 
\begin{gather*}
    \nu_1(u)=\nu_1^{(0)}+\nu_1^{(1)}u+\cdots+\nu_1^{(K)}u^K,\quad \nu_2(u)=\nu_2^{(0)}+\nu_2^{(1)}u+\cdots+\nu_2^{(K)}u^K,
\end{gather*}
for $\nu_1^{(0)}\nu_1^{(K)}=\nu_2^{(0)}\nu_2^{(K)}.$ 
$\check{V}_{a,\mathbf{s}}(\mathcal{M})$ is typical with maximal vector $\zeta$ and highest weights
\begin{gather*}
    (\mu_1-\mu_1^{-1}au,\mu_2-\mu_2^{-1}au).
\end{gather*}
That is, $\mu_1/\mu_2\neq \pm 1$. 

Set $\xi^-=\bar{t}_{21,\mathbf{s}}^{(1)}\bar{t}_{21,\mathbf{s}}^{(2)}\cdots \bar{t}_{21,\mathbf{s}}^{(K)} \xi$, which is the unique vector in $\check{W}(\nu(u))$ (up to a constant factor) that satisfies $\bar{t}_{21,\mathbf{s}}^{(r)}\xi^-=0$ for all $r$. We call $\xi^-$ the \textit{minimal} vector of $\check{W}(\nu(u))$. Define $\zeta^-=\bar{t}_{21,\mathbf{s}}^{(1)}\zeta$ as the minimal vector of $\check{V}_{a,\mathbf{s}}(\mathcal{M})$. 

Through comultiplication $\widehat{\triangle}_{\mathbf{s}}$, we regard the tensor product $\check{W}(\nu(u))\otimes \check{V}_{a,\mathbf{s}}(\mathcal{M})$ as a representation of $\mathrm{Y}_q\big(\mathfrak{gl}_{1|1,\mathbf{s}}\big)$. Then we have 
\begin{align*}
    \widehat{\triangle}_{\mathbf{s}}\left(\bar{t}_{21,\mathbf{s}}(u)\right)\left(\xi\otimes \zeta\right)
    =\bar{t}_{21,\mathbf{s}}(u)\xi\otimes \left(\mu_{1}-\mu_{1}^{-1}au\right)\zeta +\nu_2(u)\xi\otimes u\zeta^-. 
\end{align*}
When taking $u_0=\mu_{1}^2a^{-1}$, one has
\begin{gather*}
    \xi\otimes \zeta^- =\nu_2(u_0)^{-1}u_0 ^{-1}\widehat{\triangle}_{\mathbf{s}}\left(\bar{t}_{21,\mathbf{s}}(u)\right)\left(\xi\otimes \zeta\right)\in \mathrm{N}^-.\left(\xi\otimes \zeta\right),
\end{gather*}
if $\nu_2(u_0)\neq 0$. 
 It follows that 
$$\mathrm{N}^-.(\xi\otimes \zeta^-)=(\mathrm{N}^-.\xi)\otimes \zeta^-,$$
which forces $$\check{W}(\nu(u))\otimes \zeta^-\in \mathrm{N}^-.(\xi\otimes \zeta).$$
Therefore, 
$$\check{W}(\nu(u))\otimes \check{V}_{a,\mathbf{s}}(\mathcal{M})= \mathrm{N}^-.(\xi\otimes \zeta)$$
for $\nu_2(u_0)\neq 0$.

Moreover, 
\begin{align*}
    \widehat{\triangle}_{\mathbf{s}}\left(\bar{t}_{12,\mathbf{s}}(u)\right)\left(\xi^-\otimes \zeta^-\right)&=\bar{t}_{12,\mathbf{s}}(u)\xi^-\otimes \left(\mu_{2}-\mu_{2}^{-1}au\right)\zeta^- +\nu_1(u)\xi^- \otimes (\bar{t}_{12,\mathbf{s}}^{(0)}+u\bar{t}_{12,\mathbf{s}}^{(1)})\zeta^-,
\end{align*}
where 
$$(\bar{t}_{12,\mathbf{s}}^{(0)}+u\bar{t}_{12,\mathbf{s}}^{(1)})\zeta^-=a(\mu_1\mu_2^{-1}-\mu_1^{-1}\mu_2)\zeta\neq 0,$$
since $\check{V}_{a,\mathbf{s}}(\mathcal{M})$ is typical. 
When taking $u_0=\mu_{2}^2a^{-1}$, one has 
$$\xi^-\otimes\zeta=\nu_1(u_0)^{-1}a^{-1}\left(\frac{\mu_1}{\mu_2}-\frac{\mu_2}{\mu_1}\right)^{-1}\widehat{\triangle}_{\mathbf{s}}\left(\bar{t}_{12,\mathbf{s}}(u)\right)\left(\xi^-\otimes \zeta^-\right)\in \mathrm{N}^+.\left(\xi^-\otimes \zeta^-\right),$$
if $\nu_1(u_0)\neq 0$. 
Similar to the argument in the previous paragraph, we have
$$\check{W}(\nu(u))\otimes \check{V}_{a,\mathbf{s}}(\mathcal{M})= \mathrm{N}^+.(\xi^-\otimes \zeta^-)$$
for $\nu_1(u_0)\neq 0$.

Now we argue that \eqref{tensorreps} is irreducible by induction on $K$. Let $\zeta_i$ be the maximal vector of $\check{V}_{a_{i},\mathbf{s}}(\mathcal{M}_i)$. Define $\zeta_i^-=\bar{t}_{21,\mathbf{s}}^{(1)}\zeta_i$ for each $i$.  For $K=2$, if
$$\frac{a_1}{a_2}\neq \frac{\mu_{1,2}^2}{\mu_{2,1}^2}~~~\text{and}~~~\frac{\mu_{1,1}^2}{\mu_{2,2}^2},$$
we obtain
\begin{align*}
    \check{V}_{a_1,\mathbf{s}}(\mathcal{M}_1)\otimes \check{V}_{a_2,\mathbf{s}}(\mathcal{M}_2)= \mathrm{N}^-.(\zeta_1\otimes \zeta_2)=\mathrm{N}^+.(\zeta_1^-\otimes \zeta_2^-)
\end{align*}
as before. If $\check{V}_{a_1,\mathbf{s}}(\mathcal{M}_1)\otimes \check{V}_{a_2,\mathbf{s}}(\mathcal{M}_2)$ is not irreducible, it has a proper submodule generated by the maximal vector $\zeta_1\otimes \zeta_2$ or the minimal vector $\zeta_1^-\otimes \zeta_2^-$ in terms of Lemma \ref{lemm:linearcombin}; this is impossible.  

Assume that $\check{W}(\nu(u))$ is isomorphic to \eqref{tensorreps3} such that every factor is typical. Thus, $\check{W}(\nu(u))\otimes \check{V}_{a_{l},\mathbf{s}}(\mathcal{M}_l)$ is irreducible if 
\begin{gather}\label{aij-conditions}
    \frac{a_i}{a_j}\neq \frac{\mu_{i,2}^2}{\mu_{j,1}^2}~~~\text{and}~~~\frac{\mu_{i,1}^2}{\mu_{j,2}^2}\qquad\text{for each pair}~~(i,j).
\end{gather}

We note that the $\mathrm{U}_q\big(\widehat{\mathfrak{gl}}_{1|1,\mathbf{s}}\big)$-module \eqref{tensorreps2} is irreducible if its restriction to the $q$-super Yangian $\mathrm{Y}_q\big(\mathfrak{gl}_{1|1,\mathbf{s}}\big)$ is irreducible. To summarize the above arguments, we conclude that
\begin{theorem}
    The $\mathrm{U}_q\big(\widehat{\mathfrak{gl}}_{1|1,\mathbf{s}}\big)$-module \eqref{tensorreps2} is irreducible if condition \eqref{aij-conditions} holds. Moreover, every finite-dimensional irreducible representation of $\mathrm{U}_q\big(\widehat{\mathfrak{gl}}_{1|1,\mathbf{s}}\big)$ is isomorphic to a tensor product of typical evaluation representations with the form \eqref{tensorreps2} satisfying the condition \eqref{aij-conditions}. 
\end{theorem}
We also have
\begin{corollary} 
   If the irreducible highest weight representation $V_{\mathbf{s}}(\lambda(u);\bar{\lambda}(u))$ satisfyies \eqref{fd:cd-af:11} for $\deg Q(u)=\deg \widetilde{Q}(u)=K$ and $Q_0Q_K=\widetilde{Q}_0\widetilde{Q}_K=1$, then
    the set of vectors
    $$\Big\{\,\zeta,\,\bar{t}_{21,\mathbf{s}}^{(k_1)}\cdots \bar{t}_{21,\mathbf{s}}^{(k_l)}\zeta  \,\Big|\,1\leqslant k_1<\cdots<k_l\leqslant K\,\Big\}$$
    forms a basis for $V_{\mathbf{s}}(\lambda(u);\bar{\lambda}(u))$. Moreover, $\dim V_{\mathbf{s}}(\lambda(u);\bar{\lambda}(u))=2^K$. 
\end{corollary}

\vspace{1em}

\section*{Acknowledgments}
H. Lin is supported by the Postdoctoral Fellowship Program of CPSF (GZC20252014). H. Zhang is  supported by the support of the National Natural Science Foundation of China (No. 12271332).

\end{document}